\documentclass[a4paper, 10pt]{article}
\usepackage{amssymb,amsmath,amsthm,url,mathrsfs,graphicx}
\usepackage[letterpaper,hmargin=1.0in,vmargin=1.0in]{geometry}

\usepackage{color}

\usepackage{xr-hyper}

\usepackage{xspace}
\usepackage{calc}
\usepackage{euscript}
\usepackage{amsfonts}
\usepackage{enumerate}

\usepackage[numbers]{natbib}

\usepackage[textsize=tiny]{todonotes}

\newcommand\Item[1][]{%
  \ifx\relax#1\relax  \item \else \item[#1] \fi
  \abovedisplayskip=0pt\abovedisplayshortskip=0pt~\vspace*{-\baselineskip}}

\newtheorem{maintheorem}{Theorem}[section]

\newtheorem{lemma}[maintheorem]{Lemma}
\newtheorem{proposition}[maintheorem]{Proposition}
\newtheorem{corollary}[maintheorem]{Corollary}

\newtheorem{definition}[maintheorem]{Definition}
\newtheorem{remark}[maintheorem]{Remark}
\newtheorem{conjecture}[maintheorem]{Conjecture}

\newtheorem{question}[maintheorem]{Question}

\def\P{\mathbb{P}}

\newcommand{\indic}[1]{\mathbf{1}_{\left\{#1\right\}}}
\newcommand{\bra}[1]{\left[#1\right]}
\DeclareMathOperator*{\argmax}{argmax}

\newcommand{\Ind}[1]{\mathbf{1}_{  \{#1\} }}

\newcommand{\pd}{\partial}

\newcommand{\Var}{{\mathrm{Var}}}
\newcommand{\Fill}{{\mathrm{Fill}}}
\newcommand{\ink}{{\mathrm{ink}}}
\newcommand{\hink}{ \widehat{{\mathrm{ink}}}}

\newcommand{\IP}{{\mathrm{IP}}}
\newcommand{\RW}{{\mathrm{RW}}}
\newcommand{ \mixex}{ t_{\mathrm{mix}}^{\mathrm{EX}(k)} }
\newcommand{ \mixrw}{ t_{\mathrm{mix}}^{\mathrm{RW}(1)} }
\newcommand{ \mixrwk}{ t_{\mathrm{mix}}^{\mathrm{RW}(k)} }
\newcommand{ \mixIP}{ t_{\mathrm{mix}}^{\mathrm{IP}(k)} }

\newcommand{\E}{{\mathbb{E}}}
\newcommand{\hE}{{\mathbb{\widehat{E}}}}
\newcommand{\hP}{{\mathrm{\widehat{P}}}}

\newcommand{\PP}{\mathscr{P}}

\renewcommand{\Pr}{ \mathrm P}

\newcommand{ \rel}{ t_{\mathrm{rel}} }
\newcommand{ \mix}{ t_{\mathrm{mix}} }

\newcommand{ \TV}{ \mathrm{TV} }

\newcommand{\eps}{\varepsilon}
\newcommand{ \cL}{ \mathcal L }

\newcommand{\la}{\lambda}

\DeclareMathSymbol{\leqslant}{\mathalpha}{AMSa}{"36} 
\DeclareMathSymbol{\geqslant}{\mathalpha}{AMSa}{"3E} 
\DeclareMathSymbol{\eset}{\mathalpha}{AMSb}{"3F}     
\renewcommand{\le}{\;\leqslant\;}                   
\renewcommand{\ge}{\;\geqslant\;}                   

\newcommand{\EE}{\mathcal{E}}

\newcommand{\sfrac}[2]{\mbox{\small $\frac{#1}{#2}$}}
\newcommand{\ssfrac}[2]{\mbox{\footnotesize $\frac{#1}{#2}$}}
\newcommand{\half}{\ssfrac{1}{2}}

\newcommand{\N}{\mathbb N}
\newcommand{\R}{\mathbb R}
\newcommand{\Z}{\mathbb Z}

\newcommand{\WW}{\mathrm{W}}
\newcommand{\RR}{\mathrm{R}}
\newcommand{\BB}{\mathrm{B}}
\newcommand{\KK}{\mathrm{K}}

\newcommand{\oo}{\mathbf{o}}

\newcommand{\whp}{\mathrm{w.h.p.}}

\usepackage[normalem]{ulem}
\RequirePackage[colorlinks,citecolor=blue,urlcolor=blue]{hyperref}

\begin{document}
\title{The exclusion process mixes (almost) faster than
independent particles}
\author{Jonathan Hermon
\thanks{
The University of British Columbia, Department of Mathematics,  1984 Mathematics Road, Vancouver, BC V6T 1Z2, Canada.    
E-mail: {\tt jhermon@math.ubc.ca.} Financial support by
the EPSRC grant EP/L018896/1.}
\and Richard Pymar
\thanks{Department of Economics, Mathematics and Statistics, Birkbeck, University of London, London, WC1E 7HX, UK. E-mail: {\tt r.pymar@bbk.ac.uk}} 
}
\date{}
\maketitle



\begin{abstract}
Oliveira conjectured that the order of the mixing time of the exclusion process with $k$-particles on an arbitrary $n$-vertex graph is at most that of the mixing-time of $k$ independent particles. We verify this up to a constant factor for $d$-regular graphs when each edge rings at rate $1/d$ in various cases: \\\indent(1) when  $d = \Omega( \log_{n/k} n)$, \\\indent(2) when $\mathrm{gap}:=$ the spectral-gap of a single walk is $ O ( 1/\log^4 n)  $  and  $k  \ge n^{\Omega(1)}$,  \\\indent(3) when $k  \asymp n^{a}$ for some constant $0<a<1$. \\In these cases our analysis yields a probabilistic proof of a weaker version of Aldous' famous spectral-gap conjecture (resolved by Caputo et al.). We also prove a general bound of $O(\log n \log \log n / \mathrm{gap})$, which is within a $\log \log n$ factor from Oliveira's conjecture when $k \ge n^{\Omega (1)}$. As applications we get new mixing bounds: \\\indent(a) $O(\log n \log \log n)$ for expanders, \\\indent(b) order $ d\log (dk) $ for the hypercube $\{0,1\}^d$, \\\indent(c) order $(\mathrm{Diameter})^2 \log k $ for vertex-transitive graphs of moderate growth and for supercritical per-\\\indent\phantom{(c)} colation on a fixed dimensional torus. 
\end{abstract}
\vspace{3em}
\paragraph*{\bf Keywords:}
{\small Exclusion process, mixing-time, chameleon process, particle system.\\
{\normalsize \bf  AMS 2010 Subject Classification:}
  {\small
    Primary: 60J27, 60K35;
    secondary: 82C22}}\normalsize

\newpage
\tableofcontents
\newpage
\section{Introduction}The \emph{symmetric exclusion process}  $\mathrm{EX}(k) $ on a finite, connected graph $G= (V, E)$ (with vertex set $V$ and edge set $E$) is the following continuous-time Markov process. In a configuration, each vertex is occupied by either
a black particle or a white particle (where particles of the same colour are indistinguishable),
such that the total number of black particles is  $k < |V |=:n$.  For
each edge $e$ independently, at the times of a Poisson process of rate $r_e>0$, switch the particles at the endpoints of $e$. In this work we take $G$ to be $d$-regular and set $r_e\equiv 1/d $. The \emph{interchange process} $\IP(k)$ is similarly defined, apart from the fact that we label the black particles by the set $[k]:=\{1,\ldots,k\}$, so that they become distinguishable.

The exclusion process is among the most fundamental and well-studied processes in the literature on interacting
particle systems \cite{liggettbook2,liggettbook1}, with ties to card shuffling \cite{lacoincycle,lacoin,Wilson}, statistical mechanics \cite{Martinelli,KOV,Q} and numerous other processes (see, e.g., \cite[Ch.\ 23]{levin} and \cite{liggettbook2}). Apart from having a rich literature on the model on infinite
graphs, such as the lattices $\Z^d$, the exclusion process on finite graphs  has  been one of the major examples
driving quantitative
study of finite Markov chains. Couplings and random walks collision \cite{aldous,olive}, comparison techniques \cite{comparison} (see the discussion in \cite[Appendix A]{olive}) log-Sobolev inequalities \cite{diaconis,LY,Yau}, path coupling \cite{greenberg,levinex,levin,Wilson} and variants of the evolving
sets method \cite{CP,evolving,morris,olive} have been  applied to this process. Sharp
results have been obtained for certain graphs including the complete graph \cite{lacoincomplete,LY}, the discrete tori $(\Z/L\Z)^d$ \cite{morris}, the path \cite{lacoin} (including the asymmetric case \cite{asymetric2,asymetric}), the cycle \cite{lacoincycle}, and a variety of random graphs \cite{olive}. Bounds on the mixing time of the related interchange process have also been obtained for various graphs \cite{jonasson}.

For a continuous-time Markov process $Q$ we denote by $\mix^{Q}(\eps)$ the total-variation $\eps$-mixing time of $Q$ (see e.g.\! \eqref{e:tvmixdef}). When $\eps=1/4$ we omit it from this
notation.
Oliveira \cite{olive} showed that for some absolute constant $C$,  for general graphs and  rates,
\begin{equation}
\label{e:Olivemain}
\forall \eps \in (0,1), \quad   \max_{k} \mixex (\eps) \le C \mix^{\RW(1)} \log (n/\eps),
\end{equation}
where $\RW(r)$ is the process of $r \in \{1,\ldots,n\} $ independent continuous-time random walks on $G$, each having the same transition rates $(r_e:e \in E)$.
It was left as an open problem to determine whether the following stronger relation holds
\begin{equation}
\label{e:Olivecon}
\forall \eps \in (0,1),\, \quad \mixex (\eps) \le C \mix^{\RW(k)} (\eps).
\end{equation} A heuristic reasoning for this conjecture is the fact that the exclusion process satisfies a strong negative dependency property called negative association \cite{NA}, which in some sense is even stronger than independence (see \S\ref{s:NA}). One of the motivations given in  \cite{olive} for \eqref{e:Olivemain} is that it serves as a proxy for \eqref{e:Olivecon}, on which it is commented that ``if at all true, is well beyond the reach of present techniques". Part of the appeal of \eqref{e:Olivecon} is its connection to Aldous' spectral-gap conjecture, now resolved by Caputo, Liggett and Richthammer \cite{Caputo}, which asserts that the spectral-gaps of processes $\mathrm{EX}(k),\IP(r),\RW(1)$ are the same for all $r \in [n]$ and $k \in [n-1] $. A further discussion of connections to this conjecture can be found in \S\ref{s:con}.

In this work we consider the mixing time of EX($k$) for general finite $d$-regular graphs  with rates $r_{e} \equiv \sfrac{1}{d} $ and obtain bounds in terms of the spectral-profile and relaxation-time. We obtain a general upper bound which is within a $\log\log n$ factor of Oliveira's conjecture when $k= n^{\Omega(1)}$ and prove the conjecture in certain special cases for all $k$ (which includes hypergraphs). Finally, we give lower bounds on the mixing time of EX$(k)$ in terms of independent random walks.

Note that $\mathrm{EX}(k)$ is in one-to-one correspondence with  $\mathrm{EX}(n-k)$, as we may consider the set of vacant (white) vertices instead of the occupied (black) ones. Hence we may assume throughout that $k \le n/2 $. 

\subsection{Our main general results}\label{ss:main}
 We present various bounds on $\mixex$ and show how they relate to verifying \eqref{e:Olivecon} in general, and for specific graphs.
 The first result we present bounds $\mixex$ in terms of $\rel:=\sfrac{1}{\mathrm{gap}}$ (the \emph{relaxation-time}, which  is the inverse of the  \emph{spectral-gap}, the smallest positive eigenvalue of $- \cL $ for $\cL$  the generator of  RW(1)) and a quantity related to the decay of the heat-kernel of a random walk, denoted $P_t$. Specifically, for each $\epsilon\in(0,1)$, let (recall $n:=|V|$)
 \begin{equation}\label{eq:rast}
 r_\ast(\epsilon):=\inf\{t:\max_{v\in V} P_t(v,v)-1/n\le \sfrac{\epsilon}{(\log n)^2}\}.
 \end{equation}
\begin{maintheorem}[General mixing bound]
\label{thm:main1}
There exist universal constants $C_{1.1},\,c_{1.1}>0$ such that for every $n$-vertex  $d$-regular graph $G$ with rates $r_{e} \equiv \sfrac{1}{d} $ we have that 
\begin{equation}
\label{e:main1}
\forall \eps\in(0,1),\quad \max_k \mixex (\eps) \le C_{1.1}( \rel + r_\ast(c_{1.1})) \log (n/\eps) .
\end{equation} 
\textcolor{black}{In particular, $\max_k \mixex (\eps) \lesssim  \mix^{\RW(\lceil \sqrt{n} \rceil)} (\eps)+ r_\ast(c_{1.1}) \log (n/\eps)  $.
}\end{maintheorem}

For expanders $\rel \asymp 1 $, while it follows from the spectral decomposition that  $r_\ast(\epsilon) \asymp_\epsilon \log \log n $.
Hence we obtain the bound $\max_k\mixex\lesssim\log n\log\log n$ for expanders (Oliveira's conjecture gives $\log n$). In fact, this is the only natural example we have where $r_*(\epsilon) \gg \rel $. In general (for $n$-vertex regular graphs) it can be shown (see \eqref{e:ts_*2} in \S\ref{S:prelim}) that \begin{align}\label{e:tastsast} r_\ast(\epsilon)\lesssim_\epsilon(\log n)^4\wedge\rel\log\log n\end{align}  from which we verify \eqref{e:Olivecon} if $\rel=\Omega((\log n)^4)$ and $k=n^{\Omega(1)}$. Moreover we establish that \eqref{e:Olivecon} holds in general up to a $\log\log n$ factor for $k=n^{\Omega(1)}$.

 Next, we bound $\mixex (\eps)$ in terms of $t_{\mathrm{sp}}(\eps) $ the bound on the $\eps$ $L_\infty$-mixing time obtained via the spectral profile -- see \eqref{e:profiles} for a definition and  $\rel$.
\begin{maintheorem}[Mixing for sublinear number of particles]
\label{thm:main3}
For each $\delta\in(0,1)$ there exist universal constants $C_{1.2}(\delta),$ $C'_{1.2}(\delta)>0$ such that for every $n$-vertex  $d$-regular graph $G$ with rates $r_{e} \equiv \sfrac{1}{d} $ and all $k \le n^\delta$ we have that
\begin{equation}
\label{e:main4}
\begin{split}
\forall \eps \in (0,1), \quad  \mixex (\eps) &\le C_{1.2}(\delta)t_{\mathrm{sp}}(\sfrac{\eps}{k})  \le C'_{1.2}(\delta)[t_{\mathrm{sp}}(\sfrac{1}{2})+\rel \log (k/ \eps) ]. 
\end{split} 
\end{equation}In particular if $k\in[n^\beta,n^{1-\beta}]$ for some $\beta\in(0,1/2)$ then 
\begin{align}\label{e:main5}
\forall \eps \in (0,1), \quad \mixex(\eps) \lesssim \mixrwk(\eps).
\end{align}
\end{maintheorem}
This result is one of our principal improvements to the main result of Oliveira \cite{olive}  as it gives refined bounds for the case $k=n^{o(1)}$. 
By applying this theorem we verify \eqref{e:Olivecon} for all $k$ under the condition $t_{\mathrm{sp}}(\sfrac1{2})\lesssim \rel$, see Corollary~\ref{cor:example}. This condition holds for vertex-transitive graphs of moderate growth and for supercritical percolation on a fixed dimensional torus $(\mathbb{Z}/L\mathbb{Z})^d$ (see \S\ref{s:mod}). In these cases we obtain $\mixex \asymp ( \mathrm{diam}(G))^2 \log k $ uniformly in $k \le n/2 $. Morris \cite{morris} obtained the same bound for $G=(\Z/L\Z)^d$ and  Oliveira proved the same bound on the giant component of supercritical percolation on $(\Z/L\Z)^d$ for $k=n^{\Omega(1)} $. 
 
We now explain how \eqref{e:main5} follows from \eqref{e:main4}. If $n\ge k=n^{\Omega (1)} $ then from the definition of the spectral profile we have   $t_{\mathrm{sp}}(\sfrac{\eps}{k})  \asymp  \rel  \log (n/\eps)  $ for each $\eps\in(0,1)$ (we remark that the upper bound here holds for all $k\le n$). Further, it can be shown (see \eqref{e:mixkrel} in \S\ref{S:prelim}) that for such $k$ and $\eps$, $\mixrwk(\eps)\asymp \rel\log(n/\eps)$, and so we verify \eqref{e:Olivecon} for $k\asymp n^\delta$ with $\delta\in(0,1)$. For expanders we obtain $\mixex\lesssim_\delta \log n$ for $k\le n^{\delta}$.

%
%
%
In the seminal work \cite{Wilson} where he invented the so-called Wilson method, Wilson proved that for the hypercube $\{ \pm 1 \}^d$ one has that $\mix^{\mathrm{EX}(2^{d-1})} \gtrsim  d^2$  \cite[p.\ 308]{Wilson}. He conjectured that  $ \mix^{\mathrm{EX}(2^{d-1})} \asymp d^2$ (to be precise, one may interpret the last sentence in \cite[\S9.1]{Wilson} as saying that  $\mix^{\mathrm{IP}(2^{d})} \lesssim d^2  $\textcolor{black}{, which was  verified by the first named author and Salez \cite{IPhypercube} after this paper appeared online}). Using Theorem~\ref{thm:main3} we show that for the hypercube we have $\mixex\lesssim d\log(d k)$ uniformly in $k\le 2^{d-1}$, see \S\ref{s:hypercube} \textcolor{black}{(in fact, we treat general product graphs)}. We also obtain a lower bound of the same order. To the best of our knowledge, previously the best available upper bound for the hypercube was  $\max_k \mixex \lesssim d^2\log d $ and for expanders was $\max_k \mixex \lesssim (\log n)^2 $, both due to Oliveira \cite{olive} (see \eqref{e:Olivemain}).

The last main bound on $\mixex$ is again just in terms of $\rel$, but requires the degree to be growing sufficiently fast.

\begin{maintheorem}[Mixing for graphs of high degree]\label{thm:maindeg}

There exist universal constants $C_{\mathrm{deg}},C_{1.3}>0$ such that for every $n$-vertex $d$-regular graph $G$ with rates $r_e\equiv\sfrac1{d}$ if $d \ge C_{\mathrm{deg}} \log_{n/k} n$ then
\begin{equation}
\label{e:main3}
\forall \eps\in(0,1),\quad  \mixex (\eps) \le C_{1.3} \rel  \log (n/\eps).  
\end{equation}

\end{maintheorem}
This theorem verifies \eqref{e:Olivecon} for $k=n^{\Omega(1)}$ when $d=\Omega(\log_{n/k} n)$.
\subsection{Lower bounds}
We provide now a general lower bound on $\mixex$ in terms of $\rel$. We remark that there are few known general lower bounds on $\mixex$ in the existing literature.

Theorem \ref{thm:lower} shows that under a mild delocalization assumption regarding some eigenvector corresponding to the spectral-gap, one has that $\mixex  \gtrsim   \mix^{\RW(k)} $ when $k = n^{\Omega (1)} $. Proposition \ref{p:LS} provides a general condition ensuring that such delocalization holds. Moreover, Corollary \ref{cor:example} provides a sufficient condition for  $\mixex  \asymp  \mix^{\RW(k)} $ for all $k$.

To motivate our result, consider an $n$-vertex regular expander and attach a path of length $L:=\lceil \log n \rceil $ to one of its vertices. We expect that in this case $\max_k \mixex \lesssim \rel \log L$, and so $\mixex  \ll \rel \log k $ for $k=(\log n)^{\omega(1)}$. This demonstrates that in general we cannot expect $\mixex \gtrsim  \rel \log k $. We now give a sufficient condition for this to hold. Here we make no assumptions on $G$ nor on the rates $\mathbf{r}:=(r_e:e \in E)$. Recall that  $\mathcal{L} $ denotes the generator of RW$(1)$ and let $\pi:=\mathrm{Unif}(V)$ be the stationary distribution. For $ f,g\in \mathbb{R}^V$ define   $\|f\|_p^p:=\mathbb{E}_{\pi}[|f|^p]=\sum_x \pi(x)|f(x)|^p$ for $p \in (0,\infty)  $ and $\|f\|_{\infty}:=\max_{v \in V}|f(v)|$.
     
\begin{maintheorem}
\label{thm:lower} Let $\la>0$ be an eigenvalue of $-\mathcal{L} $ and    $f \neq 0 $ a corresponding eigenfunction. If $\eps, \delta \in (0,1/4)$ and $k \le n/2$ are such that $\|f\|_1 \ge k^{-1/4+\delta}\|f\|_2$ and $4\delta\log k - \log (16/\eps )  \ge 0$ then \[\mixex (1-\eps)  \ge \sfrac{1}{2\la}(4\delta\log k - \log (16/\eps ) ). \]
\end{maintheorem}
Note that in order to apply Theorem \ref{thm:lower}  it suffices to find one  eigenfunction\ $f$ satisfying $\sfrac{\|f\|_1}{\|f\|_2} \ge k^{-\sfrac{1}{5}}$. Denote the eigenvalues of $-\mathcal{L} $ by $0=\la_1 < \la_2 \le \cdots \le \la_n $. In practice, when applying Theorem \ref{thm:lower} one should pick $\la=\la_2$. Observe that $\|f\|_2 \le \sqrt{n} \|f\|_1 $ for all $f$ (not necessarily an eigenfunction).  

 Proposition \ref{p:LS} below provides a general upper bound on $\|f\|_2/\|f\|_1 $ for an eigenfunction $f$ corresponding to an eigenvalue $\la>0$ of $-\mathcal{L}$ in terms of $\la / c_{\mathrm{LS}}$, where  $c_{\mathrm{LS}}=c_{\mathrm{LS}}^{\mathrm{RW}(1)} $  is the log-Sobolev constant of the graph (defined in \eqref{e:logsob} of \S\ref{s:profiles}).

%
%

\begin{proposition}
\label{p:LS}
For (non-zero) $f \in \R^V $ such that $\mathcal{L} f=- \la f$ we have 
\begin{equation}
\label{e:distortion}
\log ( \|f\|_2/2\|f\|_1) \le  \la/c_{\mathrm{LS}}.
\end{equation}
\end{proposition}
It is natural to expect that $\mixex $ is at least ``weakly" monotone in $k$ for $k \le n/2$. While this is immediate for $\mixIP $, we do not know how to show this for the exclusion process.
\begin{conjecture}[Weak monotonicity of the mixing time in the number of particles]
There exists an absolute constant $C>0$ such that if $k_1 \le k_2 \le n/2 $ then $\mix^{\mathrm{EX}(k_1)} \le C\mix^{\mathrm{EX}(k_2)}$.
\end{conjecture}
Embarrassingly, we can resolve only the case when $k_1=1$.

\begin{proposition}
\label{p:mixexkatleastmixrw1}
There exists an absolute constant $c>0$ such that $\min_{k \in [n-1] }\mix^{\mathrm{EX}(k)} \ge c \mix^{\RW(1)}$.
\end{proposition}
We remark that in Proposition \ref{p:mixexkatleastmixrw1} we make no assumption on $G$ nor on the rates.
\subsection{On the interchange process}
As is the case in \cite{olive}, our arguments can be used to upper-bound $\IP(k) $  as long as $k \le (1-\theta)n $ for some constant $\theta \in (0,1)$ (in this case constants $C_{1.1},C_{1.2}$ and $C_{1.4}$ will depend on $\theta$).

In a recent work \cite{Alon} Alon and Kozma  showed that in the regular case with $r_e=1/d$  the $L_{\infty}$-mixing-time of $\mathrm{IP}(n)$ is $ \lesssim \mix^{\RW(1)} \log n $ \textcolor{black}{(their result is more general, but contains an additional multiplicative term, which need not be of order 1 for general rates or when the graph is not regular)}. This is obtained via a comparison argument which hinges on an elegant use of the octopus inequality of Caputo, Liggett and Richthammer \cite{Caputo}. 

\subsection{Extensions and further applications}\label{s:introrelax}
We present a couple of ways in which some of our assumptions can be relaxed; for further details see Appendix~\ref{s:relax}. \begin{itemize}
\item The assumption of regularity can be replaced with an assumption on neighbouring vertices having comparable degrees. In this case, the results of Theorems~\ref{thm:main1}--\ref{thm:maindeg} still hold subject to a few modifications. \item The requirement $d \ge C_{\mathrm{deg}} \log_{n/k} n $ in \eqref{e:main3} can be replaced (under some additional conditions) with the assumption that the $\ell$th neighbourhood of each vertex is at least of size $C_\mathrm{deg}\log_{n/k}n$ for some fixed~$\ell$. \end{itemize}

\textcolor{black}{A sequence of Markov chains is said to exhibit \emph{precutoff} if for some $\delta_n=o(1)$ the $1-\delta_n$ and the $\delta_n$ mixing times of the $n$th chain in the sequence are comparable, i.e.\ $\mix^{(n)}(\delta_n) \asymp \mix^{(n)}(1-\delta_n)$. Our results imply precutoff in various circumstances.}
 
The following corollaries summarize various scenarios in which the bounds of Theorems \ref{thm:main1}-\ref{thm:lower} and Proposition \ref{p:mixexkatleastmixrw1}  take particularly simple forms \textcolor{black}{and precutoff occurs}.  In each of the four following statements we let $G_{m}=(V_m,E_m)$ be a sequence of finite $d_m$-regular graphs of increasing sizes $n_m$ with rates $r_{e}^{(m)} \equiv \sfrac{1}{d_m} $. We emphasize the identity of the graph we are considering by adding it as a superscript or in parentheses.
\begin{corollary}[{Proof in Appendix~\ref{appb1}}]
\label{cor:example}  If $\rel(G_m) \asymp t^{G_m}_{\mathrm{sp}}(\sfrac{1}{2})  $ then uniformly in $k_m \le n_m/2 $ we have  
\begin{equation}
\label{e:mixproptotsp}
 \mix^{\mathrm{EX}(k_m),G_m} \asymp \rel(G_m) \log (k_m +1)\asymp \mix^{\mathrm{RW}(k_m),G_m}.   \end{equation}
\textcolor{black}{Moreover,  the sequence $\mathrm{(EX}(k_m),G_m)$ exhibits a precutoff,  provided $k_m \gg 1 $.}  
 \end{corollary}
 \begin{corollary}[{Proof in Appendix~\ref{appb1}}]
 \label{cor:1.9}
If  $\mix^{\RW(1),G_m} \asymp t^{G_{m}}_{\mathrm{sp}}(\sfrac{1}{2})  $, then for all fixed $\delta \in (0,1)$, uniformly in $k_m \le n_m^{1-\delta} $ we have
\begin{equation}
\label{e:mixproptotsp2}
 \mix^{\mathrm{EX}(k_m),G_m} \asymp_\delta \mix^{\RW(1),G_m} + \rel(G_m) \log (k_m +1)\asymp  \mix^{\mathrm{RW}(k_m),G_m}.   \end{equation}\end{corollary} \textcolor{black}{Moreover,  the sequence $\mathrm{(EX}(k_m),G_m)$ exhibits a precutoff, provided that $ \rel(G_m) \log (k_m +1)\gg \mix^{\RW(1),G_m}  $ and  $1 \ll k_m \le n_m^{1-\delta} $ for some $\delta \in (0,1)$.} 

 \begin{corollary}[{Proof in Appendix~\ref{appb1}}]
There exist constants $c,c'>0$ such that for all $\delta_{m} \in (0,1/5) $ if   $ \sfrac{1}{c_{\mathrm{LS}}(G_m)}  \le c \delta_{m}   \rel(G_m)  \log n_m $ for all $m$ then  \begin{align} \label{e:cor1.10}\mix^{\mathrm{EX}(k_m),G_m}  \ge c'(\mix^{\RW(1),G_m}  \vee \delta_{m} \rel(G_m) \log n_m  )\end{align} for all $m$ and all   $ k_m \in [ n_m^{5\delta_{m}},\sfrac{n_m}{2}]$. 
\end{corollary}

\begin{corollary}[{Proof in Appendix~\ref{appb1}}]
 \label{cor:1.11}
\textcolor{black}{Let $\delta \in (0,1)$. If $\sfrac{1}{c_{\mathrm{LS}}(G_m)} \lesssim  \frac{  \rel(G_m) \log n_m }{\log \log n_m} $  then   \begin{align}\label{e:13} \mix^{\mathrm{EX}(k_m),G_m} \asymp  \rel(G_m) \log( k_m+1)  \asymp \mix^{\mathrm{RW}(k_m),G_m},  \end{align} 
 uniformly for $k_m \in [n_m^{\delta}, \sfrac{n_m}{2}]$, and the sequence $\mathrm{(EX}(k_m),G_m)$ exhibits a precutoff provided  $n_m^{\delta} \le k_m \le n_m/2 $.} 

If $\sfrac{1}{c_{\mathrm{LS}}(G_m)} \lesssim    \rel(G_m)  $  then  uniformly in $k_m \le \sfrac{n_m}{2}  $ we have that \begin{align}\label{e:14} \rel(G_m) \log (k_m +1)\lesssim \mix^{\mathrm{EX}(k_m),G_m} \lesssim \rel(G_m) \log (k_m \vee \log n_m).    \end{align}
\textcolor{black}{Moreover, if   $ k_m \gtrsim \log n_m $ (and $k_m \leq n_m/2$) the sequence $\mathrm{(EX}(k_m),G_m)$ exhibits a precutoff and  $ \mix^{\mathrm{EX}(k_m),G_m} \asymp \mix^{\mathrm{RW}(k_m),G_m}  $.}
\end{corollary}

\subsection{Aldous' spectral-gap conjecture}\label{s:con}
In the spirit of Aldous' spectral-gap conjecture, now resolved by Caputo, Liggett and Richthammer \cite{Caputo}, which asserts that the spectral-gaps of processes $\mathrm{EX}(k),\IP(r),\RW(1)$ are the same for all $r \in [n]$ and $k \in [n-1] $, one may conjecture the  stronger relation
\begin{equation}
\label{e:strongconjecture}\begin{split}
&\forall \mathbf{x} \in (V)_k,\, t \ge 0, \\&  \|\Pr_{\mathbf{x}}^{\IP(k)}(\mathbf{x}(t) \in \bullet)-\pi_{\IP(k)}  \|_{\TV} \le \|\Pr_{\mathbf{x}}^{\RW(k)}(\mathbf{x}(t) \in \bullet)-\pi_{\RW(k)}  \|_{\TV}.  
\end{split}
\end{equation}  
Observe that a positive answer to \eqref{e:strongconjecture} will provide another proof to Aldous' conjecture. Indeed, \eqref{e:strongconjecture} yields $\rel^{\IP(k)} \le \rel^{\RW(k)}=\rel^{\RW(1)}$, which can be deduced from \eqref{e:mixrel}. Conversely, the inequalities $\rel^{\IP(k)} \ge \rel^{\mathrm{EX}(k)} \vee \rel^{\RW(1)} $ for all $k \in [n] $ (where we define $\rel^{\mathrm{EX}(n)}=0$) and $\rel^{\mathrm{EX}(k)} \ge \rel^{\RW(1)}$ for all $k \in [n-1]$  are the easier direction of Aldous' conjecture (see  \cite{Caputo}). Similarly, our Theorems \ref{thm:main1}-\ref{thm:maindeg} show that for regular graphs $\max_k \rel^{\mathrm{EX}(k)} \lesssim  \rel+r_\ast  $ (recall that often $r_* \lesssim \rel$), while if $d \ge C_{\deg} \log_{k/n} n$ then  $\rel^{\mathrm{EX}(k)} \lesssim  \rel  $, and (for all $d$)   $\max_{k \le n^{\delta} } \rel^{\mathrm{EX}(k)} \lesssim_{\delta}  \rel  $. While this is of course weaker than the result of Caputo et al., what is interesting here is that our proof is entirely probabilistic. 

\textcolor{black}{It is plausible that Aldous' conjecture could be strengthened to an operator $L_2$ inequality, between the generator of the interchange process and that of the corresponding mean-field system (see \cite[Conjecture 1]{IPhypercube}). This would have striking consequences, including verifying Oliveira's conjecture (even for the $L_2$ mixing--time). See \cite{Alon} for an application for the emergence of macroscopic cycles in the cycle decomposition of the permutation obtained by running the interchange process. We note that such an operator inequality was recently proved for the zero range process in \cite{ZRP}.
}
\begin{question}
Is it the case that there exists an absolute constant $C>1$ and some non-decreasing continuous $f:[0,1] \to [0,1]$ with $f(0)=0$ such that for all $t \ge 0$
\begin{equation*}
\label{e:revereseconjecture}
\begin{split}
 \forall \mathbf{x} \in (V)_k,\,  \quad  &\|\Pr_{\mathbf{x}}^{\IP(k)}(\mathbf{x}(t) \in \bullet)-\pi_{\IP(k)}  \|_{\TV}\\& \ge f( \|\Pr_{\mathbf{x}}^{\RW(k)}(\mathbf{x}(Ct) \in \bullet)-\pi_{\RW(k)}  \|_{\TV}),
\\ \forall A \in \binom{V}{k},\, k \le n/2 \quad  &\|\Pr_{A}^{\mathrm{EX}(k)}(A_{t}\in \bullet)-\pi_{\mathrm{EX}(k)}  \|_{\TV} \\&\ge f( \|\Pr_{A}^{\widehat{ \RW}(k)}(\widehat{\mathbf{x}}(Ct) \in \bullet)-\pi_{\widehat{ \RW}(k)}  \|_{\TV}),
\\  \forall \eps \in (0,1/4), \quad& C \max_{k} \mixex (\eps) \ge   \mix^{\IP(n)} (f(\eps)),  
\end{split}  
\end{equation*}  
where $\widehat{ \RW}(k) $ is the projection of $\RW(k) $ obtained by forgetting the labeling of the particles?
\end{question} 

\subsection*{Asymptotic notation}
We write $o(1)$ for terms which vanish as $n \to \infty$. We write $f_n=o(g_n)$ or $f_n \ll g_n$ if $f_n/g_n=o(1)$. We write $f_n=O(g_n)$ and $f_n \lesssim g_n $ (and also $g_n=\Omega(f_n)$ and $g_n \gtrsim  f_n$) if there exists a constant $C>0$ such that $|f_n| \le C |g_n|$ for all $n$. We write  $f_n=\Theta(g_n)$ or $f_n \asymp g_n$ if  $f_n=O(g_n)$ and  $g_n=O(f_n)$. Throughout $\log \log n$ is to be interpreted as $\log \log (n \vee e^e)$, where $a \vee b:=\max\{a,b\}$ and $a \wedge b:=\min\{a,b\}$.
\subsection*{Organization of the paper}

In \S\ref{S:prelim}, we recall some properties of the exclusion process (its graphical construction and negative association), prove Proposition~\ref{p:LS}, show how the mixing time of $k$ particles is related to the mixing time of one particle conditioned on the others, and provide an auxiliary bound on the $L_2$ distance. In \S\ref{S:cham} we introduce the chameleon process as the main tool which allows us to bound the mixing time of one particle conditioned on the others. We also prove Theorem~\ref{thm:main1} subject to some technical propositions (the majority of whose proofs appear in the appendix), the most significant of which being Proposition~\ref{P:beta}. We give a detailed overview of how we use the chameleon process in \S\ref{s:overview} and turn these heuristics into formal arguments in \S\ref{S:neighbours} and \S\ref{S:loss}, proving Proposition~\ref{P:beta} for the case of large degree. In \S\ref{S:modifications} we show how to modify the arguments already presented in order to prove Proposition~\ref{P:beta} for small degree graphs, as well as Theorems~\ref{thm:main3} and~\ref{thm:maindeg}. We present the proof of the lower bounds in \S\ref{s:lower}, and give further applications of our results in \S\ref{s:examples}. 

\section{Preliminaries}\label{S:prelim}
\subsection{Mixing times}

Note that since $\mathrm{EX}(k) $ and $\IP(k)$  are irreducible and have symmetric
transition rates, the uniform distributions on their state spaces  $\binom{V}{k}$ (the set of all subsets of $V$ of size $k$) and $(V)_k$ (the set of all $k$-tuples of distinct vertices), respectively,  are stationary. Recall that the total variation distance of two distributions on a finite set $\Omega $ is 
\begin{align*}
\|\mu - \nu \|_{\TV}:=\sum_{a:\mu(a)>\nu(a)}(\mu(a)-\nu(a)).
\end{align*} Throughout, we use the convention that  $(X_t)_{t \ge 0 } $ is a continuous-time random walk on the graph $G$ with the same jump rates as above (\textit{i.e.}, a realisation of $\mathrm{EX}(1)$), and that $(A_t)_{t \ge 0}$ and $(\mathbf{x}(t))_{t\ge0}$  are  $\mathrm{EX}(k) $ and $\IP(k)$, respectively (we sometimes use $(\mathbf{w}(t))_{t\ge0},(\mathbf{y}(t))_{t\ge0}$ or $(\mathbf{z}(t))_{t\ge0}$ instead of  $(\mathbf{x}(t))_{t\ge0}$). We denote the uniform distribution on $V$ by $\pi $ and on $\binom{V}{k}$ and $(V)_k $ by $\pi_{\mathrm{EX}(k)}$ and $\pi_{\IP(k)}$. 
 We write $\Pr_x $ (resp.\ $\Pr_A^{\mathrm{EX}(k)} $, $\Pr_\mathbf{x}^{\mathrm{IP}(k)}$) for the law of $(X_t)_{t \ge 0 } $  given $X_0=x$ (resp.\  $(A_t)_{t \ge 0 } $  given $A_0=A$, $(\mathbf{x}(t))_{t\ge0}$ given $\mathbf{x}(0)=\mathbf{x}$).     The total variation $\eps$-mixing times of a single walk and of $\mathrm{EX}(k)$ are 
 \begin{align}\label{e:tvmixdef}
 \mix(\eps)=\mix^{\RW(1)}(\eps):= \inf \{t: \max_{x \in V} \|\Pr_x(X_t \in \bullet) - \pi \|_{\TV} \le \eps \}, \\\label{e:tvEXmixdef}
\mixex(\eps):= \inf \{t: \max_{A \in \binom{V}{k}  } \|\Pr_A^{\mathrm{EX}(k)}(A_t \in \bullet) - \pi_{\mathrm{EX}(k)} \|_{\TV} \le \eps \}. \end{align}
The mixing times $\mixIP (\eps)$ and $\mixrwk (\eps) $ of $\IP(k)$ and $\RW(k)$, respectively, are analogously defined.     Recall that when $\eps=1/4$ we omit it from the above
notation.   

The $\delta $ $L_{\infty}$-mixing time of a single walk (throughout, we consider the $L_2$ and $L_{\infty}$ distances and mixing times only w.r.t.\ a single walk) is defined as
\[  \mix^{(\infty)} (\delta):=\inf\{t: \max_{x,y \in V} |nP_t(x,y)-1  | \le \delta  \}  \]  and we set $\mix^{(\infty)}:=\mix^{(\infty)}(1/2)$. Recall that   $P_t$ denotes  the  heat-kernel of a single walk. The relaxation-time is defined as   $$\rel:=\sfrac{1}{\mathrm{gap}}=\lim_{t \to \infty} \sfrac{-t}{ \log [\max_{x \in V } P_t(x,x)-1/n]}, $$ i.e. it  is the inverse of the  \emph{spectral-gap}, the smallest positive eigenvalue of $- \cL $, where $\cL$ is the generator of a single walk.

%

We now note that we can characterize $\mix^{\RW(k)}(\eps) $ in terms of  $ \mix^{\RW(1)}(\eps / k) $, which in turn can be characterized in terms of the relaxation-time when $k=n^{\Omega (1)}$. This was used in the discussion following Theorem~\ref{thm:main3} (mixing for sublinear number of particles).
Indeed,
\begin{equation}
\label{e:kvs1}
\forall k \in \N ,\, \eps \in (0,1/4), \quad  \half \mix^{\RW(1)}(4\eps / k)  \le \mix^{\RW(k)}(\eps) \le \mix^{\RW(1)}(\eps / k). 
\end{equation}The second inequality is easy, while the first requires considering the separation distance, and noting that $$\min_{\mathbf{x},\mathbf{y}\in V^k} \Pr_{\mathbf{x}}^{\RW(k)}(\mathbf{x}(t)=\mathbf{y})=[\min_{x,y \in V } \Pr_{x}^{\RW(1)}(X_t=y)]^{k},$$ cf.\ \cite{lacoinproduct}.
 Generally, (\cite[Lemma 20.6 and Lemma 20.11]{levin}) for a  Markov chain on a state space $V$ of size $n$  with a symmetric generator  
\begin{equation}
\label{e:mixrel}
\forall \eps \in (0,1), \quad \rel | \log \eps| \le \mix^{\RW(1)} (\eps /2) \le  \mix^{(\infty)} (\eps ) \le \rel | \log n/ \eps|.   
\end{equation}  
 It follows by combining \eqref{e:kvs1} and \eqref{e:mixrel}  that for all $C \ge 1,\eps \in (0,1) $ and all $k \in[ 4 \eps n^{1/C},(n/\eps)^C] $
\begin{equation}
\label{e:mixkrel}
\sfrac{1}{2C} \rel \log (n/(2\eps) ) \le \mix^{\RW(k)}(\eps)  \le (C+1)\rel \log (n/\eps).
\end{equation}

We verify now the claimed bound on $r_\ast(\epsilon)$ of \eqref{e:tastsast}. 

For ($n$-vertex) regular graphs,     $P_t(v,v)-\sfrac{1}{n} \lesssim (t+1)^{-1/2}   $ (e.g.\! \cite{sensitivity,Lyonsev}) for all $t$. 
Hence $r_\ast(\epsilon)\le C(\epsilon)(\log n)^4$ for some constant $C$ depending only on $\epsilon$.
As \[\forall t \ge 0,\, i \in \N, \quad P_{it}(v,v)-\sfrac{1}{n} \ge (P_{t}(v,v)-\sfrac{1}{n})^i  \] (which follows via the spectral decomposition), by \eqref{e:mixrel} (used in the third inequality) we get that 
\begin{equation}
\label{e:ts_*2}
\begin{split}
r_\ast(\epsilon) \lesssim \mix^{(\infty)}( \sfrac{\epsilon n}{(\log n)^2}) &\lesssim_\epsilon (\log n)^4 \wedge\mix^{(\infty)} \sfrac{\log \log n}{\log n} \\&  \lesssim (\log n)^4 \wedge \rel \log \log n.
\end{split}
\end{equation}

\subsection{The spectral-profile, evolving sets and log-Sobolev}
\label{s:profiles}
As the generator $\cL$ is symmetric, it is self-adjoint with respect to the inner-product on $\mathbb{R}^V$ induced by $\pi$, given by $\langle f,g\rangle_\pi=\mathbb{E}_\pi[fg]:=\sum_x\pi(x)f(x)g(x)$.  Recall that the \emph{spectral-gap} is  $\mathrm{gap}:=\la_2 $  satisfies
\begin{equation}
\label{e:gapdef}
\la_2:=\min \{\EE(h,h)/\Var_{\pi}h:h \in \R^V \text{ is non-constant} \}, \end{equation}   
 where $\EE(f,f):=\langle-\mathcal{L}f,f\rangle_{\pi}=\half \sum_{x,y}\pi(x) \cL(x,y)(h(x)-h(y))^2 $.

Recall also that the \emph{log-Sobolev constant} is given by \begin{align}\label{e:logsob}
c_{\mathrm{LS}}:=\inf  \{\sfrac{\EE(h,h)}{\mathrm{Ent}_{\pi}h^{2}} : h^{2} \in (0,\infty)^{V}    \} ,
\end{align} where $\mathrm{Ent}_{\pi}f:=\mathbb{E}_{\pi}[f \log(f/\|f\|_1)]$.

Denote $\Lambda(\eps):=\min \{\EE(h,h)/\Var_{\pi}h : h \in \R^{V}, \pi( \mathrm{supp}(h)) \le \eps   \}$, where $ \mathrm{supp}(h):=\{x \in V :h(x) \neq 0 \}$ is the \emph{support} of $h$.

We now recall a couple of results from \cite{spectral}. While some of the results below were originally stated in the case where $\cL $ is of the form $K-I$, where $I$ is the identity matrix and $K$ is a transition matrix of a discrete-time Markov chain (possibly with non-zero diagonal entries), they hold for general $\cL$, as we can always write $\cL:=\max_x |\cL(x,x)|(K-I) $ for   some transition matrix $K$ (possibly with positive diagonal entries). (All the quantities considered below scale linearly in $\max_x |\cL(x,x)|$.)

\begin{proposition}[\cite{spectral} Lemma 4.2] For all $\eps \in (0,1) $
\label{p:spectral1}
\[(1-\eps)\Lambda(\eps) \ge c_{\mathrm{LS}} \log (1/\eps).  \]
\end{proposition}
\begin{remark}
It was shown in \cite{L2} that $17/c_{\mathrm{LS}} \le \max_{\eps \le 1/2  } \frac{\log (1/\eps)}{\Lambda(\eps) }$.
\end{remark}

\begin{proposition}[\cite{spectral} Lemma 2.1]
\label{p:spectral2} For any (non-zero)  $u \in \R_+^V$ we have that
\[\sfrac{\EE(u,u)}{\Var_{\pi}u} \ge \half \Lambda \left( 4\|u\|_1^2/ \Var_{\pi}u \right).\]
\end{proposition}
\begin{proof}[Proof of Proposition \ref{p:LS}] Let $f \in \R^V $ satisfy  $- \cL f=\la f$. We assume $ \|f\|_2 \ge 2 \|f\|_1 $, as otherwise there is nothing to prove.  By Propositions \ref{p:spectral1} and  \ref{p:spectral2}  we have that
\begin{align*}\la \ge \sfrac{\EE( f,f)}{\Var_{\pi}f} \ge \half \Lambda \left( 4\|f\|_1^{2} / \Var_{\pi}f \right) &\ge\half \Lambda \left( 4\|f\|_1^{2} /\|f\|_2^2 \right) \\&\ge c_{\mathrm{LS}} \log (\|f\|_2/2\|f\|_1). \qedhere  \end{align*}
\end{proof}

Recall that the $L_p$ norm of a signed measure $\sigma$ is
\begin{equation*}
\label{eq: Lpdef}
\|\sigma \|_{p,\pi}:=\|\sigma / \pi \|_p, \quad \text{where} \quad (\sigma / \pi)(x)=\sigma(x) / \pi(x).
\end{equation*}
In particular, for a distribution $\mu$ its $L_2$ distance from $\pi$ satisfies
\[\|\mu - \pi \|_{2,\pi}^2:=\|\mu / \pi  - 1\|_2^{2}=\mathrm{Var}_{\pi}(\mu / \pi). \]
Let $\mu_t:=\Pr_{\mu}^t $ and $u_t:=\mu_{t} / \pi$. It is standard that $\frac{d}{dt}\mathrm{Var}_{\pi}(u_{t} )=-2 \EE(u_{t},u_{t}) $ (\textit{e.g.}\ \cite[p.\ 284]{levin}). By \eqref{e:gapdef} $ \EE(u_{t},u_{t}) \ge \la_2   \mathrm{Var}_{\pi}(u_{t} ) $ from which it follows that $\frac{d}{dt}\mathrm{Var}_{\pi}(u_{t} ) \le -2\la_2  \mathrm{Var}_{\pi}(u_{t} ) $, and so by Gr\"onwall's lemma
\begin{equation}
\label{e:Poincare}
\|\mu_{t} - \pi \|_{2,\pi}^2 \le \|\mu - \pi \|_{2,\pi}^2 \exp(-2\la_2 t).  
\end{equation}
This is the well-known Poincar\'e inequality. The $\eps$ $L_p$-\emph{mixing time} is defined as
\[t_{\mathrm{mix}}^{(p)}(\eps):=\inf\{t:\max_{x}\|\Pr_x^t - \pi\|_{p,\pi} \le \eps \}. \]
It is standard (\textit{e.g.}\ \cite{spectral} or \cite[Prop.\ 4.15]{levin}) that for reversible Markov chains, for all $x \in V$ and $t$ we have  \begin{equation}
\label{e:maxdiag}
\max_{x,y}|\sfrac{P_t(x,y)}{\pi(y)}-1|=\max_{x}\sfrac{P_t(x,x)}{\pi(x)}-1  \quad \text{and} \quad \|\Pr_{x}^t - \pi \|_{2,\pi}^2=\sfrac{P_{2t}(x,x)}{\pi(x)}-1.
\end{equation}  Thus  $t_{\mathrm{mix}}^{(\infty)}(\eps^2)=2 t_{\mathrm{mix}}^{(2)}(\eps) $ for all $\eps \le ( \max_x \frac{1-\pi(x)}{\pi(x)})^{1/2}$. 
The \emph{spectral-profile}  \cite{spectral} and \emph{isoperimetric-profile/evolving-sets} \cite{evolving} bounds on the $\eps$ $L_{\infty}$ mixing time  are respectively given by
\begin{equation}
\label{e:profiles}
\begin{split}
&t_{\mathrm{sp}}(\eps):= \int_{4/n}^{4/\eps} \frac{2d \delta}{\delta \Lambda(\delta) }, 
\\ & t_{\mathrm{evolving-sets}}(\eps):=  \max_x |\cL(x,x)|\int_{4/n}^{4/\eps \wedge 1/2} \frac{4d \delta}{\delta \Phi^2(\delta) } + \rel \log (8/\eps) \Ind{\eps \le 8},
\end{split}
\end{equation}
where $\Phi(\delta):= \inf \left\{\frac{\sum_{a \in A,b \notin A}\pi(a) \cL(a,b)}{\pi(A)} : A \subset V \text{ such that } \pi(A) \le \delta \right\}$. A generalization of the well-known discrete Cheeger inequality is that \cite[Lemma 2.4]{spectral} 
\begin{equation}
\label{e:spb0}
\Phi^2(\delta)/(2\max_x |\cL(x,x)|) \le \Lambda(\delta) \le\Phi(\delta)/(1-\delta),
\end{equation}
from which it follows that $t_{\mathrm{sp}}(\eps) \le t_{\mathrm{evolving-sets}}(\eps) $. Theorem 1.1 in \cite{spectral} asserts that
\begin{equation}
\label{e:spb1}
\forall \eps \in (0,n] , \quad t_{\mathrm{mix}}^{(\infty)}(\eps) \le t_{\mathrm{sp}}(\eps) \le t_{\mathrm{evolving-sets}}(\eps)  .
\end{equation}
Plugging the estimate of Proposition \ref{p:spectral1} in \eqref{e:profiles} and then integrating over $\delta$ gives \cite[Corollary 4.1]{spectral} (cf.\ \cite{Kozma} for a slightly different argument).
\begin{proposition}
\label{p:specLS} There exists an absolute constant $C$ such that $$t_{\mathrm{sp}}(\sfrac{1}{2}) \le C\sfrac{\log \log n}{ c_{\mathrm{LS}}}.$$
\end{proposition} 
 
Using Proposition \ref{p:spectral2} (noting that $\| u_t \|_1 =1 $) the following refines \eqref{e:Poincare}.
\begin{proposition}[\cite{spectral} Theorem 1.1]
\label{p:spectral3} 

For any initial distribution $\mu $ we have that
\begin{equation}
\label{e:spb2}
\|\mu_{t} - \pi \|_{2,\pi}^2 \le M , \quad \text{if} \quad t \ge \int_{4/\|\mu - \pi \|_{2,\pi}^2}^{4/M} \frac{d \delta}{\delta \Lambda(\delta) }.   \end{equation}
In particular, for all $0<c<1$ we have that 
\begin{equation}
\label{e:spb3}
\|\mu_{t} - \pi \|_{2,\pi}^2 \le c \|\mu - \pi \|_{2,\pi}^2  , \quad \text{if} \quad t \ge \frac{\log (1/c)}{\Lambda(4/c \|\mu - \pi \|_{2,\pi}^2)}  .
 \end{equation}
\end{proposition}
The following lemma is a simple consequence of Proposition \ref{p:spectral1} together with \eqref{e:spb2}.
\begin{lemma}
\label{lem:r}
Let $r_{*}$ be as in \eqref{eq:rast}.  For every $c >0$ we have that   \begin{equation}
  \label{e:rLS}
  r_{*}(c) \lesssim_c \frac{ \log \log n}{ c_{\mathrm{LS}}\log n}.
  \end{equation}
\end{lemma} 
 
 \subsection{Graphical construction}\label{S:graphical}
We present a graphical construction of the processes EX$(k)$, IP$(k)$ and RW$(1)$, similar to that of Liggett \cite{liggettbook2} and Oliveira \cite{olive}. This construction enables us to define the processes on the same probability space, to then allow for direct comparison. We consider the following two ingredients:
\begin{enumerate}
\item a Poisson process $\Lambda$ of rate $\frac1d|E|$;
\item an i.i.d. sequence of uniformly-distributed $E$-valued random variables $\{e_n\}_{n\in\mathbb{N}}$.
\end{enumerate}Next we define the transpositions $f_e:V\to V$ for $e=\{u,v\}\in E$ as
\[
f_e(x)=\begin{cases}
u,&\mbox{if }x=v,\\
v,&\mbox{if }x=u,\\
x,&\mbox{otherwise.}
\end{cases}
\]We extend $f_e$ to act on subsets of $V$ and $k$-tuples by setting $f_e(A)=\{f_e(a):\,a\in A\}$ and $f_e(\mathbf{x})=(f_e(\mathbf{x}(1)),\ldots,f_e(\mathbf{x}(k)))
$. Then for $0\le s\le t<\infty$ we define permutations $I_{[s,t]}$ as $
I_{[s,t]}=f_{e_{\Lambda[0,t]}}\circ f_{e_{\Lambda[0,t]-1}}\circ\cdots\circ f_{e_{\Lambda[0,s)+1}},
$
for $\Lambda[s,t]>0$ (denoting the number of instances of the Poisson process $\Lambda$ during time interval $[s,t]$), otherwise we set $I_{[s,t]}$ to be the identity map. Hence $I_{[s,t]}$ is the composition of the transpositions $f_{e_j}$ that are chosen during $[s,t]$ composed in the order they occur. The following proposition is fundamental and its proof follows by inspection.

\begin{proposition}[Proof omitted]
Fix $t>0$. Then
\begin{enumerate}
\item For each $u\in V$ the process $\{I_{[s,s+t]}(u)\}_{t\ge0}$ is a realisation of $\mathrm{RW(1)}$ initialised at $u$ at time~$s$.
\item For each $A\in \binom{V}{k}$ the process $\{I_{[s,s+t]}(A)\}_{t\ge0}$ is a realisation of $\mathrm{EX}(k)$  initialised at $A$ at time $s$.
\item For each $\mathbf{x}\in (V)_k$ the process $\{I_{[s,s+t]}(\mathbf{x})\}_{t\ge0}$ is a realisation of $\mathrm{IP}(k)$  initialised at $\mathbf{x}$ at time $s$.
\end{enumerate}
\end{proposition}
\subsection{Negative association}
\label{s:NA} 
Let $Y_1,\ldots,Y_m$ be real-valued random variables. Let $\mathbf{Y}_A:=(Y_a)_{a \in A}$. We say that they are \emph{negatively correlated} if $\mathrm{Cov}(Y_i,Y_j) \le 0 $ for all $i \neq j$. We say that they are \emph{negatively associated}  if \[\text{(\textbf{NA})} \quad \mathbb{E}[f(\mathbf{Y}_A)g(\mathbf{Y}_B)] \le \mathbb{E}f(\mathbf{Y}_A)\mathbb{E}g(\mathbf{Y}_B),  \]
for all disjoint $A , B \subset [m]  $ and all $f,g$ non-decreasing w.r.t.\ the co-ordinate-wise partial order  $\le_{\mathrm{cw}}^i $  on $\R^{i}$ (for $i=|A|,|B|$, respectively) defined via $(x_1,\ldots,x_i)\le_{\mathrm{cw}}^i (y_1,\ldots,y_i) $ if $x_j \le y_j $ for all $j \in [i]$.   We say they are \emph{conditionally negatively associated} (\textbf{CNA}) if for all $D \subset [m] $ the same holds when conditioning on $\mathbf{Y}_D$, \textit{i.e.}, \begin{equation*}(\textbf{CNA})\begin{split} \quad \forall D \subset [m], \quad   &\mathbb{E}[f(\mathbf{Y}_A)g(\mathbf{Y}_B) \mid \mathbf{Y}_D  ] \\&\le \mathbb{E[}f(\mathbf{Y}_A) \mid \mathbf{Y}_D]\mathbb{E[}g(\mathbf{Y}_B) \mid \mathbf{Y}_D] \end{split}\end{equation*} for all disjoint $A,B$ and all non-decreasing $f,g$.  Borcea, Br\"and\'en and Liggett \cite{NA} showed that (for the exclusion process) $(\Ind{v \in A_t}:v \in V)$ is CNA, when $A_0$ is either deterministic or  a product measure. It follows by taking the limit as $t \to \infty$ that the CNA property holds also for the stationary distribution $\pi_{\mathrm{EX}(k)}=\mathrm{Unif}(\binom{V}{k}) $ (\textit{i.e.}, for $(\Ind{v \in A}:v \in V)$, when $A \sim \pi_{\mathrm{EX}(k)} $).

It is clear that the NA property implies pairwise negative correlation (\textit{i.e.}, $\mathrm{Cov}(\Ind{v \in A_t},\Ind{u \in A_t}) \le 0 $). While in \cite{olive} only the negative correlation property was used, we will make crucial use of the CNA property. 
\subsection{From mixing of $k$ particles to mixing of 1 particle conditioned on the rest}
\label{s:useful}
 By the contraction principle it suffices to bound the mixing time of $\IP(k)$ as for all $k$   
\begin{equation}
\label{e:trianglein}
\begin{split}
& \max_{A \in \binom{V}{k} }\|\Pr_{A}^{\mathrm{EX}(k)}[A_t \in \bullet]-\pi_{\mathrm{EX}(k)}(\bullet)\|_{\TV}\\& \le \max_{\mathbf{x} \in (V)_k }\|\Pr_{\mathbf{x}}^{\mathrm{IP}(k)}[\mathbf{x}(t) \in \bullet]-\pi_{\mathrm{IP}(k)}(\bullet)\|_{\TV} \le \max_{\mathbf{x},\mathbf{y} \in (V)_k }\Delta_{\mathbf{x},\mathbf{y}}(t),\\ &\text{where } \Delta_{\mathbf{x},\mathbf{y}}(t):=\max_{\mathbf{x},\mathbf{y} \in (V)_k }\|\Pr_{\mathbf{x}}^{\mathrm{IP}(k)}[\mathbf{x}(t) \in \bullet]-\Pr_{\mathbf{y}}^{\mathrm{IP}(k)}[\mathbf{y}(t) \in \bullet]\|_{\TV}.   \end{split} \end{equation}  
We may interpolate between any two configurations $\mathbf{x},\mathbf{y} \in (V)_k$ via a sequence of at most $k+1$ configurations,  $\mathbf{x}=\mathbf{z}_0,\mathbf{z}_1, \ldots ,\mathbf{z}_j=\mathbf{y} \in (V)_k$  such that $\mathbf{z}_i$ and $\mathbf{z}_{i-1}$ differ on exactly one co-ordinate for all $i \in [j]$. By symmetry, we may assume this is the $k$-th co-ordinate (the total variation distance at time $t$ w.r.t.\ two initial configurations is invariant under an application of the same permutation to their co-ordinates). By the triangle inequality, at a cost of picking up a factor $k$, we get that it suffices to consider two initial configurations which disagree only on their last co-ordinates:
\begin{equation}
\label{e:trianglein2}
\max_{\mathbf{x},\mathbf{y} \in (V)_k }\Delta_{\mathbf{x},\mathbf{y}}(t) \le k \max_{(\mathbf{w},y),(\mathbf{w},z) \in (V)_k:\mathbf{w} \in (V)_{k-1},y,z \in V }\Delta_{(\mathbf{w},y),(\mathbf{w},z)}(t).  \end{equation}   
Let $\mathbf{w}(t)=(\mathbf{w}_{1}(t),\ldots,\mathbf{w}_{k-1}(t))$ be the positions of the first $k-1$ co-ordinates at time $t$. Given  $\mathbf{w}(t)$,   the positions of the $k$-th co-ordinates at time $t $  of both configurations on the r.h.s.\ $y(t)$ and $z(t)$ converge (as $t \to \infty$) to the uniform distribution on $\mathbf{w}(t)^{\complement}:=V \setminus \{\mathbf{w}_{i}(t):i \in [k-1]  \} $. It is thus natural to compare the two to $U \sim \mathrm{Unif}(\mathbf{w}(t)^{\complement}) $ (given $\mathbf{w}(t)$) using the triangle inequality:
\begin{equation}
\label{e:trianglein3}\begin{split}
&\max_{(\mathbf{w},y),(\mathbf{w},z) \in (V)_k:\mathbf{w} \in (V)_{k-1} }\Delta_{(\mathbf{w},y),(\mathbf{w},z)}(t) \\&\le 2 \max_{(\mathbf{w},y) (V)_k:\mathbf{w} \in (V)_{k-1} } \| \cL_{(\mathbf{w}(t),y(t))}-\cL_{(\mathbf{w}(t),U)} \|_{\TV},\end{split} \end{equation}
where $\cL_X$ denotes the law of $X$. Hence we reduced the problem of showing that $\Delta_{\mathbf{x},\mathbf{y}}(t) \le \eps $ to that of showing that the maximum on the r.h.s.\ of \eqref{e:trianglein3} is at most $\sfrac{\eps}{2k} $. The total-variation distance in the maximum is that of the last co-ordinate from  $U \sim \mathrm{Unif}(\mathbf{w}(t)^{\complement})$, averaged over $\mathbf{w}(t) $. Hence loosely speaking, we reduced the problem to that of bounding the  $\sfrac{\eps}{2k} $-mixing time of the last co-ordinate, given the rest of the co-ordinates (in some averaged sense).  
 
\subsection{An auxiliary lower bound on the $L_2$ distance}
Let $\PP(V)$ be the collection of all distributions on $ V $. For $A  \subsetneq V $ and $\delta \in (0,1) $, let \[\PP_{A,\delta}:=\{\mu \in \PP(V): \mu(A) \ge \pi(A)+\delta \pi (A^c) \}.\] Note that $\nu_{A,\delta}:= \delta \pi_A+(1-\delta)\pi \in \PP_{A,\delta} $, where $\pi_A$
denotes $\pi$ conditioned on $A$ (i.e. $\pi_A(a) = \pi(a)\indic{a\in A}/\pi(A)$).
Moreover, $\min \{\delta':\nu_{A,\delta'} \in \PP_{A,\delta}   \}=\delta $. It is thus intuitive that for a convex distance function between distributions, $\nu_{A,\delta} $ is the closest distribution to $\pi$ in $\PP_{A,\delta}$. The assertion of the following proposition can be verified using Lagrange multipliers, noting that the density function of the distribution with respect to $\pi$ has to be constant on $A$ and on $A^{\complement}$. 
\begin{proposition}[\cite{L2} Proposition 4.1]
\label{prop: Lagrange}
 Let $A  \subsetneq V $. Denote $\nu_{A,\delta}:= \delta \pi_A+(1-\delta)\pi$.  Then
\begin{equation}
\label{eq: Lagrange}
\forall \delta \in (0,1)  \quad  \min_{\mu \in \PP_{A,\delta}}\| \mu-\pi\|_{2,\pi}^2=\|\nu_{A,\delta}-\pi\|_{2,\pi}^2=\delta^{2} \pi(A^{\complement})/\pi(A).
\end{equation}
\end{proposition}
\section{The chameleon process}\label{S:cham}
Our main tool is the use of the \emph{chameleon process}, a process invented by Morris \cite{morris} and used by Oliveira \cite{olive} and Connor-Pymar \cite{CP} to keep track of the distribution of a single particle in an interchange process, conditional on the locations of the other particles (see Proposition \ref{p:inkatb} for a precise formulation). As explained in \S\ref{s:useful}, this can be used to upper bound the mixing time of the interchange process (and thus also of the exclusion process). This is quantified in Proposition \ref{p:chambound}. We will make use of several variants of this process. In some situations the process consists of rounds of unvarying duration and is very similar to that used in \cite{olive}; whereas in others the length of rounds can vary in a way similar to \cite{morris}. The precise nature of the process depends on the values of $k$ and $d$, and the current state of the process. We shall present first the version most similar to \cite{olive} (and with which we prove Theorems~\ref{thm:main1} (general mixing bound) and~\ref{thm:maindeg} (mixing for graphs of high degree)) and show in \S\ref{SS:thm2proof} how this can be adapted to prove Theorem~\ref{thm:main3} (mixing for sublinear number of particles).

\subsection{Description of the process}\label{SS:cham_desc}
  We start this section with the construction of the chameleon process. 

The first step is to modify slightly the graphical construction of \S\ref{S:graphical}. We suppose now that edges ring at rate $2/d$ and an independent fair coin flip determines whether particles on a ringing edge switch places or not. More formally, consider the following ingredients:
\begin{enumerate}
\item a Poisson process $\Lambda=\{\tau_1,\tau_2,\ldots\}$ of rate $\frac2{d}|E|$;
\item an i.i.d.\! sequence of uniformly-distributed $E$-valued random variables $\{e_n\}_{n\in\mathbb{N}}$;
\item an i.i.d.\! sequence of coin flips $\{\theta_n\}_{n\in\mathbb{N}}$ with $\mathbb{P}(\theta_n=1)=\mathbb{P}(\theta_n=0)=1/2$.
\end{enumerate}
Recall the definition of $f_e$ from \S\ref{S:graphical} and set $f_e^1=f_e$ and let $f_e^0$ be the identity function. We modify the definition of the maps $I_{[s,t]}$ from \S\ref{S:graphical} as follows:
\[
I_{[s,t]}=f_{e_{\Lambda[0,t]}}^{\theta_{\Lambda[0,t]}}\circ f_{e_{\Lambda[0,t]-1}}^{\theta_{\Lambda[0,t]-1}}\circ\cdots\circ f_{e_{\Lambda[0,s)+1}}^{\theta_{\Lambda[0,s)+1}}.
\]
The joint distribution of the maps $I_{[s,t]}$, $0\le s\le t<\infty$ is the same as in \S\ref{S:graphical} by the thinning property of the Poisson process. 

The choice of $k$ in the following setup is relevant for obtaining an upper bound on $\mixIP(\eps)$.
The chameleon process is a continuous-time Markov process built on top of the modified graphical construction and consisting of \emph{burn-in periods}, and of \emph{rounds}. We first describe a version in which the duration of each round is a fixed parameter $t_\mathrm{round}$, known as the \emph{round length} and to be chosen in the sequel. This version will be used to prove Theorems~\ref{thm:main1} (general mixing bound) and~\ref{thm:maindeg} (mixing for graphs of high degree). In the chameleon process there is always one particle on each vertex, although not all particles are distinguishable. Each particle has an associated \emph{colour}: one of black, red, pink, and white. Formally, given a $(k-1)$-tuple $\mathbf{z}\in(V)_{k-1}$, let $\mathbf{O}(\mathbf{z}):=\{\mathbf{z}(1),\ldots,\mathbf{z}(k-1)\}$ be the set of coordinates of $\mathbf{z}$. The state space of the chameleon process is given by
\begin{align*}
&\Omega_k(V):=\\&\{(\mathbf{z},R,K,W):\,\mathbf{z}\in (V)_{k-1},\text{ and sets }\mathbf{O}(\mathbf{z}), R,K,W\text{ partition }V\}.
\end{align*}
We denote the state at time $t$ of the chameleon process started from $M_0=(\mathbf{z},R,K,W)$ as $M_t=(\mathbf{z}(t),\mathrm{R}_t,\mathrm{K}_t,\mathrm{W}_t)$. We say a particle at vertex $v$ is \emph{black} at time $t$ if $v\in\mathbf{O}(\mathbf{z}(t))$, \emph{red} if $v\in \mathrm{R}_t$, \emph{pink} if $v\in \mathrm{K}_t$, and \emph{white} if $v\in \mathrm{W}_t$. The black particles are distinguishable and their number remains constant throughout the process. We shall also denote the vector of positions of the black particles at time $t$ by $\BB_t$ (\textit{i.e.}, $\BB_t=\mathbf{z}(t) $). By abuse of notation we write $|\BB_t|$ for $|\mathbf{O}(\mathbf{z}(t))|$, the number of black particles (note that $\BB_t$ is a vector, not a set).  Marginally, the evolution of $\BB_t $ is simply that of the interchange process on $k-1 $ particles, starting from $\mathbf{z}$.  Conversely, the white (resp.\ pink and red) particles are indistinguishable, and their number changes as time varies.   
Suppose the chameleon process starts at time 0 from configuration $M_0=(\mathbf{z},R,\eset,W)$.

In order to define a quantity $H_t$ we suppose that all particles are either unmarked or marked and at time $t$ all particles are unmarked. Then suppose that at each instance during time interval $(t,t+1)$ at which an edge connecting an unmarked red particle and an unmarked white particle rings we mark both of these particles. We set $H_t$ to be half the number of marked particles at time $t+1$.

We make the following definition:
\begin{definition}
\label{d:p-good}
Let $\alpha \in (0,1/4)$ and $t>0$. We say that a configuration  $M_0=(\mathbf{z},R,\eset,W)$  of the chameleon process is $(\alpha,t)$-\emph{good} if $$\E_{M_{0}}[H_t ] \ge 2 \alpha (|R| \wedge |W| ).$$ Let $p(M_0)=p(M_{0},t):=\Pr_{M_{0}}[H_t \ge \alpha (|R| \wedge |W| ) ] .$
\end{definition}
For an  $(\alpha,t)$-good configuration with $\alpha \le 1/4$, by Markov's inequality 
\begin{equation}
\label{e:p>alpha/2}\begin{split}
p(M_0)&=1-\Pr_{M_{0}}[|R| \wedge |W|-H_t \ge (1- \alpha) (|R| \wedge |W| ) ] \\&\ge 1-\sfrac{\E_{M_{0}}[|R| \wedge |W|-H_t ]}{(1-\alpha) (|R| \wedge |W| )} \ge \sfrac{\alpha}{1-\alpha} \ge \sfrac{4\alpha}{3} .    \end{split} 
\end{equation}
Fix some $\alpha \in (0,1/4) $ to be determined later. At time 0, we start with no pink particles. Similarly, at the beginning of each round we have that $\mathrm{K}_t=\eset$.
We only start a round once we have an $(\alpha,t_\mathrm{round}-1)$-good configuration. Initially, we let the process make successive burn-in periods, each of duration  $t_{\mathrm{mix}}^{(\infty)}(n^{-10})$ and during which the process updates according to the updates of the underlying modified graphical construction, until the first time that at the end of a burn-in period we obtain an   $(\alpha,t_\mathrm{round}-1)$-good configuration. Similarly, if at the end of a round the configuration is not  $(\alpha,t_\mathrm{round}-1)$-good, then we let the process make successive burn-in periods, each of duration  $t_{\mathrm{mix}}^{(\infty)}(n^{-10})$, until the first time that at the end of a burn-in period we obtain an   $(\alpha,t_\mathrm{round}-1)$-good configuration. Denote the beginning of the $i$th round by $\rho_i $ and its end by $\hat \tau_i:=\rho_i+t_\mathrm{round} $. We now describe a round of the chameleon process.   

Each round consists of two phases. The first is a \emph{constant-colour relaxation phase} of duration $t_\mathrm{round}-1$, while the second is a \emph{pinkening phase} of unit length.
Loosely speaking, during a round the chameleon process evolves as the underlying interchange process, apart from the fact that pink particles are created by the recolouring of pairs of red and white particles (each pair consisting of a red and a white particle) during events known as \emph{pinkenings}.   Whenever an edge $e_j $ rings at some time $\tau_j $ for which the two endpoints are occupied by a red and a white particle at this time, we colour both these particles pink, unless we have already obtained $2  \lceil \alpha (|R| \wedge |W| )  \rceil $ pink particles. 

\begin{remark}
One place in which our chameleon process differs from Oliveira's process is that we will always depink at the end of a round, whereas Oliveira waits to have a substantial number of pink particles before depinking.
\end{remark}

The updates of the chameleon process during a single round are as follows:
\begin{itemize}
\item Intervals of time of the form $J_{i}:=(\rho_i,\hat\tau_i-1]$, for $i\in\mathbb{N}$, are \emph{constant-colour phases} during which the chameleon process updates according to the updates of the underlying modified graphical construction, \textit{i.e.}, if $t=\tau_j\in J_{i}$ for some $i\in\mathbb{N}$ then update as 
\[
(\mathbf{z}(t),\RR_t,\eset,\WW_t)=(f_{e_j}^{\theta_j}(\mathbf{z}(t_-)),f_{e_j}^{\theta_j}(\RR_{t_-}),\eset,f_{e_j}^{\theta_j}(\WW_{t_-})).
\]
\item Intervals of time of the form $\hat J_i:=(\hat\tau_i-1,\hat\tau_i)$, for $i\in\mathbb{N}$, are \emph{pinkening phases} during which we update as in the constant-colour phase except for times $t=\tau_j\in \hat J_i$ at which both 1 and 2 below hold: 
\begin{enumerate}
\item $e_j$ having a red endpoint $r\in \RR_{t_-}$ and a white endpoint $w\in \WW_{t_-}$,
\item $|\KK_{t_-}|<2 \lceil \alpha(|\RR_{t_-}|\wedge|\WW_{t_-}|) \rceil .$
\end{enumerate} For such times we update as 
\[
(\mathbf{z}_t,\RR_t,\KK_t,\WW_t)=(\mathbf{z}_{t_-},\RR_{t_-}\setminus\{r\},\KK_{t_-}\cup\{r,w\},\WW_{t_-}\setminus\{w\}).
\]
and call $t$ a \emph{pinkening time}.
\item Times of the form $t=\hat \tau_i $, for $i\in\mathbb{N}$, are called \emph{depinking times} and are of two types: \begin{itemize}\item Type 1 if $|\KK_{t_-}|=2 \lceil \alpha(|\RR_{t_-}|\wedge|\WW_{t_-}|) \rceil$  and an independent biased coin $\hat d_i$ is equal to 1, where $\Pr[\hat d_i=1 \mid M_{\rho_i}  ]=\sfrac{\alpha/2}{p(M_{\rho_i},t_\mathrm{round}-1)}$ (recall that $\rho_i $ is the beginning of the $i$th round). We then flip an independent fair (un-biased) coin $d_i$. If it lands heads ($d_i=1$) we colour all pink particles red, and if it lands tails we colour all pink particles white.
\item Type 2 if $|\KK_{t_-}|<2 \lceil \alpha(|\RR_{t_-}|\wedge|\WW_{t_-}|) \rceil$ or $\hat d_i =0 $. We then uniformly choose half of the pink particles (there is always an even number of pink particles) and colour these red, and the remaining half we colour white.
\end{itemize}
\end{itemize}
Observe that as soon as $\RR_t=\eset $ (resp.\ $\WW_t = \eset $) it will remain empty while $|\WW_s|=n-|\BB_0| $ (resp.\ $|\RR_s|=n-|\BB_0| $) for all $s \ge t$. After such time there will be no additional rounds.  

Note that by \eqref{e:p>alpha/2} we have that $\Pr[\hat d_i=1 \mid M_{\rho_i}  ] \le 1 $ and by definition of $p(\bullet,\bullet)$ we have that the probability of a type 1 depinking at time $\hat\tau_i$ is exactly $\alpha/2 $ for all $i$ (such that $|\RR_{\rho_i} | \wedge |\WW_{\rho_i} | \neq 0 $). This means that if the number of red particles at the beginning of the round is $r$, then it stays $r$ w.p.\ $1-\alpha/2 $ and otherwise with equal probability it changes to $r \pm \Delta(r) $, where $\Delta(r):=\lceil \alpha [r \wedge |\WW_{\rho_i} | ] \rceil=\lceil \alpha [r \wedge (n-|\BB_0| ) ] \rceil $.

For $M_0=(B,R,\varnothing,W)$ let $\hat M_t:=(\hat\BB_t,\hat\RR_t,\hat\WW_t)$ be the configuration at time $t$ obtained from the modified graphical construction with $\hat\BB_0=B$, $\hat\RR_0=R$ and $\hat\WW_0=W$, \textit{i.e.}\ without any colour-changing of particles. The definition of $(\alpha,t)$-good extends naturally to the process $\hat M_t$. Let $t_0:= t_{\mathrm{mix}}^{(\infty)}(n^{-10}) $ and
\begin{equation}
\label{e:beta1}\begin{split}
&\beta(\alpha,t):= \max_{B,R,W}\sup_{s \ge t_{0}  } \Pr[\hat M_s \text{ is not $(\alpha,t)$-good} \mid    \hat M_0=(B,R,W)],\end{split}
\end{equation}
where the maximum is taken over all partitions of $V$ into sets $\mathbf{O}(B),R,W$ with $B \in (V)_j $ for some $j \le n/2$ satisfying $\{B(i):i  \in [j] \} =\mathbf{O}(B) $. 

Recall the definition of $r_\ast(\epsilon)$ in \eqref{eq:rast}. For each $\epsilon\in(0,1)$, we define similar quantities:
\begin{equation}
\label{e:ts_*}
\begin{split}
& t_\ast(\epsilon):=\inf \{t:\max_{v \in V}P_t(v,v)-1/n \le \sfrac{\epsilon}{\log n} \},
\\ & s_\ast(\epsilon):=\inf \{t:\max_{v \in V}P_t(v,v)-1/n \le \sfrac{\epsilon}{t_\ast(\epsilon)} \}.
\end{split}
\end{equation} 
For the proof of Theorem~\ref{thm:main1} (general mixing bound) we will show that for some positive constants $\alpha,C_\mathrm{round},\epsilon $, if we take $$t_\mathrm{round}=C_\mathrm{round}(\rel + t_{\ast}(\epsilon)+s_{\ast}(\epsilon))+1$$ 
we have that $\beta(\alpha,t_\mathrm{round}-1) \le n^{-10} $.  We state this as the following proposition.

\begin{proposition}\label{P:beta}
There exist constants $\epsilon,\alpha, C_\mathrm{round} >~0 $, such that for all $n$ sufficiently large
\[\beta(\alpha,C_\mathrm{round}(\rel+t_\ast(\epsilon)+s_\ast(\epsilon)))\le n^{-10}.\]
\end{proposition}

 We will explain in \S\ref{s:proofofmain} how this implies the assertion of Theorem \ref{thm:main1} (general mixing bound). 
For Theorem~\ref{thm:maindeg} (mixing for graphs of high degree) we show that it suffices to take $t_\mathrm{round}=C_\mathrm{round}\rel +1 $, see Proposition~\ref{P:betadeg} (which is the analogue of the previous proposition). The situation is more involved for the proof of Theorem~\ref{thm:main3} (mixing for sublinear number of particles), see \S\ref{SS:thm2proof}.

\subsection{Further technical results}\label{SS:technical}
We present the key tools regarding the chameleon process that, together with Proposition~\ref{P:beta}, will be used to complete the proof of Theorem~\ref{thm:main1} (general mixing bound) in the following subsection.

Following Oliveira \cite{olive} we introduce a notion of \emph{ink}, which represents the amount of \emph{redness} either at a vertex or in the whole system. We write $\mathrm{ink}_t(v)$ for the amount of ink at vertex $v$ at time $t$ defined as 
$
\mathrm{ink}_t(v):=\indic{v\in \RR_t}+\frac12\indic{v\in \KK_t},
$
and the amount of ink in the whole system at time $t$ as
$
\mathrm{ink}_t:=|\RR_t|+\frac12|\KK_t|.
$
Notice that, by the construction of the chameleon process, the value of $\mathrm{ink}_t$ can only change at depinking times of Type 1.
 The following proposition
 links the amount of ink at a vertex to the probability that vertex is occupied by the $k$-th particle, in a $k$-particle interchange process. The statement is identical to Proposition~5.2 of Oliveira (the difference being our chameleon process is constructed slightly differently). The proof is almost identical to the proof of Lemma 1 of \cite{morris}, and we include our version  for completeness.

\begin{proposition}[{Proof in Appendix~\ref{A:expink}}]
\label{p:inkatb}
Consider a realisation $(\mathbf{x}(t))_{t\ge0}$ of the $k$-particle interchange process started from configuration $\mathbf{x}=(\mathbf{z},x)$ and a corresponding chameleon process started from configuration $(\mathbf{z},\{x\},\varnothing,V\setminus(\mathbf{O}(\mathbf{z})\cup\{x\}))$. Then for each $t\ge0$ and $\mathbf{b}=(\mathbf{c},b)\in (V)_k$, $\mathbf{c} \in  (V)_{k-1}$,
\[
\Pr^\mathrm{IP(k)}[\mathbf{x}(t)=\mathbf{b}]=\mathbb{E}\left[\mathrm{ink}_t(b)\indic{\mathbf{z}(t)=\mathbf{c}}\right].
\]
\end{proposition}
\begin{remark}
Right after we colour two particles pink, since we do not reveal whether the edge ring of the edge connecting them was ignored or not, we cannot tell which one of them is at which location. The action of colouring them by pink symbolizes this uncertainty, which is the real reason that the assertion of the last proposition holds.
\end{remark}
The next observation is that $\mathrm{ink}_t$ is a martingale. This can be readily checked from the behaviour of the chameleon process at depinking times. Moreover as $t\to\infty$, $\mathrm{ink}_t$ converges to one of the two absorbing states 0 and $=n-k+1$. We define $\mathrm{Fill}$ as the event that this limit is $n-k+1$, \textit{i.e.}, that eventually the only particles present in the system are red and black. One consequence of the martingale property of $\mathrm{ink}_t$ is that $\Pr[\mathrm{Fill}]=(n-k+1)^{-1}.$

\begin{lemma}[cf.\ \cite{olive} proof of Lemma 7.2] \label{L:Fill}
 The event $\Fill $ is independent of $(\BB_t:t \ge 0)$. 
\end{lemma}
\noindent \emph{Sketch proof:} This follows from the fact that the coins $(d_i:i \in \N)$ are independent of the coins $(\hat d_i:i \in \N)$ and of the graphical representation.   
\qed
\\ \\
Let us write $\mathbb{\widehat E} $ and $\mathrm{\widehat P} $ for the expectation and probability conditioned on the event $\Fill$. We may add subscript $(\mathbf{w},y) \in (V)_k $ such that $\mathbf{w} \in (V)_{k-1}$ and $y \in V $  to indicate that the initial configuration of the interchange process is $(\mathbf{w},y) $ and thus for the chameleon process $\RR_0=y $ and $\BB_t=\mathbf{w}(t)$ for all $t$, where $\mathbf{w}(t)=(\mathbf{w}_{1}(t),\ldots,\mathbf{w}_{k-1}(t))$ is the vector of the positions of the first $k-1$ co-ordinates at time $t$. In this case, we let $y(t)$ denote the position of the $k$-th co-ordinate at time $t$. The main inequality relating the total-variation distance to the chameleon process is the following:

\begin{proposition}[{\cite[Lemma 2]{morris}, \cite[Lemma 6.1]{olive}; proof in Appendix~\ref{s:p4.6}]}]
\label{p:chambound}
Let $\Delta_{\mathbf{x},\mathbf{y}}(t)$ be as in  \eqref{e:trianglein}. Then
\begin{equation}
\label{e:IPtoCP}
\max_{\mathbf{x},\mathbf{y} \in (V)_k }\Delta_{\mathbf{x},\mathbf{y}}(t) \le 2k \max_{(\mathbf{w},y)\in (V)_k:\mathbf{w}\in (V)_{k-1},y \in V  }  \hE_{(\mathbf{w},y)}[ 1 - \ink_t / (n-k+1) ]. \end{equation}
\end{proposition}
The following proposition, which is essentially Proposition B.1 in \cite{olive}, allows us to bound the r.h.s.\ of \eqref{e:IPtoCP}.  For $j\in\mathbb{N}$, we define event \[A(j):=\{\text{config.\! at time }t(j) \text{ is not }(\alpha,t_\mathrm{round}-1)\text{-good}\}.\] The term $t_{\mathrm{mix}}^{(\infty)}(n^{-10}) $ below corresponds to the initial burn-in period, while the error term $\widehat{\Pr}_{(\mathbf{w},y)}[\cup_{j=0}^{i-1}  A(j) ] $ corresponds to the probability that additional burn-in periods occurred by the end of the $i$th round (\textit{i.e.}, that at time $t(j):=t_{\mathrm{mix}}^{(\infty)}(n^{-10})+j t_\mathrm{round} $ the configuration was not good). Hence, the assertion of the proposition is that the expected fraction of ``missing ink" $ 1 - \ink_t / (n-k+1) $ decays exponentially in the number of rounds.

\begin{proposition}[{Proof in Appendix~\ref{A:2}]}]
\label{p:chambound2}  There exists $c_{\alpha} \in (0,1) $ such that for all $i \in \N $ and $(\mathbf{w},y)\in (V)_k $, 
\begin{equation}
\label{e:IPtoCP'}
\begin{split}
  \hE_{(\mathbf{w},y)}[ 1 - \ink_{t(i)} / (n-k+1) ] & \le \sqrt{n-k+1} c_{\alpha}^i+ \widehat{\Pr}_{(\mathbf{w},y)}[\cup_{j=0}^{i-1}  A(j) ]   \\ &  \le  \sqrt{n-k+1} c_{\alpha}^i+(\Pr[\mathrm{Fill}])^{-1} \Pr_{(\mathbf{w},y)}[\cup_{j=0}^{i-1}  A(j) ] \\&\le \sqrt{n-k+1} c_{\alpha}^i + i (n-k+1) \beta(\alpha,t_\mathrm{round}-1).
\end{split} \end{equation}
\end{proposition}

\subsection{Proof of Theorem \ref{thm:main1} (general mixing bound) }
\label{s:proofofmain}
\begin{proof} 
Firstly, recall again that for $n$-vertex regular graphs, $P_t(v,v)-\sfrac{1}{n} \lesssim (t+1)^{-1/2}   $ from which it follows that there exists a universal constant $\kappa$ such that for any $\epsilon\in(0,1)$ $t_\ast(\epsilon)+s_\ast(\epsilon)\le 2 r_\ast(\epsilon^3/\kappa^2)$.
Next, using  sub-multiplicativity \cite[p.\ 54]{levin} we have that $\mixex((2n)^{-i})$ $\le i\mixex(\sfrac{1}{4n})$. It follows that it suffices to consider $\eps=\sfrac{1}{4n}$. 
We may assume $n $ is at least some sufficiently large constant $N$ (this was implicitly/explicitly used in several places), as there are only finitely many graphs for $n\le N$ (and hence finitely many processes, since we assume edge-rates are all $1/d$). Combining Propositions \ref{P:beta}, \ref{p:chambound} and \ref{p:chambound2} concludes the proof (use Proposition \ref{p:chambound2} with $i = \lceil \sfrac{4}{1-c_{\alpha}} \log n   \rceil $, noting that the term $t_{\mathrm{mix}}^{(\infty)}(n^{-10})$ in the definition of $t(i)$ is $\lesssim \rel \log n$).
\end{proof}
\section{An overview of our approach}
\label{s:overview}
The approach taken by Oliveira  \cite{olive} is to let the constant-colour  and the pinkening phases both be of order $\mix^{\mathrm{EX}(2)}$. The two main steps in his analysis are  (i) to show that  $\mix \asymp \mix^{\mathrm{EX}(2)}$ and (ii) by the choice of the duration of the constant-colour phase, using a delicate negative correlation argument deduce that with probability bounded from below a certain fraction of the red (or white, whichever set is of smaller size) particles will become pink in each pinkening phase. Both steps are much more difficult than what one might expect. As explained below, assuming regularity allows us to take our pinkening phase to be of duration of one time unit. 

 Since the red and the white particles play symmetric roles, we may assume that at the end of the last round prior to the current time we have $r \le (n-k+1)/2 $ red particles (\textit{i.e.}, there are at least as many white particles as there are red; otherwise, switch their roles in what comes).

We now sketch the main ideas behind the proof of Proposition \ref{P:beta} in more detail. We will focus in this section and in \S\ref{S:neighbours}--\ref{S:loss} on graphs with degree $d$ satisfying $d\ge 10^4$, and describe how to extend the argument to all $d$ in \S\ref{SS:smalld}.

In order for a configuration to be $(\alpha,t)$--good for some constant $\alpha$, it suffices that (given the current configuration) with probability bounded from below, after $t$ time units, at least some $c$-fraction of the red particles will have at least a $c$-fraction of their neighbours white. To see this, observe that if a red particle has $j \ge cd$ white neighbours, the chance an edge connecting it to any of them rings before the two particles at its end-point moved is $  \frac{ j}{2d-1}=\Omega(c)$ (and the probability this happens in at most 1 time unit is $\Omega(c) $).

Observe that if (at the end of a constant-colour phase) a vertex has at most $(\sfrac{r}{n}+\sfrac{c}{4})d$ red neighbours and at most $(\sfrac{k-1}{n}+\sfrac{c}{4})d $ black neighbours, then it has at least $\half(1-\sfrac{k-1}{n}-c)d \ge (\sfrac{1}{4}-\sfrac{c}{2})d $ white neighbours (as $r \le \half(n-k+1)$). Hence, instead of controlling the number of white neighbours of a vertex, conditioned on it being red, we may control the number of red neighbours  and the number of black neighbours separately. This is done is \S\ref{S:redvred} and
\S\ref{S:redvblack}, respectively. 

We say that two particles \emph{interacted} if an edge connecting them rang. 
Exploiting the CNA property in conjunction with $L_2$-contraction considerations allows us to control the number of red with red interactions during the pinkening phase, provided that $t_\mathrm{round}-1 \ge C \rel $. \textcolor{black}{This $L_2$ argument is the key that allows us to avoid taking  $t_\mathrm{round} \ge C \mix $, as Oliveira does.}

Controlling the number of red with black interactions during a pinkening phase requires exploiting the NA property to derive certain large deviation estimates for the occupation measure of the black particles, as well as a certain decomposition which allows us to overcome the dependencies between the black and the red particles. \textcolor{black}{This is the most difficult and subtle part of the argument.}     

\subsection{Controlling red neighbours: an overview} It turns out that controlling the number of red neighbours is the easy part. Observe that the dynamics performed by the red particles during a single constant-colour phase of the chameleon process is simply a symmetric exclusion process. Thus  by NA if given $\RR_{\rho_i} $ (recall that $\rho_i $ is the beginning of the $i$th round) the expected number of red particles neighbouring vertex $v$ at time $\rho_i+ t_\mathrm{round}-1$ is at most $( \sfrac{r}{n}+c)d  $, the (conditional) probability (given $\RR_{\rho_i}$) of having more than $  ( \sfrac{r}{n}+2c)d  $ red particles around vertex $v$ at time $\rho_i+ t_\mathrm{round}-1$ can be made arbitrary small, provided $d$ is large enough (as explained above, we may assume the degree is arbitrarily large; where $c>0$ is some small absolute constant). Crucially, by CNA the same holds even when we condition on $v$ being occupied by a red particle at the end of the constant-colour phase (\textit{i.e.}, at time $\rho_i+ t_\mathrm{round}-1 $). This motivates considering the following set for round $i$:
\begin{equation}
\label{e:nicei}\begin{split}
&\mathrm{Nice}(i):=\\&\{v: \text{expected no.\! red neighbrs of $v$ in  }t_\mathrm{round}-1 \text{ time units} \le d( \sfrac{|\mathrm{R}_{\rho_i}|}{n}+ c)  \}, \end{split}\end{equation}
where the expectation inside the event above is conditional on $\mathrm{R}_{\rho_i}$.

 It suffices to control the expected number of red  particles which lie in $\mathrm{Nice}(i) $ at the end of the constant-colour phase of the $i$th round, as by the above reasoning it is very unlikely for each such red particle to have more than $d( \frac{|\mathrm{R}_{\rho_i}|}{n}+ 2c)$ red neighbours at that time. Using NA one can argue that if the last expectation is large, then the  actual number of such red particles is unlikely to deviate from it by a lot. However, it turns out to not be necessary for our purposes.

To control the aforementioned (conditional) expectation (given $\mathrm{R}_{\rho_i}$) we observe that the last expectation equals 
\begin{equation}
\label{e:Pnicecomp}
|\mathrm{R}_{\rho_i}| \Pr_{\mathrm{Unif}(\mathrm{R}_{\rho_i})}[X_{t_\mathrm{round}-1} \in \mathrm{Nice}(i)].
\end{equation}
By Proposition~\ref{prop: Lagrange} and some algebra (see Lemma \ref{L:piNice} for the actual details) we deduce that if $ \Pr_{\mathrm{Unif}(\mathrm{R}_{\rho_i})}[X_{t_\mathrm{round}-1} \in \mathrm{Nice}(i)]$ is smaller than $\pi(\mathrm{Nice}(i))-c $,  then we must have that the $L_2$ distance of $\Pr_{\mathrm{Unif}(\mathrm{R}_{\rho_i})}[X_{t_\mathrm{round}-1} \in \bullet]  $ from $\pi$ is proportional to $\sfrac{1}{\sqrt{ \pi(V \setminus \mathrm{Nice}(i))}} $.  By a simple counting argument (see Lemma \ref{L:niceC}), we must have that 
\begin{equation}
\label{e:countingargument}
|V \setminus \mathrm{Nice}(i)| \lesssim |\mathrm{R}_{\rho_i}|,
\end{equation}
which means that the last $L_2$ distance is   $\gtrsim  \sfrac{1}{\sqrt{ \pi(\mathrm{R}_{\rho_i})}} \asymp \|\mathrm{Unif}(\mathrm{R}_{\rho_i}) - \pi \|_{2,\pi}$. 
 
 In simple words, if the duration of the constant-colour relaxation phase is such that the $L_2$ distance from the uniform distribution of a random red particle, chosen uniformly at random, drops by the end of the phase by some sufficiently large constant factor, compared to its value at the beginning of the round (which is $\|\mathrm{Unif}(\mathrm{R}_{\rho_i}) - \pi \|_{2,\pi}$), then with a large probability (in some quantitative manner) a certain fraction of the red particles will have few red neighbours at the end of the relaxation phase (Lemma~\ref{L:piNice}).  Using the Poincar\'e inequality \eqref{e:Poincare} it follows from our choices of the durations of the rounds that the aforementioned $L_2$  distance indeed drops by a constant factor, which can be made arbitrarily large by adjusting the constant $C_\mathrm{round}$.

\textcolor{black}{For the sake of being precise, we note that the above argument breaks down when $|\mathrm{R}_{\rho_i}| \wedge |\mathrm{W}_{\rho_i}| \ge \varrho_0 n $ for a certain $\varrho_0$ depending on the choice of $c$. Fortunately, in this regime we can work directly  with the white particles and argue that at the end of the constant colour phase the expected number of red particles with at least $c\varrho_0 d$  white neighbours is of order $n$. This will be obtained as a relatively simple consequence of  the Poincar\'e inequality, and only requires the duration of a round to be $\Omega(\rel) $.  }

\subsection{Controlling black neighbours: an overview}\label{s:blackoverview}Controlling the number of black neighbours turns out to be a much harder task. 
By abuse of notation (treating $\BB_{t}$ and $\BB_{t+s} $ as sets) consider
\begin{equation}
\label{e:Zvts}
Z_v(t,s):=\sum_{u}\Ind{u \in \BB_t}P_s(u,N(v))=\E[ |\BB_{t+s} \cap N(v)| \bigm|  \BB_t ] ,
\end{equation} where $N(x) $ is the neighbour set of vertex $x$. Using the NA property it is not hard to show (see Lemma \ref{L:largedev}) that if $(\Ind{u \in \BB_0}:u \in V) $ has marginals close to $k/n $ (\textit{i.e.}, after a burn-in period) then  $\Pr[Z_v(t,s) > (\frac{k}{n}+c)d]$ decays exponentially in $\sfrac{1}{\max_{x,y}P_s(x,y)} $ for all $s$ and $v \in V $. This estimate, which is one of the key ideas in this work, is inspired from the proof of the main result in \cite{BH} (and a variant of that  result whose proof also utilized NA).  If $s \ge t_{\ast}(\epsilon)$ it is immediate from the definition of $t_*(\epsilon)$, that $\max_{x,y}P_s(x,y) \le \sfrac{\epsilon}{\log n} $ and so this probability is $\ll n^{-20} $ for suitably chosen $\epsilon$. 

Unfortunately, this does not yield the desired conclusion, since conditioned on having a red particle at $v$ at time $t+s$ changes the distribution of the number of black neighbours of $v $ at that time. To overcome this difficulty, we have to take the duration of the round to be $ t_\mathrm{round}:=C_{\mathrm{round}}( t_{\ast}(\epsilon) +s_{*}(\epsilon)  +\rel)+1 $, and consider two cases. We show that for each red particle, the expected number of neighbouring particles  it has at the end of the constant-colour relaxation phase, which interacted with it during the first $t_*(\epsilon) $ time units of the round can be made at most $cd$, provided we take $C_{\mathrm{round}} $ to be large enough (see Lemma~\ref{L:interact}). This is obtained by exploiting the definition of $s_*(\epsilon)$, along with a delicate use of negative correlation. Lastly, we show that a variant of the aforementioned large deviation estimate applies to the black particles that did not interact during the first $t_*(\epsilon) $ time units of the round with the considered red particle, and that for such black particles we need not worry about the dependencies with this red particle.

\section{Results to control neighbours of red particles}\label{S:neighbours}
\subsection{The red neighbours}\label{S:redvred}
Recall that $P_t$ is the heat-kernel of a single walk on $G$. 
We write $T$ for $t_\mathrm{round}-1$, \emph{i.e.}\ $T$ denotes the length of a constant-colour phase. Motivated by \eqref{e:nicei} and the following paragraph we make the following definition.

\begin{definition}
For each subset $S\subseteq V$, let $e_T(v,S):=\sum_{u:\,v \sim u}P_T(u,S) $ and define $\mathrm{Nice}(S)$ as:
\[
\mathrm{Nice}(S):=\Big\{v\in V:\,e_T(v,S)< d\left(\sfrac1{32}+\sfrac{|S|}{n}\right)\Big\}.
\]
\end{definition}
\begin{remark}
We will later (see \S\ref{SS:smalld}) modify this definition by replacing adjacency with proximity. This will allow us to deal with the case of small degree.
\end{remark}From this definition we see that the set $\mathrm{Nice}(S)$ consists of vertices which have ``few" neighbours (in expectation) at time $T$ which came from (at time 0) the set $S$. The reader should think of $S$ as the set occupied by the red particles at the beginning of a round. In \S\ref{S:loss} we make use of this definition with $S$ being the set of red vertices. Motivated by \eqref{e:countingargument}, we now lower-bound the size of $\mathrm{Nice}(S)$ by a simple counting argument, involving only its definition:
\begin{lemma}\label{L:niceC}
For each $S\subseteq V$, \[
|\mathrm{Nice}(S)^\complement|\le \left(\sfrac1{32}+\sfrac{|S|}{n}\right)^{-1}|S|.
\]
\end{lemma}
\begin{proof}
The definition of $\mathrm{Nice}(S)$ yields that $ d\left(\sfrac1{32}+\sfrac{|S|}{n}\right)|\mathrm{Nice}(S)^\complement|=\sum_{v\in \mathrm{Nice}(S)^\complement} d\left(\sfrac1{32}+\sfrac{|S|}{n}\right) $ is
\begin{align*}
\le \sum_{v\in \mathrm{Nice}(S)^\complement}\sum_{u:\,v \sim u}P_T(u,S) \le \sum_u\sum_{v:\,v \sim u}P_T(u,S)=d|S|,
\end{align*}which proves the result.
\end{proof}
\begin{remark}
The analogous result which is used for the small degree case is Lemma~\ref{L:niceC2}.
\end{remark}
The next lemma (motivated by \eqref{e:Pnicecomp}) gives a bound on the probability that a random walk started uniformly from set $S$ is in $\mathrm{Nice}(S)$ at time $T$. 
The proof  uses Proposition~\ref{prop: Lagrange} combined with the Poincar\'e inequality \eqref{e:Poincare}.

\begin{lemma}[{Proof in Appendix~\ref{proofpiNice}]}]\label{L:piNice}
Denote the uniform distribution on $S$ by $\pi_S $.
For each $\varepsilon\in(0,1)$, there exist $C_{\ref{L:piNice}}(\varepsilon)>1$ such that for all $C_\mathrm{round}>C_{\ref{L:piNice}}(\varepsilon)$ and all $S\subset V$ with $2|S|\le n$,
\[\Pr_{\pi_S}[X_T\in\mathrm{Nice}(S)]\ge\pi(\mathrm{Nice}(S))-\varepsilon.\] 
\end{lemma}
\begin{remark}
See Lemma~\ref{L:piNice2} for the version of this result to be used in the proof of Theorem~\ref{thm:main3} (mixing for sublinear number of particles).
\end{remark}
For $S\subseteq V$ we define $N(S):=\mathrm{Nice}(S)\cap I_{[0,T]}(S)$, which are the $\mathrm{Nice}(S)$ vertices occupied at time $T$ by particles initially in $S$,  and further for $\theta\in(0,1)$, we define a subset of $N(S)$ as
\[
BN(S)_\theta:=\Big\{v\in N(S):\,\sum_{u:\,v \sim u}\indic{I_{[0,T]}^{-1}(u)\in S}>\theta  d\Big\},
\]
which are the  $N(S)$ vertices which have ``many" ($>\theta  d $) neighbours also occupied at time $T$ by particles initially in $S$. Similarly, we define a set $GN(S)_\theta$ to be $N(S)\setminus BN(S)_\theta$ (here the $B$ in $BN(S)_{\theta}$ stands for ``bad" and the $G$ in $GN(S)_{\theta}$ for ``good"). We control the number of such vertices with the following lemma (think of $\theta$ below as being in $(\sfrac{|S|}{n}+\sfrac{1}{32},\sfrac{|S|}{n}+\sfrac{1}{16}]$, and observe that for such $\theta$ we may pick $\la>0$ sufficiently small such that $-\lambda\theta+(e^\lambda-1)\left(\sfrac1{32}+\sfrac{|S|}{n}\right) \le -c \la$).
\begin{lemma}\label{L:BNR}
For each $S\subseteq V$, $\theta\in(0,1)$, $\lambda>0$ and $v\in V$, 
\[
\Pr[v\in BN(S)_\theta\mid v\in N(S)]<\exp\left\{ d\left(-\lambda\theta+(e^\lambda-1)\left(\sfrac1{32}+\sfrac{|S|}{n}\right)\right)\right\}.
\]
\end{lemma}
\begin{remark}
Proving this lemma relies crucially on the CNA property. See Lemma~\ref{L:BNR2} for the version of this result for small degree graphs.
\end{remark}
\begin{proof}For each $v\in\mathrm{Nice}(S)$ and $\lambda>0$,
\begin{align*}
\Pr[v\in BN(S)_\theta|\,v&\in N(S)]=\Pr\left[\sum_{u:\,v \sim u}\indic{I^{-1}_{[0,T]}(u)\in S}>\theta  d \Bigm| v\in I_{[0,T]}(S)\right]\\
\mbox{(Chernoff)}\qquad&\le e^{-\lambda\theta d}\,\mathbb{E}\left[\exp\left\{\lambda\sum_{u:v \sim u}\indic{I^{-1}_{[0,T]}(u)\in S}\right\}\Bigm|v\in I_{[0,T]}(S)\right]\end{align*}\begin{align*}
\mbox{(CNA then NA)}\qquad&\le e^{-\lambda\theta d}\prod_{u:\,v \sim u}\mathbb{E}\left[\exp\left\{\lambda\indic{I^{-1}_{[0,T]}(u)\in S}\right\}\right]\\
&=e^{-\lambda\theta d}\prod_{u:\,v\sim u}\left(1+(e^\lambda-1)\Pr[u\in I_{[0,T]}(S)]\right)\\
(1+x\le e^x)\qquad&\le e^{-\lambda\theta d}\exp\left\{\sum_{u:\,v \sim u}(e^\lambda-1)\Pr[u\in I_{[0,T]}(S)]\right\}\\
&=\exp\left\{-\lambda\theta\hat d+(e^\lambda-1)\sum_{u:\,v \sim u}P_T(u,S)\right\}\\
(v\in \mathrm{Nice}(S))\qquad&<\exp\left\{d\left(-\lambda\theta+(e^\lambda-1)\left(\sfrac1{32}+\sfrac{|S|}{n}\right)\right)\right\},
\end{align*}as required.
\end{proof}
\subsection{The black neighbours}\label{S:redvblack}
Recall the modified graphical construction from \S\ref{SS:cham_desc}. Recall also that an \emph{interaction} occurs between two particles occupying vertices $u,v$ in the  exclusion/interchange process (constructed using the modified graphical construction as described in~\S\ref{SS:cham_desc}) when edge $\{u,v\}$ rings. For $a,b\in V$, and $t\ge0$, let $N_{t}(a,b)$
denote the number of interactions during time interval $[0,t]$ of the particles at vertices $a$ and $b$ at time 0.

For each $v\in V$ and $0\le t<T$, we also define a random variable $\hat N_{t}(v)$ to be the number of interactions during time interval $[0,t]$ of the particle at vertex $v$ at time $0$ with its time-$T$ neighbours, \textit{i.e.},
\[
\hat N_{t}(v):=\sum_{u:\, I_{[0,T]}(v) \sim u}N_{t}(v,I^{-1}_{[0,T]}(u)).
\]
The next lemma gives control on the expected value of $\hat N_{t_\ast(\epsilon)}(v)$. We will apply this to control the expected number of black particles which interact with red particles during time interval $[0,t_\ast]$  for any initial configuration of black and red particles.

\begin{lemma}\label{L:interact} 
For all $\epsilon\in(0,1)$ and $C_\mathrm{round}\ge 1$   we have
\[\max_{v\in V}\mathbb{E}[\hat N_{t_\ast(\epsilon)}(v)]\le 8d\epsilon.\]
\end{lemma}
\begin{remark}
The proof of this lemma makes use of the NA property. See Lemma~\ref{L:interact2} for the version of this result for small degree graphs.
\end{remark}
\begin{proof}
 We first write
\begin{align*}
\hat N_{t_\ast(\epsilon)}(v)&=\sum_u\indic{ I_{[0,T]}(v) \sim u }N_{t_\ast(\epsilon)}(v,I^{-1}_{[0,T]}(u))\\&=\sum_w\indic{ I_{[0,T]}(v) \sim I_{[0,T]}(w) }N_{t_\ast(\epsilon)}(v,w).
\end{align*}
Let $\tilde N_t(v,w)$ denote the amount of time particles from $v$ and $w$ spend adjacent during the time interval $[0,t]$. We claim that for each $w\in V$, and $0\le t<T$,
\[
\mathbb{E}\left[\indic{I_{[0,T]}(v) \sim I_{[0,T]}(w)}N_t(v,w)\right]=\sfrac2d\,\mathbb{E}\left[\indic{I_{[0,T]}(v) \sim I_{[0,T]}(w) }\tilde N_t(v,w)\right].
\]
To see this, notice that conditionally on the unordered pair of trajectories $\{I_{[0,t]}(v),I_{[0,t]}(w)\}$, the number of times  particles started from vertices $v$ and $w$ interact is Poisson with parameter $\frac2{d}\tilde N_t(v,w)$ (as these interactions do not affect the unordered pair of trajectories).
Therefore we have 
\begin{align*}
& \frac{d}{2} \mathbb{E}[\hat N_{t_\ast(\epsilon)}(v)]\\&= \sum_w\mathbb{E}\left[\indic{I_{[0,T]}(v) \sim I_{[0,T]}(w) }\tilde N_{t_\ast(\epsilon)}(v,w)\right]\\
&=\int_0^{t_\ast(\epsilon)}\sum_w\mathbb{E}\left[\indic{I_{[0,T]}(v) \sim I_{[0,T]}(w) }\indic{I_{[0,s]}(w)\sim I_{[0,s]}(v)}\right]ds\\
&=\int_0^{t_\ast(\epsilon)}\sum_w\sum_{a,b:\,a\sim b}\mathbb{E}\left[\indic{I_{[s,T]}(a)\sim  I_{[s,T]}(b)}\indic{I_{[0,s]}(w)=b,\, I_{[0,s]}(v)=a}\right]ds\\
&=\int_0^{t_\ast(\epsilon)}\sum_{a,b:\,a\sim b}\mathbb{E}\left[\indic{I_{[0,s]}(v)=a}\indic{I_{[s,T]}(a)\sim  I_{[s,T]}(b)} \right]ds\\
&=\int_0^{t_\ast(\epsilon)}\sum_{a,b:\,a\sim b}\Pr[I_{[0,s]}(v)=a]\,\Pr[I_{[s,T]}(a)\sim  I_{[s,T]}(b)]ds\\
&=\int_0^{t_\ast(\epsilon)}\sum_{a,b:\,a\sim b}\Pr[I_{[0,s]}(v)=a]\sum_{c,d:\,c\sim d}\Pr\left[I_{[s,T]}(a)=c,\, I_{[s,T]}(b)=d\right]ds\\
&\le \int_0^{t_\ast(\epsilon)}\sum_{a,b:\,a\sim b}\Pr[I_{[0,s]}(v)=a] \\&\phantom{\le \sfrac2{d}\int_0^{t_\ast(\epsilon)}}\cdot\sum_{c,d:\,c\sim d}\Pr\left[I_{[s,T]}(a)\in\{c,d\}\right]\,\Pr\left[I_{[s,T]}(b)\in\{c,d\}\right]ds,
\end{align*}
where the last line follows from the NA property. Now, since $T\ge t_\ast(\epsilon)+s_\ast(\epsilon)$, for each $0\le s\le t_\ast(\epsilon)$ we have that $T-s \ge s_\ast(\epsilon)$ and so 
\[
\Pr[I_{[s,T]}(b)\in\{c,d\}]\le \max_{b,c,d}\Pr[I_{[0,s_\ast(\epsilon)]}(b)\in\{c,d\}]\le \frac{2\epsilon}{t_\ast(\epsilon)}.
\]
We thus obtain
\begin{align*}
&\mathbb{E}[\hat N_{t_\ast(\epsilon)}(v)]\\&\le\ \frac{4\epsilon}{dt_\ast(\epsilon)}\int_0^{t_\ast(\epsilon)}\sum_{a,b:\,a\sim b}\Pr[I_{[0,s]}(v)=a]\sum_{c,d:\,c\sim  d}\Pr\left[I_{[s,T]}(a)\in\{c,d\}\right]ds\\
& \le \frac{8d\epsilon}{dt_\ast(\eps)}
\int_0^{t_\ast(\eps)}\sum_{a,b:\,a\sim b}\Pr[I_{[0,s]}(v)=a]\,ds\le 8d\epsilon .\qedhere
\end{align*} 
\end{proof}

Motivated by the discussion in \S\ref{s:overview}, for each $a,u,x,v\in V$ and $\epsilon \ge0$, we define
\begin{equation}
\label{e:Q(a)}
Q(a)=Q(a,u,x,v,\epsilon):=\Pr\left[I_{[0,T]}(a)=u,\,N_{t_\ast(\epsilon)}(a,x)=0\bigm| I_{[0,T]}(x)=v\right].\end{equation}

The next lemma gives the large-deviation bound (for any initial configuration of black and red particles) on the number of black particles which are time-$T$ neighbours with a red particle and which do not interact with that red particle during time interval $[0,t_\ast(\epsilon)]$.
The proof is similar to the proof of Lemma~\ref{L:BNR} in that it revolves around a Chernoff bound and the NA property.

\begin{lemma}[{Proof in Appendix~\ref{proofoflargedev}]}]\label{L:largedev}
Fix $\epsilon\in(0,10^{-4}]$ and let $Q(a)=Q(a,u,x,v,\epsilon)$ be as in \eqref{e:Q(a)}. There exists $n_0$ such that for all $n\ge n_0$ we have for all $2 \le k\le n/2$, all $u,x,v\in V$, and all $B\in (V)_{k-1}$,
\[
\sup_{s\ge t_{\mathrm{mix}}^{(\infty)}(n^{-10})} \Pr\left[\sum_{a\in V}\indic{a\in \BB_{s}}Q(a)>\frac{k}{n}+\frac1{16}\Bigm| \BB_0=B\right]\le n^{-13}.
\]
\end{lemma}
\begin{remark}
The large deviation bound on the black particle measure needed for the proofs of Theorems~\ref{thm:main3} (mixing for sublinear number of particles) and~\ref{thm:maindeg} (mixing for graphs of high degree) is Lemma~\ref{L:black}.\end{remark}

\section{Loss of red in a round: proof of Proposition~\ref{P:beta} for $d\ge 10^4$}\label{S:loss}
In this section we prove Proposition~\ref{P:beta} for $d\ge 10^4$. 
We begin with some new definitions. For each $a\in V$, let $\phi_a$ be the first time of the form $\tau_j\in(T,T+1)$ for which $a\in e_j$ (setting $\phi_a=\infty$ if no such time exists). If $\phi_a<\infty$, then define $F_a=I^{-1}_{(T,\phi_a)}(b)$ where $b$ is the other vertex on edge $e_j$; if instead $\phi_a=\infty$ then we write $F_a=\ast$. (This notation is similar to that appearing in \cite[Sec.\ 9.2]{olive}.)
%
Recall also the definition of an $(\alpha,t)$-good configuration from Definition~\ref{d:p-good}. We determine the kinds of configurations that are $(\alpha,T)$-good.

\begin{lemma}\label{L:redloss}
Suppose $d\ge10^4$. If $C_\mathrm{round}>C_{\ref{L:piNice}}(10^{-4})$  then any configuration $M=(B,R,\eset,W)$  of the chameleon process satisfying 
\[
\max_{b,v,a}\sum_{z\in \BB}Q(z,b,v,a,10^{-4})\le \frac{k}{n}+\frac1{16},
\]
is $(\alpha_1,T)$-good, for $T=C_\mathrm{round}(t_\mathrm{rel}+t_\ast(10^{-4})+s_\ast(10^{-4}))$ and some $\alpha_1>0$.
\end{lemma}
\begin{proof}
Recall the definition of $H_t$ from \S\ref{SS:cham_desc}.
Without loss of generality suppose $|R|\le|W|$. We bound $H_T$ by only counting pink particles created from red and white particles satisfying: the red particle is on some vertex $a$ at time $T$ and the white on some vertex $b$ with $a\sim b$, and $\phi_a=\phi_b<\infty$. Observe that we have
\begin{align*}
H_T&\ge\sum_{b\in I_{[0,T]}(W)}
\indic{\bigcup_{a\in I_{[0,T]}(R)}\{F_a=b,\,\phi_a=\phi_b\}}\\&=\sum_{b\in I_{[0,T]}(W)}\sum_{a\in I_{[0,T]}(R)}
\indic{F_a=b,\,\phi_a=\phi_b},
\end{align*}where the equality follows from the fact that the events $\{F_a=b,\,\phi_a=\phi_b\}$ are disjoint.
Recall the definitions of  $N(R)$ (as a subset of the Nice$(R)$ vertices) and $GN(R)$ (as the subset of $N(R)$ which are ``good'') from the discussion after Lemma~\ref{L:piNice}.
 Taking an expectation in the above inequality gives, for any $\theta\in(0,1)$ and $M=(B,R,W)$,
\begin{align}\notag
\mathbb{E}_M[H_T]&\ge\sum_{a,b:\,a\sim b}\Pr\left[a\in I_{[0,T]}(R),\,b\in I_{[0,T]}(W),\,F_a=b,\,\phi_a=\phi_b\right]\\\notag
&=\sum_{a,b:\,a\sim b}\Pr\left[a\in I_{[0,T]}(R),\,b\in I_{[0,T]}(W)\right]\Pr\left[F_a=b,\,\phi_a=\phi_b\right]\\\label{eq:expec_diff}
&\ge\sum_{a,b:\,a\sim b}\Pr\left[a\in GN(R)_\theta,\,b\in I_{[0,T]}(W)\right]\Pr\left[F_a=b,\,\phi_a=\phi_b\right],
\end{align}
where the second equality follows by independence of the edge-rings before and after time $T$.
Notice now that we have
\begin{align*}
\Pr\left[F_a=b,\,\phi_a=\phi_b\right]&=\Pr\left[F_a=b,\,\phi_a=\phi_b|\,F_a\neq\ast\right]\Pr[F_a\neq\ast]\\&=\sfrac1{2d-1}\Pr[F_a\neq\ast]\ge\sfrac1{4d},
\end{align*}
where the inequality follows from the fact that some edge incident to vertex $a$ will ring during time interval $(T,T+1)$ with probability $1-e^{-1}>1/2$. Plugging this into \eqref{eq:expec_diff} gives
\begin{align}
\mathbb{E}_M[H_T]&\ge \frac1{4d}\sum_{a,b:\,a\sim b}\Pr\left[a\in GN(R)_\theta,\,b\in I_{[0,T]}(W)\right].\label{eq:expec_diff2}
\end{align}

Instead of considering pairs of red and white particles, we consider pairs of red and red, and pairs of red and black. So we now decompose
\begin{align}\begin{split}
&\Pr\left[a\in GN(R)_\theta,\,b\in I_{[0,T]}(W)\right]\\&= \Pr[a\in GN(R)_\theta]\big(1-\Pr[b\in I_{[0,T]}(R)\mid a\in GN(R)_\theta]\big)\\&\phantom{=}-\Pr[a\in GN(R)_\theta,\,b\in I_{[0,T]}(B)].\label{eq:decomp1}\end{split}
\end{align}
Using Lemma~\ref{L:BNR} we have, for any $\theta\in(0,1)$ and $\lambda>0$, the bound \begin{align}
\Pr[a\in GN(R)_\theta]\ge\left(1-L(\lambda,\theta,d,|R|)\right)\Pr[a\in N(R)],\label{eq:gnr}\end{align}
where $L(\lambda,\theta,d,r):=\exp\left\{-\lambda\theta d+(e^\lambda-1)\left(\frac1{32}+\frac{r}{n}\right)d\right\}.$

We decompose the term $\Pr\left[a\in GN(R)_\theta,\,b\in I_{[0,T]}(B)\right]$ according to the starting location of particle at vertex $a$ at time $T$:
\begin{align}\notag
&\Pr[a\in GN(R)_\theta,\,b\in I_{[0,T]}(B))]\le \Pr\left[a\in N(R),\,b\in I_{[0,T]}(B)\right]\\\notag
&=\sum_{v\in R}\Pr\left[a\in \mathrm{Nice}(R),\,b\in I_{[0,T]}(B),\,a=I_{[0,T]}(v)\right]\\\label{eq:gnrB}
&=\sum_{v\in R}\indic{a\in \mathrm{Nice}(R)}\Pr\left[b\in I_{[0,T]}(B),\,a=I_{[0,T]}(v)\right]
\end{align}
where in the last line we have used the fact that being in Nice is a deterministic property.

Using the definition of $GN(R)_\theta$ for any $\theta\in(0,1)$ and $\lambda>0$ and combining equations \eqref{eq:expec_diff2}-\eqref{eq:gnrB} we obtain:
\begin{align}\begin{split}
\mathbb{E}_M[H_T]\ge &\sfrac1{4d}\sum_{a}\Pr[a\in N(R)](1-L(\lambda,\theta,d,|R|))(d-\theta d)\\
&-\sfrac1{4d}\sum_{a,b:\,a\sim b}\sum_{v\in R}\indic{a\in \mathrm{Nice}(R)}\Pr\left[b\in I_{[0,T]}(B),\,a=I_{[0,T]}(v)\right].\end{split}\label{eq:blah2}
\end{align}
 We now further decompose $\Pr\left[b\in I_{[0,T]}(B),\,a=I_{[0,T]}(v)\right]$ into two terms, depending on whether the trajectories of particles started from vertices $a$ and $b$ are adjacent, and use Markov's inequality to give
\begin{align}\label{eq:decompN}\begin{split}
&\Pr\left[b\in I_{[0,T]}(B),\,a=I_{[0,T]}(v)\right]
\\&\le \Pr\left[ N_{t_\ast(10^{-4})}(I^{-1}_{[0,T]}(a),I^{-1}_{[0,T]}(b))=0,\,b\in I_{[0,T]}(B),\,a=I_{[0,T]}(v)\right]\\&\phantom{\le}+\mathbb{E}\left[N_{t_\ast(10^{-4})}(I^{-1}_{[0,T]}(a),I^{-1}_{[0,T]}(b))\indic{a=I_{[0,T]}(v)}\right].
\end{split}
\end{align}
Combining equations \eqref{eq:blah2} and~\eqref{eq:decompN} we obtain, for any $\theta\in(0,1)$ and $\lambda>0$,
\begin{align}\begin{split}
&\mathbb{E}_M[H_T]\\&\ge\sfrac1{4}\sum_{a}\Pr[a\in N(R)]\left(1-L(\lambda,\theta,d,|R|)\right)(1-\theta)\\
&\phantom{\le}-\sfrac1{4d}\sum_{a,b:\,a\sim b}\sum_{v\in R}\indic{a\in \mathrm{Nice}(R)}\Pr\Big[ N_{t_\ast(10^{-4})}(I^{-1}_{[0,T]}(a),I^{-1}_{[0,T]}(b))=0,\\
&\phantom{\le-\sfrac1{4d}\sum_{a,b:\,a\sim b}\sum_{v\in R}\indic{a\in \mathrm{Nice}(R)}\Pr\Big[\,\,}b\in I_{[0,T]}(B),\,a=I_{[0,T]}(v)\Big]\\
&\phantom{\le}-\sfrac1{4d}\sum_{a,b:\,a\sim b}\sum_{v\in R}\indic{a\in\mathrm{Nice}(R)}\mathbb{E}\left[N_{t_\ast(10^{-4})}(I^{-1}_{[0,T]}(a),I^{-1}_{[0,T]}(b))\indic{a=I_{[0,T]}(v)}\right].\end{split}\label{eq:expec_diff3}
\end{align}
For the second term on the r.h.s.\ we have,
\begin{align}\notag
&\sfrac1{4d}\sum_{a,b:\,a\sim b}\sum_{v\in R}\indic{a\in \mathrm{Nice}(R)}\\&\phantom{\sfrac1{4d}\sum_{a,b:\,a\sim b}}\cdot\Pr\left[ N_{t_\ast(10^{-4})}(I^{-1}_{[0,T]}(a),I^{-1}_{[0,T]}(b))=0,\,b\in I_{[0,T]}(B),\,a=I_{[0,T]}(v)\right]\notag\\
&=\sfrac1{4d}\sum_{a,b:\,a\sim b}\sum_{v\in R}\indic{a\in \mathrm{Nice}(R)}\notag\\&\phantom{\sfrac1{4d}\sum_{a,b:\,a\sim b}}\cdot\sum_{z\in B}\Pr\left[ N_{t_\ast(10^{-4})}(v,z)=0,\,b=I_{[0,T]}(z),\,a=I_{[0,T]}(v)\right]\notag\\
&=\sfrac1{4d}\sum_{a,b:\,a\sim b}\sum_{v\in R}\indic{a\in \mathrm{Nice}(R)}\Pr[a=I_{[0,T]}(v)]\sum_{z\in B}Q(z,b,v,a,10^{-4})\notag\\
&\le\left(\sfrac{k}{n}+\sfrac1{16}\right)\cdot\sfrac1{4d}\sum_{a,b:\,a\sim b}\sum_{v\in R}\indic{a\in \mathrm{Nice}(R)}\Pr[a=I_{[0,T]}(v)]\notag\\
&=\left(\sfrac{k}{n}+\sfrac1{16}\right)\cdot\sfrac1{4}\sum_a\Pr[a\in N(R)],\label{eq:expec_diff3_term2}
\end{align}
where the inequality follows from the assumption on the configuration $M$.

The third term on the r.h.s\ of \eqref{eq:expec_diff3} is
\begin{align}
&\sfrac1{4d}\sum_{a,b:\,a\sim b}\sum_{v\in R}\indic{a\in\mathrm{Nice}(R)}\mathbb{E}\left[N_{t_\ast(10^{-4})}(I^{-1}_{[0,T]}(a),I^{-1}_{[0,T]}(b))\indic{a=I_{[0,T]}(v)}\right]\notag\\
&\le \sfrac1{4d}\sum_{v\in R}\mathbb{E}\left[\sum_a\indic{a=I_{[0,T]}(v)}\sum_{b:\,b\sim I_{[0,T]}(v)}N_{t_\ast(10^{-4})}\big(v, I_{[0,T]}^{-1}(b)\big)\right]\notag\end{align}\begin{align}
&=\sfrac1{4d}\sum_{v\in R}\mathbb{E}\left[\sum_{b:\,b\sim I_{[0,T]}(v)}N_{t_\ast(10^{-4})}\big(v, I_{[0,T]}^{-1}(b)\big)\right]\notag\\
&=\sfrac1{4d}\sum_{v\in R}\mathbb{E}\big[\hat N_{t_\ast(10^{-4})}(v)\big]\le \sum_{v\in R}2\times10^{-4} =2\times 10^{-4}|R| ,\label{eq:expec_diff3_term3}
\end{align}
where the second inequality follows from Lemma~\ref{L:interact}. Plugging equations~\eqref{eq:expec_diff3_term2} and~\eqref{eq:expec_diff3_term3} into~\eqref{eq:expec_diff3} gives, for any $\theta\in(0,1)$ and $\lambda>0$,
\begin{align}
\begin{split}
&\mathbb{E}_M[H_T]\\&\ge\sfrac14\sum_a\Pr[a\in N(R)]\bigg\{\left(1-L(\lambda,\theta,d,|R|)\right)(1-\theta)-\sfrac{k}{n}-\sfrac1{16}\bigg\}-{2\times10^{-4}}|R|.
\end{split}\label{eq:expec_diff4}
\end{align}
Choosing $\lambda=0.05$, $\theta=\frac{9}{16}-\frac{k}{2n}$ and using the bound $|R|/n\le \frac12-\frac{k}{2n}$, \textcolor{black}{we have that $-\lambda\theta+(e^\lambda-1)(\sfrac{1}{32}+\sfrac{|R|}{n}) \le -\lambda(\frac{9}{16}-\frac{1}{32}-\frac 12)+(\frac{\lambda^2}{2}+\frac{\lambda^3}{6}+\lambda^4 )(\frac{1}{32}+\frac 12)   $, where we have used $\la \le e^\lambda-1\le \frac{\lambda^2}{2}+\frac{\lambda^3}{6}+\lambda^4  $ for $\la \in [0,0.05]$, and so}
\[
\sfrac1{d}\log L(\lambda,\theta,d,|R|)=-\lambda\theta+(e^\lambda-1)(1/32+|R|/n)\le -0.0008,
\]
and so since $d\ge10^4$, we obtain
\[
\left(1-L(\lambda,\theta,d,|R|)\right)(1-\theta)-\sfrac{k}{n}-\sfrac1{16} \ge -\sfrac{k}{n}(1-\sfrac{1-e^{-8}}{2}) +\sfrac{6-7 e^{-8} }{16} > \sfrac1{16}.
\]
Plugging this into \eqref{eq:expec_diff4} gives the bound
\begin{align}\label{eq:expec_diff5}
\mathbb{E}_M[H_T]\ge \sfrac1{64}\mathbb{E}[|N(R)|]-{2\times 10^{-4}}|R|.
\end{align}

 Notice now that $\mathbb{E}[|N(R)|]=|R|\,\Pr_{\pi_R}(X_T\in \mathrm{Nice}(R))$, for $(X_t)$ a realisation of RW$(G)$, and so by Lemmas~\ref{L:niceC} and~\ref{L:piNice} we have that, since $C_\mathrm{round}>C_{\ref{L:piNice}}(10^{-4})$, 
$\mathbb{E}[|N(R)|]\ge |R|\left(\sfrac1{17}-10^{-4}\right)$.
Hence from~\eqref{eq:expec_diff5} we obtain the bound
\[
\mathbb{E}_M[H_T]\ge |R|\left(\sfrac1{1088}-\sfrac{129}{64}\times10^{-4}\right)>0.0007|R|.
\] 
 The proof is completed by taking any $\alpha_1\le0.0007$. \qedhere
\end{proof}

\begin{proof}[Proof of Proposition~\ref{P:beta} for $d\ge 10^4$]\ \\
 Recall the notation $t_0=t_\mathrm{mix}^{(\infty)}(n^{-10})$ and let $t\ge t_0$.   By  Lemma~\ref{L:largedev} we have that for any $B\in (V)_{k-1}$,  and $n$ sufficiently large, by a union bound
\begin{align*}
&\Pr\left[\max_{b,v,a}\sum_{z\in \BB_{t}}Q(z,b,v,a,10^{-4})\le \sfrac{k}{n}+\sfrac1{16}\Bigm| \BB_0=B\right]\\&\ge 1-\sum_{b,v,a}\Pr\left[\sum_{z\in B_{t}}Q(z,b,v,a,10^{-4})>\sfrac{k}{n}+\sfrac1{16}\Bigm| \BB_0=B\right]\ge 1-n^{-10}.
\end{align*}Therefore if we have $C_\mathrm{round}>C_{\ref{L:piNice}}(10^{-4})$ then, by Lemma~\ref{L:redloss}, since $d\ge 10^4$, with probability at least $1-n^{-10}$, $M_{t}$ (the configuration of the chameleon process at time $t$) is $(\alpha_1,T)$-good, for $T=C_\mathrm{round}(t_\mathrm{rel}+t_\ast(10^{-4})+s_\ast(10^{-4}))$ and some $\alpha_1>0$, i.e.\ $\beta(\alpha_1,t_\mathrm{round}-1)\le n^{-10}$.

This completes the proof taking $\alpha=\alpha_1$.
\end{proof}

\section{Modifications to the main approach}\label{S:modifications}
\subsection{Generalising the results of \S\ref{S:neighbours}}\label{S:generalising}

For the case of $d<10^4$ for Theorem~\ref{thm:main1} (and for other values of $d$ in general) it will be useful to artificially inflate the degree of vertices by adding \textcolor{black}{``dummy"} directed edges (of zero weight) to the graph (without the addition of new vertices). The number of edges we need to add varies according to the values of $k$ and $d$ and we let $\hat d$ denote the new out-degree of all of the vertices (which is the number of undirected edges plus the number of directed out-edges from a vertex). We will always add these edges between vertices within graph distance at most $\hat d$ in the original graph. These edges are assigned weight 0 and so never ring and play no role in the dynamics of the processes, instead just affecting the structure of the graph (in particular adjacency).

Any such graph that has these additional edges is referred to as  a \emph{modified} graph and we write $v\stackrel{\rightarrow}{\sim}u$ to indicate that either $(v,u)$ or $\{v,u\}$ is an edge in a modified graph.
We denote the maximal in-degree in the modified graph by $d_{\max}^{\mathrm{in}}$. 

 \textcolor{black}{If we could ensure that the in-degrees were all equal (to the out-degree), the argument from \S\ref{S:neighbours}-\ref{S:loss} would work almost verbatim (the main difference is that one has to replace adjacency by adjacency in the modified graph, apart from when controlling the number of interactions between a pair of particles). If the graph is vertex-transitive one can easily ensure this. Alas, in general one  cannot do this.} \textcolor{black}{As we inflate the degree only when $d<10^4$ it follows that there exists an absolute constant $D$ such that $ d_\mathrm{max}^{\mathrm{in}} \le D \hat d $. We now explain how, from a high-level perspective, this leads only to minor changes in the outline of the argument. While we expect most readers to be satisfied with this outline, we present all details of the proof below.  }
 
\textcolor{black}{   Recall that when $d \ge 10^4$ the (indirect) argument for controlling the expected number of white neighbours that red particles have by considering the number of red and black neighbours they have   breaks down when $|R| \wedge |W| > \varrho_{0} n $ for a certain constant $\varrho_{0}$. Fortunately, in this regime we could  work directly with the white particles.    In the case  $d<10^4$ the indirect argument  breaks down for a smaller value of $\varrho_{0}$, due to an additional factor $D \ge \frac{d_\mathrm{max}^{\mathrm{in}}}{\hat d}  $ appearing in analogous statements to  ones from \S\ref{S:neighbours}-\ref{S:loss} related to the analysis of the number of red neighbours. This is not a problem, as when  $|R| \wedge |W| > \varrho_{0} n $ (and  $d<10^4$) we argue that at the end of the constant colour phase, the expected number of red particles with at least one  white neighbour in the modified graph is of order $n$ (provided  $T \gtrsim \rel $, where $T+1$ is the duration of the round). As the modified graph has bounded degree, this suffices to argue that $\E[H_T] \gtrsim n $, as required. }

Recall the definition of Nice from \S\ref{S:redvred}. We modify this definition to deal with these modified graphs and in the sequel this is the definition of Nice that we use (i.e.\ every future use of Nice refers to this new defintion).

\begin{definition}
For each subset $S\subseteq V$, let $e_T(v,S):=\sum_{u:\,v \stackrel{\rightarrow}{\sim} u}P_T(u,S) $ and define $\mathrm{Nice}(S)$ as:
\[
\mathrm{Nice}(S):=\Big\{v\in V:\,e_T(v,S)< \hat d\left(\sfrac1{32}+\sfrac{|S|}{n}\right)\Big\}.
\]
\end{definition}

The equivalent statement of Lemma~\ref{L:niceC} is the following. We omit the proof as it follows in a similar manner (i.e. a simple counting argument together with the fact that the modified graph has out-degree $\hat d $ at each site).

\begin{lemma}[Proof omitted]\label{L:niceC2}
For each $S\subseteq V$, \[
|\mathrm{Nice}(S)^\complement|\le \left(\frac1{32}+\frac{|S|}{n}\right)^{-1}\frac{d_\mathrm{max}^\mathrm{in}}{\hat d}|S|.
\]
\end{lemma}
%
%
%
\begin{remark}
Using Lemma~\ref{L:niceC2} it is possible to prove (proof omitted) that Lemma~\ref{L:piNice} still holds with the new definition of Nice and so will make use of this lemma in this section. 
\end{remark}

Next, we recall the definition of $BN(S)$ for $S\subseteq V$ from \S\ref{S:redvred} and redefine it for modified graphs as:
\[
BN(S)_\theta:=\Big\{v\in N(S):\,\sum_{u:\,v \stackrel{\rightarrow}{\sim} u}\indic{I_{[0,T]}^{-1}(u)\in S}>\theta \hat d \Big\}.
\]
We also redefine set $GN(S)_\theta$ to be $N(S)\setminus BN(S)_\theta$. The analogue of Lemma~\ref{L:BNR} is the following.

\begin{lemma}\label{L:BNR2}
For each $S\subseteq V$, $\theta\in(0,1)$, $\lambda>0$ and $v\in V$, 
\[
\Pr[v\in BN(S)_\theta\mid v\in N(S)]<\exp\left\{\hat d\left(-\lambda\theta+(e^\lambda-1)\left(\sfrac1{32}+\sfrac{|S|}{n}\right)\right)\right\}.
\]
\end{lemma}

Next, we recall Lemma~\ref{L:interact} which bounds the expected value of $\hat N_t(v)$. We redefine this quantity in terms of modified graphs as follows:
\[
\hat N_{t}(v):=\sum_{u:\, I_{[0,T]}(v) \stackrel{\rightarrow}{\sim} u}N_{t}(v,I^{-1}_{[0,T]}(u)).
\]
We present below the version of Lemma~\ref{L:interact} to be used for modified graphs. The proof is omitted and we remark that the main difference in the proof of this result compared with Lemma~\ref{L:interact} is that we define $\tilde N_t(v,w)$ to be the amount of time particles from $v$ and $w$ spend adjacent w.r.t.\ $G$ (as opposed to w.r.t.\ the modified graph) during the time interval $[0,t]$. 

\begin{lemma}[Proof omitted]\label{L:interact2} For all $\epsilon\in(0,1)$ and $C_\mathrm{round}\ge 1$   we have
\[\max_{v\in V}\mathbb{E}[\hat N_{t_\ast(\epsilon)}(v)]\le 4\epsilon(d_\mathrm{max}^\mathrm{in}+\hat d).\]
\end{lemma}

We now show that after a burn-in period we have a large deviation estimate of the black particle measure.  
After a burn-in period, the occupation by the black particle measure has marginals extremely close to $k/n$ and has the NA property. A simple calculation involving the Laplace transform (Lemma~\ref{L:black}) shows that it satisfies large deviation estimates similar to the ones available in the independent case. From this, along with a union bound, one can derive (Corollary~\ref{C:black}) that at each given time after a burn in period, the probability of having a configuration satisfying that given this current configuration, the probability of having more than $(\frac{k}{n}+c)d $ black neighbours of a vertex after $T$ additional time units (where $T+1$ is the duration of a round) is $\ll n^{-10}$ (\emph{i.e.}, if we start a round at this time, the probability that at the end of the constant-colour phase we have at least $(\frac{k}{n}+c)d $ black neighbours is small). 

The proof is similar to the proof of Lemma~\ref{L:BNR}.  
For $\eps\in(0,1)$, $n\in\mathbb{N}$ and $2\le k\le n/2$, we denote $m_{\eps,n,k}:=\max\Big\{\log\frac{\eps n}{e^2k},\frac{\eps n}{2k}\big(\frac12-\frac{\eps n}{k}\big)\Big\}.$
\begin{lemma}[{Proof in Appendix~\ref{proofblack}]}] \label{L:black}
Fix $\varepsilon\in(0,1)$. There exists $n_0=n_0(\eps)$ such that for all $n\ge n_0$, $2\le k\le n/2$, $B\in(V)_{k-1}$, $v\in V$, and $s\ge t_{\mathrm{mix}}^{(\infty)}(n^{-10})$,
\begin{align*}
 \Pr\Big[\sum_{u:\, v  \stackrel{\rightarrow}{\sim} u }\indic{u\in \BB_{s}}\ge\left(\sfrac{k}{n}+\varepsilon\right)\hat d \Bigm| \BB_0=B\Big]\le \exp\Big(-\hat d\eps m_{\eps,n,k}\Big).
\end{align*}
\end{lemma}
\begin{corollary}\label{C:black}
Fix $\varepsilon\in(0,1)$ and for each $t>0$ let $\mathcal{F}_t$ denote the  $\sigma$-algebra generated  by $\BB_t$. There exists $n_0=n_0(\eps)$ such that for all $n\ge n_0$, $2\le k\le n/2$, $B\in(V)_{k-1}$, $v\in V$ and $s_2\ge s_1\ge  t_{\mathrm{mix}}^{(\infty)}(n^{-10})$,
\begin{align*}
 &\Pr\left[\Pr\Big[\sum_{u:\,v  \stackrel{\rightarrow}{\sim} u}\indic{u\in \BB_{s_2}}\ge\left(\sfrac{k}{n}+\varepsilon\right) \hat d\Bigm| \mathcal{F}_{s_1}\Big]\ge \exp(-\sfrac12 \hat d\eps m_{\eps,n,k})\right]\\&\le \exp(-\sfrac12 \hat d\eps m_{\eps,n,k}).
\end{align*}
\end{corollary}
\begin{proof}
The proof immediately follows from Lemma~\ref{L:black} using Markov's inequality.
\end{proof}
\begin{remark}\label{R:black}
The above corollary also holds for sufficiently small $\hat c\in (0,1)$ taking $s_2\ge s_1\ge t_\mathrm{mix}^{(\infty)}(\hat c/k)$. This follows from the fact that Lemma~\ref{L:black} holds when $n^{-10}$ is replaced with a sufficiently small $\hat c$  and so in particular holds when replaced with $\hat c/k$. 
\end{remark}

For graphs with sufficiently small degree, once the number of red particles is at least some fraction of $n$, it turns out that we can avoid analysing the number of red and black neighbours of a vertex (conditioned on being red), and instead directly lower-bound the number of white neighbours (in fact in this case we do not even need burn-in periods). 
To see why, observe that the number of red particles without a nearby white particle after the relaxation phase is comparable (as $|R|\asymp n$) to the number of vertices without a nearby white particle at this time. This can be controlled with a simple argument making use of the Poincar\'e inequality, see Lemma~\ref{L:white1}. For the remaining red vertices in the proximity of a white particle, we can easily lower-bound the probability of their interaction during a unit time interval.
%

For a subset $S\subseteq V$, we define another subset $Q\subseteq V$ in the following way:
\[
Q(S)=\big\{v\in V:\,\sum_{u:\,v\stackrel{\rightarrow}{\sim} u}P_T(u,S)<\hat d/{16}\big\}.
\]
The reader should think of $S$ as the set occupied by the white particles at the beginning of a round. Recall that w.l.o.g.\! we always consider in \S\ref{S:neighbours}--\S\ref{S:modifications} the case that there are as many white particles as there are red, and so $|S|/n \ge 1/4$.

 We achieve control on the number of white neighbours via the following lemma. 

\begin{lemma}\label{L:white1}
For any $\eps\in(0,1)$ and any $S\subset V$ with $|S|/n\ge 1/4$, if $T\ge \rel |\log(1/\eps)|$ then $|Q(S)|\le 8\eps n\sfrac{d_\mathrm{max}^\mathrm{in}}{\hat d}$.
\end{lemma}
\begin{proof}
If $T\ge\rel |\log(1/\eps)|$ then since $|S|/n\ge 1/4$, the $L_2$-distance of $P_{\pi_S}(X_T\in\bullet)$ from $\pi$ is at most $2\eps$ by the Poincar\'e inequality \eqref{e:Poincare}, and hence this is also a bound on the $L_1$-distance. Therefore by a simple counting argument and reversibility
\begin{align}
|\{u:\,P_T(u,S)<|S|/(2n)\}|< 4\eps n.\label{eq:PTS}
\end{align}
We prove the statement of the lemma by contradiction. So suppose $|Q(S)|>8\eps n\sfrac{d_\mathrm{max}^\mathrm{in}}{\hat d}$, \emph{i.e.}\ there are more than $8\eps n\sfrac{d_\mathrm{max}^\mathrm{in}}{\hat d}$ vertices $v$ for which we have $\sum_{u:\,v\stackrel{\rightarrow}{\sim} u}P_T(u,S)< \hat d/16\le \hat d|S|/(4n)$. Then for each $v\in Q(S)$, we must have at least $\hat d/2$ vertices $u$ such that $v\stackrel{\rightarrow}{\sim} u$ with $P_T(u,S)<|S|/(2n)$. Now each $u$ has in-degree at most $d_\mathrm{max}^\mathrm{in}$, and thus overall there are at least $\hat d|Q(S)|/(2d_\mathrm{max}^\mathrm{in})$ vertices $u\in V$ with $P_T(u,S)<|S|/(2n)$, but since we assume $|Q(S)|>8\eps n\frac{d_\mathrm{max}^\mathrm{in}}{\hat d}$, this number of vertices is at least $4\eps n$. This is in contradiction with \eqref{eq:PTS}.
\end{proof}

\begin{lemma}\label{L:white2}
Let $S\subset V$. For each $v\in Q(S)^\complement$, \[\Pr\big[\sum_{u:\,v\stackrel{\rightarrow}{\sim} u}\indic{u\in S_T}=0\big]\le \left(\sfrac{31}{32}\right)^{\hat d/32}.\]
\end{lemma}
\begin{remark}
The proof of this lemma relies on the NA property.
\end{remark}
\begin{proof}
Notice that, since $v\in Q(S)^\complement$, we must have at least $\hat d/32$ vertices $u$ with $v\stackrel{\rightarrow}{\sim} u$ such that $\Pr[u\in S_T]\ge 1/32$. Hence by the NA property, 
\[\Pr\big[\sum_{u:\,v\stackrel{\rightarrow}{\sim} u}\indic{u\in S_T}=0\big]\le \prod_{u:\,v\stackrel{\rightarrow}{\sim} u}\Pr\big[u\notin S_T\big]\le \left(\sfrac{31}{32}\right)^{\hat d/32}.\qedhere
\]
\end{proof}

\subsection{Proof of Proposition~\ref{P:beta} for $d<10^4$}\label{SS:smalld}
In order to prove Proposition~\ref{P:beta} for $d<10^4$, we split into two cases depending on the value of $|R|\wedge|W|$ (the minimum of the number of reds and whites in the initial configuration of the chameleon process) and a constant $\varrho_0\in(0,1/4)$ to be later determined. The following lemma will be used for the case $|R|\wedge|W|\ge\varrho_0 n$.

\begin{lemma}\label{L:losssmalldbigR1}
Let $\varrho\in(0,1/4),\,C_\ast\ge1$ and consider case $d<C_\ast\log(1/\varrho)$. 
There exists a constant $C^0_\ast$ such that if $C_\ast\ge C^0_\ast$  then any configuration $M=(B,R,\eset,W)$ of the chameleon process with $|R|\wedge |W|\ge\varrho n$ 
is $(\alpha_4,T)$-good for $T\ge C_\mathrm{round}\rel$ with $C_\mathrm{round}$ and $\alpha_4$ depending only on $\varrho$ and $C_\ast$.
\end{lemma}
\begin{proof}
We inflate the degree so that $\hat d=\lceil C_\ast\log(1/\varrho)\rceil$.
Without loss of generality suppose $|R|\le |W|$.  Notice that since $k\le n/2$, we have that $|W|/n\ge 1/4$.

Notice that a white particle will get pinkened during $(T,T+1)$ if there exists a red particle such that: \begin{enumerate}
 \item the red particle is on some vertex $a$ at time $T$ with $a$ belonging to a sparse set $A$, and the white on some vertex $b$, with $a\stackrel{\rightarrow}{\sim}b$,
 \item  $\phi_a<\infty$ (\emph{i.e.}\ vertex $a$ is on a ringing edge during time interval $(T,T+1)$),
 \item at time $\phi_a$ the other vertex $a'$ incident to the ringing edge is occupied by the white particle (which may have turned pink by this time),
 \item during time interval $[T,\phi_a)$ the white particle moves along a shortest trajectory from $b$ to $a'$.
\end{enumerate} We remark that this will only result in pink particles being created \emph{at time $\phi_a$} if the white particle is in fact still white at time ${\phi_{a}}_-$ (and otherwise it gets pinkened prior to this time). We choose the set $A$ to have minimal size while satisfying $\sum_{a\in A}\Pr[a\in I_{[0,T]}(R)]\ge \hat d^{-2\hat d}|R|$ and with the property that no two elements of $A$ are within graph distance (in the original graph) $2\hat d$. It can be shown (\emph{e.g.}\! with a greedy construction) that $|A|\le \hat d^{-2\hat d}n$.

Observe that we can bound
\begin{align*}
H_T&\ge \sum_{b\in I_{[0,T]}(W)}
\indic{\bigcup_{a\in I_{[0,T]}(R)\cap A}\{F_a=b,\,a\stackrel{\rightarrow}{\sim}b\}}\\&=\sum_{b\in I_{[0,T]}(W)}\sum_{a\in I_{[0,T]}(R)\cap A}
\indic{F_a=b,\,a\stackrel{\rightarrow}{\sim}b},
\end{align*}where the equality follows from the fact that each $b\in V$ is adjacent to at most one $a\in A$. Taking an expectation gives
\begin{align*}
\E_M[H_T]&\ge \sum_{a\in A}\sum_{b:\,a\stackrel{\rightarrow}{\sim} b}\Pr[a\in I_{[0,T]}(R),\,b\in I_{[0,T]}(W),\,F_a=b]\\
&=\sum_{a\in A}\sum_{b:\,a\stackrel{\rightarrow}{\sim} b}\Pr[a\in I_{[0,T]}(R),\,b\in I_{[0,T]}(W)]\Pr[F_a=b],
\end{align*}
where the equality follows by independence of the edge-rings before and after time $T$.

To lower-bound the probability $\Pr[F_a=b]$ we fix a particular trajectory the white particle must follow, from its position at time $T$ (vertex $b$) to a vertex (denoted $a'$) adjacent to $a$. The trajectory chosen is one of shortest length between $a$ and $b$. We additionally impose the condition that the particle must follow this trajectory during time interval $[T,T+1/2]$. Since the degree of each vertex is less than $\hat d$, and vertex $b$ is within graph distance (in the original graph) $\hat d$ from $a$, this event has probability bounded from below by some constant $c_1>0$ (uniformly over $a$ and $b$). The event $\{F_a=b\}$ will then be satisfied if the first edge incident to vertex $a$ to ring during $(T,T+1)$ is edge $\{a,a'\}$ and this edge first rings during time interval $(T+1/2,T+1]$, an event of probability $c_2>0$. Hence we obtain the bound $\Pr[F_a=b]\ge c_1c_2$. Note that these constants depend on $\varrho$ since $\hat d$ depends on $\varrho$.

Hence we have
\begin{align*}
\E_M[H_T]&\ge
c_1c_2\sum_{a\in A}\sum_{b:\,a\stackrel{\rightarrow}{\sim} b}\Pr[a\in I_{[0,T]}(R),\,b\in I_{[0,T]}(W)]\\
&\ge c_1c_2\sum_{a\in A} \Pr[a\in I_{[0,T]}(R),\,\exists b\in I_{[0,T]}(W):\,a\stackrel{\rightarrow}{\sim} b].
\end{align*} Recall the definition of $Q(S)$ from \S\ref{S:generalising}. Decomposing the above sum (and writing $a\stackrel{\rightarrow}{\nsim}b$ to indicate that it is not the case that $a\stackrel{\rightarrow}{\sim}b$) we have
\begin{align*}
\E_M[H_T]&\ge c_1c_2\sum_{a\in A}\Pr[a\in I_{[0,T]}(R)]\\&\phantom{\ge}-c_1c_2\sum_{a\in Q(W)}\Pr[a\in I_{[0,T]}(R),\,\forall b\in I_{[0,T]}(W),\,a\stackrel{\rightarrow}{\nsim} b]\\&\phantom{\ge}-c_1c_2\sum_{a\in A\cap Q(W)^\complement}\Pr[a\in I_{[0,T]}(R),\,\forall b\in I_{[0,T]}(W), \,a\stackrel{\rightarrow}{\nsim} b]\\
&\ge \frac{c_1c_2}{\hat d^{2\hat d}}|R|-c_1c_2|Q(W)|-c_1c_2\sum_{a\in A\cap Q(W)^\complement}\Pr[\forall b\in I_{[0,T]}(W),\,a\stackrel{\rightarrow}{\nsim} b].
\end{align*}
By Lemma~\ref{L:white1} with $\eps=(\hat d\varrho)/(32d_\mathrm{max}^\mathrm{in}\hat d^{2\hat d})$, since $|W|/n\ge 1/4$, if $C_\mathrm{round}>\log(1/\eps)$ then $|Q(W)|\le\varrho n/(4\hat d^{2\hat d})$.  Notice that for a fixed choice of $\hat d$, there exists a universal (over $G$) constant $D$ such that $d_{\mathrm{max}}^\mathrm{in}\le D\hat d$, and hence $C_\mathrm{round}$ depends only on $\varrho$ and the choice of $C_\ast$. By Lemma~\ref{L:white2} if we take $C_\ast^0=2500$ then since $\varrho<1/4$ we have that for each $a\in Q(W)^\complement$, $\Pr[\forall b\in I_{[0,T]}(W),\,a\stackrel{\rightarrow}{\nsim} b]\le \varrho/4$. Hence we obtain
\[
\E_M[H_t]\ge \frac{c_1c_2}{\hat d^{2\hat d}}|R|-\frac{c_1c_2}{\hat d^{2\hat d}}\frac{\varrho n}{4}-c_1c_2 |A|\frac{\varrho}{4}\ge \frac{c_1c_2}{\hat d^{2\hat d}}\frac{|R|}{2},
\]
which completes the proof with $\alpha_4=\sfrac12 c_1c_2\hat d^{-2\hat d}$.
\end{proof}

Now we consider how to deal with the case $|R|\wedge|W|<\varrho_0 n$. Recall Lemma~\ref{L:redloss} from \S\ref{S:loss}. In order to prove the equivalent statement for the case $d<10^4$ we follow a similar argument but also make use of degree-inflation. We first state a preliminary lemma which states that we can find a sparse subset of Nice$(S)$ which picks-up a fraction of the time-$T$ mass of a random walk started uniformly on $S$. 
 
\begin{lemma}[{Proof in Appendix~\ref{A:prelim}]}]\label{L:boundA}
Suppose $d<10^4$. For any $S\subseteq V$, there exists a constant $c_\mathrm{frac}>0$ and a subset  $A(S)$ of Nice$(S)$ such that no two members of $A(S)$ are within graph distance of $2\times 10^4$ and such that 
\[
\sum_{u\in A(S)}\Pr_{\pi_S}[X_T=u]\ge c_\mathrm{frac}\sum_{u\in \mathrm{Nice}(S)}\Pr_{\pi_S}[X_T=u].
\]
\end{lemma}
 
Notice that, due to the sparseness property of $A(S)$, in the modified graph if $v,w\in A$ and $v\stackrel{\rightarrow}{\sim} u$, then $w\stackrel{\rightarrow}{\nsim} u$.

\begin{lemma}\label{L:losssmalld}
Let $c_\mathrm{frac}$ be the constant from Lemma~\ref{L:boundA} and consider the case $d<10^4$.
There exists $\varrho_0\in(0,1/4)$ and $\epsilon\in(0,10^{-4}]$ such that if $C_\mathrm{round}>C_{\ref{L:piNice}}(\epsilon)$  then any configuration $M=(B,R,\eset,W)$ of the chameleon process with $|R|\wedge |W|<\varrho_0 n$ satisfying 
\[
\max_{b,v,a}\sum_{z\in \BB}Q(z,b,v,a,\epsilon)\le \sfrac{k}{n}+\sfrac1{16},
\]
is $(\alpha_2,T)$-good, for some universal $\alpha_2>0$, and $T=C_\mathrm{round}(t_\mathrm{rel}+t_\ast(\epsilon)+s_\ast(\epsilon))$.
\end{lemma}
\begin{proof}
We inflate the degree so that $\hat d=10^4.$ Without loss of generality suppose $|R|\le |W|$. 

Notice that a white particle will get pinkened during $(T,T+1)$ if there exists a red particle satisfying statements 1.\ to 4.\ from the proof of Lemma~\ref{L:losssmalldbigR1}. 
We choose the set $A$ to be $A(R)$ from Lemma~\ref{L:boundA}.

The first part of the proof proceeds similarly to the proof of Lemma~\ref{L:losssmalldbigR1}. We obtain the bound:
\begin{align*}
\mathbb{E}_M[H_T]&\ge c_1c_2\sum_{a\in A(R)}\,\sum_{b:\,a\stackrel{\rightarrow}{\sim}b}\Pr[a\in GN(R)_\theta,\,b\in I_{[0,T]}(W)].
\end{align*}
At this point we refer to the proof of Lemma~\ref{L:redloss}, and following the same arguments (using Lemma~\ref{L:interact2} in place of Lemma~\ref{L:interact}) arrive at the analogous statement to \eqref{eq:expec_diff5}:
\begin{align}\label{eq:expdiff2}
\mathbb{E}_M[H_T]&\ge \hat c_3\left(\sfrac1{64}\mathbb{E}[|A(R)\cap I_{[0,T]}(R)|]-\epsilon|R|\right),
\end{align}
for some $\hat c_3>0$. Notice that in applying Lemma~\ref{L:interact2} to obtain the above we have made use of the fact that for a fixed choice of $\hat d$ there exists a universal constant $D$ such that $d_\mathrm{max}^\mathrm{in}\le D\hat d$ \textcolor{black}{(i.e.\ we take $\epsilon$ in Lemma~\ref{L:interact2} to be $\epsilon/D$)}.
Now notice that by Lemmas~\ref{L:niceC}, \ref{L:piNice} \textcolor{black}{(with $\varepsilon=10^{-4}$)}, \ref{L:niceC2} \textcolor{black}{(used to argue that $\pi( \text{Nice}(R)) \geq 1- 32 D \varrho_0 $)} and~\ref{L:boundA} we have
\begin{align*}
&\mathbb{E}[|A(R)\cap I_{[0,T]}(R)|]=|R|\,\Pr_{\pi_R}[X_T\in A(R)]=|R|\sum_{u\in A(R)}\Pr_{\pi_R}[X_T=u]\\&\ge c_\mathrm{frac}|R|\sum_{u\in \mathrm{Nice}(R)}\Pr_{\pi_R}[X_T=u]=c_\mathrm{frac}\mathbb{E}[|N(R)|]\\&\ge c_\mathrm{frac}|R|\left(\pi(\mathrm{Nice}(R))-10^{-4}\right)>c_\mathrm{frac}|R|(1-32D\varrho_0-10^{-4}).
\end{align*} Combining this with \eqref{eq:expdiff2} and taking $\epsilon$ and $\varrho_0$ sufficiently small  gives the existence of a universal constant $\alpha_2$ such that 
$
\mathbb{E}_M[H_T]\ge \alpha_2|R|.
$
\end{proof} 
\begin{proof}[Proof of Proposition~\ref{P:beta} for $d< 10^4$]\ \\
Let $t\ge t_0:=t_\mathrm{mix}^{(\infty)}(n^{-10})$ and $\rho_0$ and $\epsilon$ be the constants from Lemma~\ref{L:losssmalld} and suppose $|R|\wedge|W|<\varrho_0 n$. If $C_\mathrm{round}>C_{\ref{L:piNice}}(10^{-4})$, and $T=C_\mathrm{round}(t_\mathrm{rel}+t_\ast(\epsilon)+s_\ast(\epsilon))$, by Lemma~\ref{L:losssmalld} (which we can apply here as Lemma~\ref{L:largedev} holds even for modified graphs since the number of interactions is unaffected by the addition of edges which never ring), there exists a universal $\alpha_2>0$ such that with probability at least $1-n^{-10}$, $M_{t}$ is $(\alpha_2,T)$-good, i.e.\  $\beta(\alpha_2,t_\mathrm{round}-1)\le n^{-10}$. 

On the other hand if $|R|\wedge|W|\ge\varrho_0 n$ then set $C_{\ast}= C_\ast^0\vee \sfrac{10^4}{\log(1/\varrho_0)}$ (with $C_\ast^0$ the constant from Lemma~\ref{L:losssmalldbigR1}). Then by Lemma~\ref{L:losssmalldbigR1} with $\varrho=\varrho_0$ there exist constants $\alpha_4>0$ and $C_\mathrm{round}^0$ such that if $C_\mathrm{round}> C_\mathrm{round}^0$ then with probability 1, $M_{t}$ is $(\alpha_4,T)$-good, for $T=C_\mathrm{round}t_\mathrm{rel}$, i.e.\ $\beta(\alpha_4,t_\mathrm{round}-1)=0$.

This completes the proof taking $\alpha=\alpha_2\wedge\alpha_4$.
\end{proof}

\section{Mixing for graphs of high degree: proof of Theorem~\ref{thm:maindeg}}\label{SS:thmmaindeg}
Recall that in this regime we have $d \gtrsim \log_{n/k} n$. The analogue of Proposition~\ref{P:beta} (which gives a bound on the probability that a configuration is not $(\alpha,t)$-good after a burn-in) for proving Theorem~\ref{thm:maindeg} is the following proposition.
\begin{proposition}\label{P:betadeg}
There exist constants $\alpha, C_\mathrm{round}, C_\mathrm{deg} >~0 $, such that for all $n$ sufficiently large
if $d\ge C_\mathrm{deg}\log_{n/k}n$ then $\beta(\alpha,C_\mathrm{round}\rel)\le n^{-10}$.
\end{proposition}

\begin{proof}[Proof of Theorem~\ref{thm:maindeg}]
This is identical to the proof of Theorem~\ref{thm:main1} (general mixing bound) in \S\ref{s:proofofmain} using Proposition~\ref{P:betadeg} in place of Proposition~\ref{P:beta}.
\end{proof}

In order to prove Proposition~\ref{P:betadeg} we must control red and black neighbours of red particles (it is not enough to only consider pairs of red and white particles). \textcolor{black}{We make use of the large degree to argue that after a burn-in period of duration $\mix^{(\infty)}(n^{-10})$, the probability that the number of black neighbours of a vertex $v$ will be unusually high  is extremely small. This considerably simplifies the analysis, as there is no longer a need to consider the number of intersections (during the first $t_*$ time units of the round) between a red particle and its black neighbours (at the end of the constant colour phase of the round). This is what allows us to avoid the term $r_*(c_{1.1}) \log (n/\eps) $ in Theorem \ref{thm:maindeg}.} 

\textcolor{black}{Indeed, each neighbour of $v$ is occupied by a black particle with probability at most $\frac{k-1}{n}+n^{-10} $.  Using negative association, we can bound the probability that vertex $v$ has at least  $(\sfrac{k}{n}+\zeta)d$ black neighbors, for some $\zeta >0$, by a bound similar to the probability that a Binomial$(d,\frac{k-1}{n}+n^{-10} )$ r.v.\ is at least $(\sfrac{k}{n}+\zeta)d$. By assumption on $d$ and $k$, the last probability can be made $n^{-13}$ (for each fixed $\zeta>0$), provided $C_\mathrm{deg}$ is taken to be sufficiently large.}

\begin{lemma}\label{L:losslargedegree}
Let $\zeta\in(0,1/16]$ and consider the case $d\ge  10^4\log_{n/k}n$. If $C_\mathrm{round}> C_{\ref{L:piNice}}(10^{-4})$ and $T=C_\mathrm{round}t_\mathrm{rel}$ then any configuration $M=(B,R,\eset,W)$ of the chameleon process satisfying
\[
\max_{v\in V}\Pr\left[\sum_{u:\,u\sim v}\indic{u\in \BB_T}\ge (\sfrac{k}{n}+\zeta)d\Bigm| \BB_0=B\right]\le n^{-10}
\]
is $(\alpha_3,T)$-good, for $\alpha_3>0$ a universal constant, and all $n$ sufficiently large.
\end{lemma}
\begin{proof}
This proof is very similar to the proof of Lemma~\ref{L:redloss}. We count $H_T$ in the same way and arrive at the bound (from equation~\eqref{eq:blah2})

\begin{align}\begin{split}
\mathbb{E}_M[H_T]\ge &\sfrac1{4d}\sum_{a}\Pr[a\in N(R)](1-L(\lambda,\theta,d,|R|))(d-\theta d)\\
&-\sfrac1{4d}\sum_{a,b:\,a\sim b}\sum_{v\in R}\indic{a\in \mathrm{Nice}(R)}\Pr\left[b\in I_{[0,T]}(B),\,a=I_{[0,T]}(v)\right].\end{split}\label{eq:blah3}
\end{align}

Let $E_\zeta(a)$ be the event that vertex $a$ has less than $(k/n+\zeta)d$ neighbours occupied by black particles at time $T$.  Then by the assumption on $M$, we have that $\Pr[E_\zeta(a)^\complement]\le n^{-10}$.
Let $N_t(v)$ be the number of neighbours of vertex $v$ occupied by black particles at time $t$.  

Summing over $a\in \mathrm{Nice(R)}$, $b:\, a\sim b$ and $v\in R$ in the second double sum on the r.h.s. of equation~\eqref{eq:blah3} gives
\begin{align}\notag
&\sum_{a,b:\,a{\sim}b}\sum_{v\in R}\indic{a\in \mathrm{Nice}(R)}\Pr[b\in I_{[0,T]}(\BB),a=I_{[0,T]}(v)]\\\notag&=\sum_a\sum_{v\in R}\indic{a\in \mathrm{Nice}(R)}\mathbb{E}\left[\indic{a=I_{[0,T]}(v)}\sum_{b:\,a{\sim}b}\indic{b\in I_{[0,T]}(\BB)}\right]\\\notag
&=\sum_a\sum_{v\in R}\indic{a\in \mathrm{Nice}(R)}\mathbb{E}\left[\indic{a=I_{[0,T]}(v)}N_T(a)\right]\end{align}\begin{align}\notag
&= \sum_a\sum_{v\in R}\indic{a\in \mathrm{Nice}(R)}\bigg(\mathbb{E}\left[\indic{E_\zeta(a)}\indic{a=I_{[0,T]}(v)}N_T(a)\right]\notag\\\notag&\phantom{=\sum_a\sum_{v\in R}\indic{a\in \mathrm{Nice}(R)}\mathbb{E}\bigg(}+\mathbb{E}\left[\indic{E_\zeta(a)^\complement}\indic{a=I_{[0,T]}(v)}N_T(a)\right]\bigg)\\\notag
&\le \sum_a\sum_{v\in R}\Big(\indic{a\in \mathrm{Nice}(R)}(k/n+\zeta) d\,\Pr[a=I_{[0,T]}(v)]\\\notag&\phantom{\le \sum_a\sum_{v\in R}\Big(}+d\,\Pr\left[E_\zeta(a)^\complement\cap\{a=I_{[0,T]}(v)\}\right]\Big)\\\notag
&\le\sum_{a}\Pr[a\in N(R)](k/n+\zeta) d+d\sum_a\Pr\left[E_\zeta(a)^\complement\right]\\
&\le \sum_{a}\Pr[a\in N(R)](k/n+\zeta) d+dn^{-9}.\label{eq:blah}
\end{align}
Combining equations \eqref{eq:blah3} and~\eqref{eq:blah} we have for any $\theta\in(0,1)$ and $\lambda>0$,
\begin{align*}
\mathbb{E}_M[H_T]&\ge 
\sfrac1{4d}\sum_{a}\Pr[a\in N(R)](1-L(\lambda,\theta,d,|R|))(d-\theta d)\\&\phantom{ge}-\sfrac1{4d}\left(\sum_{a}\Pr[a\in N(R)](k/n+\zeta) d+dn^{-9}\right)\\
&= \sfrac14\sum_a\Pr[a\in N(R)]\left\{\left(1-L(\lambda,\theta,d,|R|)\right)(1-\theta)-\sfrac{k}{n}-\zeta\right\}-\sfrac14n^{-9}.
\end{align*}
Choosing $\lambda=0.05$, $\theta=\frac{9}{16}-\frac{k}{2n}$ and using the bound $|R|/n\le \frac12-\frac{k}{2n}$, we have \textcolor{black}{precisely as in the paragraph following \eqref{eq:expec_diff4} }that  
\[
\sfrac1{d}\log L(\lambda,\theta,d,|R|)=-\lambda\theta +(e^\lambda-1)(1/32+|R|/n)\le -0.0008,
\]
and so since $\zeta\le 1/16$ and $d\ge 10^4$ we obtain the bound
\[
\mathbb{E}_M[H_T]\ge \sfrac1{64}\mathbb{E}[|N(R)|]-\sfrac14n^{-9}\ge \sfrac1{64}\mathbb{E}[|N(R)|]-\sfrac14n^{-8}|R|.
\]

 Notice now that $\mathbb{E}[|N(R)|]=|R|\,\Pr_{\pi_R}(X_T\in \mathrm{Nice}(R))$, for $(X_t)$ a realisation of RW$(G)$, and so by Lemmas~\ref{L:niceC} and~\ref{L:piNice} we have that, since $C_\mathrm{round}>C_{\ref{L:piNice}}(10^{-4})$ (and as there is no degree-inflation $d_\mathrm{max}^\mathrm{in}=\hat d$),
 \begin{align*}\mathbb{E}[|N(R)|]\ge |R|(\pi(\mathrm{Nice}(R))-10^{-4})&\ge |R|\left(1-\frac{|R|/n}{1/32+|R|/n}-2\times 10^{-4}\right)\\&\ge |R|\left(\sfrac1{17}-2\times 10^{-4}\right).\end{align*}

Thus we obtain
$
\mathbb{E}_M[H_T]\ge \alpha_3|R|,
$
for all $n$ sufficiently large and any $\alpha_3\le 0.0008$.
\end{proof}

\begin{proof}[Proof of Proposition~\ref{P:betadeg}]
If $k\le 10^{-5}n$ we make use of Lemma~\ref{L:losslargedegree} with $\zeta=1/16$. Recall the definition of $m_{\frac1{16},n,k}$ from Corollary~\ref{C:black}. We have the bound $$\sfrac1{32} d m_{\frac1{16},n,k}\ge \sfrac1{32} C_\mathrm{deg}\log n\left(1-\frac{\log(16e^2)}{\log (10^5)} \right)\ge \sfrac1{64}C_\mathrm{deg}\log n$$ and so combining Corollary~\ref{C:black} and Lemma~\ref{L:losslargedegree} \textcolor{black}{with $\zeta=1/16$} we deduce that if $C_\mathrm{deg}\ge 1000$ and $C_\mathrm{round}>C_{\ref{L:piNice}}(10^{-4})$ then $\beta(\alpha_3,C_\mathrm{round}t_\mathrm{rel})\le n^{-10}$ for some universal $\alpha_3>0$. On the other hand if $k>10^{-5}n$ then we will instead make use of Lemma~\ref{L:losslargedegree} with $\zeta=\frac14\times 10^{-5}$. We have the bound (for each $\eps\in(0,1)$)
$\sfrac12d\eps m_{\eps,n,k}\ge \sfrac1{4}d\eps^2\sfrac{n}{k}\left(\sfrac12-\sfrac{\eps n}{k}\right) $ and so with $\eps=\zeta=\sfrac14\times 10^{-5}$ we obtain
$\sfrac12d\eps m_{\eps,n,k}\ge 10^{-13}d$ and therefore for $C_\mathrm{deg}$ sufficiently large (e.g.\ $10^{21}$) and $C_\mathrm{round}>C_{\ref{L:piNice}}(10^{-4})$ we get using Corollary~\ref{C:black} with Lemma~\ref{L:losslargedegree} that $\beta(\alpha_3,C_\mathrm{round}t_\mathrm{rel})\le n^{-10}$ for some universal $\alpha_3>0$.\end{proof}

\section{Mixing for sublinear number of particles: proof of Theorem~\ref{thm:main3}}\label{SS:thm2proof}
We consider separately two cases depending on the growth rate of $k$.
\subsection{The case $\sqrt n\le k\le n^\delta$}

We further split into two sub-cases depending on the degree. The first is for  $d\ge C_\mathrm{deg}/(1-\delta)$ where $C_\mathrm{deg}$ is the constant (of the same name) from Theorem~\ref{thm:maindeg}. The proof of Theorem~\ref{thm:main3} in this case follows immediately by Theorem~\ref{thm:maindeg} \textcolor{black}{(recall that $t_{\mathrm{sp}}(n^{-a})  \asymp_a  \rel  \log n  $)}.

The second sub-case is for  $d<C_\mathrm{deg}/(1-\delta)$. In this case the analogue of Proposition~\ref{P:beta} (to obtain a bound on the probability that a configuration is not $(\alpha,t)$-good after a burn-in) for this case is the following.

\begin{proposition}\label{P:beta3}
There exist constants $\alpha_\delta, C_\mathrm{deg}, C_\delta >~0 $, such that for all $n$ sufficiently large
if $k\le n^{\delta}$ and $d<C_\mathrm{deg}/(1-\delta)$, then $\beta(\alpha_\delta,C_\mathrm{\delta}\rel)\le n^{-10}$.
\end{proposition}
To complete the proof of Theorem~\ref{thm:main3} for this case we apply the same arguments as in the proof of Theorem~\ref{thm:main1} (general mixing bound) in \S\ref{s:proofofmain} using Proposition~\ref{P:beta3} in place of Proposition~\ref{P:beta}.

To prove Proposition~\ref{P:beta3} in the regime where $|R|\wedge|W|$ is small we need the following lemma to control black particles. The proof of this lemma is omitted as it is similar to the proof of Lemma~\ref{L:losslargedegree}. The degree-inflation referred to in the statement is with $\hat d=\lceil C_\mathrm{deg}/(1-\delta)\rceil$.

\begin{lemma}[Proof omitted]\label{L:lossdelta2}
Let $\delta\in (0,1)$ and $\zeta\in(0,1/16]$ and consider the case $k\le n^{\delta},\,d<\sfrac{C_\mathrm{deg}}{1-\delta}$. There exist constants $C_\delta,\varrho_\delta$ such that if $C_\mathrm{round}> C_\delta$ then any configuration $M=(B,R,\eset,W)$ of the chameleon process with $|R|\wedge |W|<\varrho_\delta n$ satisfying
\[
\max_{v\in V}\Pr\left[\sum_{u:\,v\stackrel{\rightarrow}{\sim}u}\indic{u\in \BB_T}\ge (k/n+\zeta)\sfrac{C_\mathrm{deg}}{1-\delta}\Bigm| \BB_0=B\right]\le n^{-10}
\]
is $(\alpha_\delta,T)$-good, for $T=C_\mathrm{round}t_\mathrm{rel}$, $\alpha_\delta>0$ a constant depending only on $\delta$, and all $n$ sufficiently large.
\end{lemma}

\begin{proof}[Proof of Proposition~\ref{P:beta3}]
Suppose (as in the proof of Proposition~\ref{P:beta}) that $t\ge t_0$. Let  $\varrho_\delta$ be the constant from Lemma~\ref{L:lossdelta2}, let $\hat d$ equal $\lceil C_\mathrm{deg}/(1-\delta)\rceil$, and suppose $|R|\wedge|W|<\varrho_\delta n$.  By Corollary~\ref{C:black} we have that for any $\eps>0$, $B\in(V)_{k-1}$, $v\in V$, and $n=n(\eps)$ sufficiently large,
\begin{align*}
&\Pr\left[\Pr\big[\sum_{u:\,u\sim v}\indic{u\in \BB_{T+t}}\ge\left(\sfrac{k}{n}+\varepsilon\right)\hat d\Bigm| \mathcal{F}_{t}\big]\ge  \exp\left(-\sfrac{(1-\delta) \hat d\eps}{4}\log{n}\right)\Bigm| \BB_0=B\right]\\&\le  \exp\left(-\sfrac{(1-\delta) \hat d\eps}{4}\log{n}\right).
\end{align*} 
Taking $\eps=1/16$ we deduce by Lemma~\ref{L:lossdelta2} that there exist constants $C_\delta,\alpha_\delta>0$ such that if $C_\mathrm{round}>C_\delta$ then with probability at least $1-n^{-10}$, $M_t$ is $(\alpha_\delta,T)$-good, for $T=C_\mathrm{round}t_\mathrm{rel}$, i.e.\ $\beta(\alpha_\delta,t_\mathrm{round}-1)\le n^{-10}$.

On the other hand if $|R|\wedge|W|\ge\varrho_\delta n$ then set $C_{\ast}= C_\ast^0\vee \sfrac{C_\mathrm{deg}}{(1-\delta)\log(1/\varrho_\delta)}$. Then by Lemma~\ref{L:losssmalldbigR1} with $\varrho=\varrho_\delta$ there exist constants $\alpha_4(\delta)>0$ and $C_\mathrm{round}^0(\delta)$ such that if $C_\mathrm{round}>C_\mathrm{round}^0(\delta)$ then with probability at least $1-n^{-10}$, $M_{t}$ is $(\alpha_4,T)$-good, for $T=C_\mathrm{round}t_\mathrm{rel}$, i.e.\ $\beta(\alpha_4,t_\mathrm{round}-1)\le n^{-10}$.
\end{proof}

\subsection{The case $k<\sqrt n$}
In this case we require a different chameleon process.
This new version of the chameleon process has rounds of varying duration.  To be precise, the duration of each burn-in period is taken to be $\mix^{(\infty)}(\hat c /k )$ for some absolute constant $\hat c \in (0,1)$ chosen to be as large as possible while satisfying the requirements described in Remark~\ref{R:black}. Further, if at the beginning of the $j$th round we have $r $ red particles, the round starts with an $(\alpha,L(r)-1)$-good configuration, where if $r \wedge (n-|\BB_0|-r) \in (2^{i-1},2^i] $ then  
\begin{equation}
\label{e:L(r)def}
L(r)=L_{i}:=C_{\mathrm{round}} /\Lambda(C_{\mathrm{profile}}2^{i}/n) + 1,
\end{equation}
 where $\Lambda(\bullet)$  is as in \S\ref{s:profiles}, for some absolute constants $C_{\mathrm{round}},C_{\mathrm{profile}}>0$ to be determined later. The constant-colour ``relaxation'' phase for such a round is of duration $L(r)-1$, while the pinkening phase is again of unit length. Thus the duration of  the $j$th round is $t_\mathrm{round}(j):=L(|\RR_{\rho_j}|) $ and so $\hat \tau_j :=\rho_j+L(|\RR_{\rho_j}|) $, where $\rho_j$ and $\hat \tau_j$ still denote the beginning and end of the $j$th round. 

    At the end of such a round we follow the same rule as in the constant-round chameleon depinking procedure, apart from the fact that we replace above $p(M_{\rho_i},t_\mathrm{round}-1)$ by $p(M_{\rho_i},L(|\RR_{\rho_j}|)-1) $. If after a depinking time we have $r$ red particles, then we start the following round immediately if the current configuration is $(\alpha,L(r)-1)$-good. Otherwise, we perform a sequence of burn-in periods of duration $t_{\mathrm{mix}}^{(\infty)}(\hat c /k ) $ until the end of the first burn-in period after which we have an  $(\alpha,L(r)-1)$-good configuration, where $r$ denotes the number of red particles at the end of this burn-in.  
\textcolor{black}{Recall the process $\hat M_t$ used in the definition of $\beta(\alpha,t)$ in \eqref{e:beta1}.} Let $t_1 := t_{\mathrm{mix}}^{(\infty)}(\hat c /k ) $ and for $i \le \lceil \log_2 (\sfrac{n-k+1}{2}) \rceil $ define
 \begin{equation}
\label{e:beta2}\begin{split}
&\beta_{i}(\alpha):= \max_{B,R,W}\sup_{s \ge t_{1}  } \Pr[\hat M_s \text{ is not $(\alpha,L_{i}-1)$-good} \mid    \hat M_0=(B,R,W)],\end{split}
\end{equation}
where the maximum is taken over all partitions of $V$ into sets $\mathbf{O}(B),R,W$ satisfying that $|R| \wedge |W| \in (2^{i-1},2^i] $ and  $B \in (V)_j $ for some $j < \sqrt{n}$ satisfying $\{B(i):i  \in [j] \} =\mathbf{O}(B) $. We will show that if $k=|\BB_0|+1 \le \sqrt{n} $ then  for some absolute constant $\alpha,C >0 $, we have that $\max_{i \le \lceil \log_2 (\sfrac{n-k+1}{2}) \rceil }\beta_i \le n^{-10} $ (Proposition \ref{P:beta4}). 

Recall Proposition~\ref{p:chambound2} from \S\ref{S:cham}. The following proposition serves as its replacement for this setting. In simple words, it asserts that for some absolute constant $M$, if no additional burn-in periods occurred (other than the initial one, whose duration is $t_{\mathrm{mix}}^{(\infty)}(\hat c /k )$) by time $t_{\mathrm{mix}}^{(\infty)}(\hat c /k )+Mt_{\mathrm{sp}}(\sfrac{1}{4s} )$, then for all $s \in [k,n^3] $ the expected fraction of ``missing ink" at time $t_{\mathrm{mix}}^{(\infty)}(\hat c /k )+Mt_{\mathrm{sp}}(\sfrac{1}{4s} ) \lesssim t_{\mathrm{sp}}(\sfrac{1}{4s} )  $  would be at most $s^{-1}$.
 This assertion is similar to the treatment of the chameleon process in \cite{morris}, where $t_{\mathrm{evolving-sets}}$ is used instead of $t_{\mathrm{sp}}$. While it seems that one can derive it  from the analysis in \cite{morris}, we give a different proof, which we believe to be simpler.  
\begin{proposition}[{Proof in in 
 Appendix~\ref{s:upperhard}}]
\label{p:chambound3}
There exists an absolute constant $M$ such that for all  $s \in [k,n^3] $,   $k \le \sqrt{n} $ and $(\mathbf{w},y)\in (V)_k $ \textcolor{black}{with $\mathbf{w}\in (V)_{k-1}$ and $y\in V$}, if we write  $\hat t(s):=t_{\mathrm{mix}}^{(\infty)}(\hat c /k )+Mt_{\mathrm{sp}}(\sfrac{1}{4s} ) $, then
\begin{equation}
\label{e:IPtoCP''}
\begin{split}
  \hE_{(\mathbf{w},y)}[ 1 - \ink_{\hat t(s)} / (n-k+1) ] & \le s^{-1} +Mt_{\mathrm{sp}}(\sfrac{1}{4s} ) (n-k+1) \max_{i< \lceil \log_2 (n-k+1) \rceil } \beta_{i}(\alpha).  
\end{split} \end{equation}
\end{proposition}
The next proposition serves as the replacement of Proposition~\ref{P:beta}.
\begin{proposition}\label{P:beta4}
There exist constants $\alpha_\frac12, C_\mathrm{round}, C_\mathrm{profile} >~0 $, such that for all $n$ sufficiently large
 if $k\le \sqrt n$ then $\max_{i < \lceil \log_2 (n-k+1) \rceil }\beta_i(\alpha_\frac12)\le n^{-10}$ (recall that the definition of $\beta_i$ depends on constants $C_\mathrm{round},C_\mathrm{profile}$ through the definition of $L_i$).
\end{proposition}

\begin{proof}[Proof of Theorem~\ref{thm:main3}]
The result has already been shown for the case of $\sqrt n\le k\le n^\delta$, so it remains to prove the result for $k<\sqrt n$.

Using  sub-multiplicativity  we have that $\mixex((2n)^{-i})$ $\le i\mixex(\sfrac{1}{4n})$. It follows that it suffices to consider  $ \eps \in [\sfrac{1}{4n},\sfrac1{4k}] $. 
 Combining Propositions \ref{p:chambound}, \ref{p:chambound3} and \ref{P:beta4} concludes the proof,
upon observing that the term \[Mt_{\mathrm{sp}}(\sfrac{1}{4s} ) (n-k+1) \max_{i\le \lceil \log_2 (\sfrac{n-k+1}{2}) \rceil} \beta_{i}(\alpha)\] in the r.h.s.\ of \eqref{e:IPtoCP''} is at most $\lesssim n^{-10} \times n \times n^2 \log n  $ (using  $\max_i \beta_{i}(\alpha) \le n^{-10}$ and  $t_{\mathrm{sp}}(\eps) \lesssim \rel  \log (n/\eps) \lesssim n^2 \log (n/\eps) $ for $\eps \le 1/2$, e.g.\ \cite{sensitivity,Lyonsev}).  
\end{proof}

We require two lemmas for the proof of Proposition~\ref{P:beta4}. The first is a modification of Lemma~\ref{L:piNice} suitable for this setting. Instead of using the Poincar\'e inequality we use \eqref{e:spb3} to bound $
\|\Pr_{\pi_S}[X_T\in\bullet]-\pi\|_{2,\pi}^2$.

\begin{lemma}[{Proof in Appendix~\ref{proofpiNice}]}]\label{L:piNice2}
We denote the uniform distribution on $S$ by $\pi_S $.
For each $\varepsilon\in(0,1)$, there exist $C_{\ref{L:piNice2}}(\varepsilon), C_\mathrm{p}(\eps)>1$ such that for all $C_\mathrm{round}>C_{\ref{L:piNice2}}(\varepsilon)$, $C_\mathrm{profile}>C_\mathrm{p}(\eps)$ and all $S\subset V$ with $2|S|\le n$,
\[\Pr_{\pi_S}[X_T\in\mathrm{Nice}(S)]\ge\pi(\mathrm{Nice}(S))-\varepsilon,\] for any $T\ge C_\mathrm{round}/\Lambda(C_\mathrm{profile}|S|/n).$
\end{lemma}

We will also use the following lemma. The proof is very similar to the proof of Lemma~\ref{L:losslargedegree} except instead of using Lemma~\ref{L:piNice} to bound $\mathbb{E}[|N(R)|]$ we use Lemma~\ref{L:piNice2} which is where the bound on $C_\mathrm{profile}$ originates.

\begin{lemma}[Proof omitted]\label{L:losshalf}
Let $\zeta\in(0,1/16]$ and consider the case $k\le \sqrt n$. There exist constants $C_{\frac12}, C_\mathrm{p}, \varrho_{\frac12}$ such that if $C_\mathrm{round}> C_\frac12$ and $C_\mathrm{profile}>C_\mathrm{p}$ then any configuration $M=(B,R,\eset,W)$ of the chameleon process with either (i) $|R|\wedge|W|<\varrho_{\frac12} n$ and $d<2\times 10^4$, or (ii) $d\ge 2\times 10^4$; and satisfying 
\[
\max_{v\in V}\Pr\left[\sum_{u:\,v\stackrel{\rightarrow}{\sim}u}\indic{u\in \BB_T}\ge (k/n+\zeta)\hat d\Bigm| \BB_0=B\right]\le n^{-10}
\]
with $\hat d=2\times 10^4$ in case (i) and $\hat d=d$ in case (ii),
is $(\alpha_\frac12,T)$-good, for $T\ge C_\mathrm{round}/\Lambda(C_\mathrm{profile}|R|/n)$, $\alpha_\frac12>0$ a universal constant,  and all $n$ sufficiently large.
\end{lemma}

\begin{proof}[Proof of Proposition~\ref{P:beta4}]
Recall the notation $t_1 := t_{\mathrm{mix}}^{(\infty)}(\hat c /k ) $ and let $t\ge t_1$.    Let  $\varrho_\frac12$ be the constant from Lemma~\ref{L:losshalf} and suppose either (i) $|R|\wedge|W|<\varrho_\frac12 n$ and $d<2\times 10^4$ or (ii) $d\ge2\times 10^4$.  
By Corollary~\ref{C:black} and Remark~\ref{R:black} following it (and recalling the choice of $\hat c$) we have that for any $\eps>0$, $B\in(V)_{k-1}$, $v\in V$, and $n=n(\eps)$ sufficiently large,
\begin{align*}
&\Pr\left[\Pr\big[\sum_{u:\,u\sim v}\indic{u\in \BB_{T+t}}\ge\left(\sfrac{k}{n}+\varepsilon\right)\hat d\Bigm| \mathcal{F}_{t}\big]\ge  \exp\left(-\sfrac{ \hat d\eps}{2}\log{(\eps \sqrt n /e^2)}\right)\Bigm| \BB_0=B\right]\\&\le  \exp\left(-\sfrac{ \hat d\eps}{2}\log{(\eps \sqrt n /e^2)}\right)
\end{align*} 
(with $\hat d=2\times 10^4$ in case (i) and $\hat d=d$ in case (ii)).
Taking $\eps=1/16$ we deduce by Lemma~\ref{L:losshalf} with $|R|\textcolor{black}{\wedge|W|}\in(2^{i-1},2^i]$ that there exist constants $C_\mathrm{\frac12}$ and $C_\mathrm{p}$ such that if $C_\mathrm{round}>C_{\frac12}$ and $C_\mathrm{profile}>C_\mathrm{p}$ then with probability at least $1-n^{-10}$, $M_{t}$ is $(\alpha_\frac12, T)$-good, for $T=C_\mathrm{round}/\Lambda(C_\mathrm{profile}2^i/n)$, i.e.\ $\beta_i(\alpha_\frac12)\le n^{-10}$.

On the other hand, if $|R|\wedge |W|\ge\varrho_\frac12 n$, then set $C_{\ast,\frac12}= C_\ast^0\vee \sfrac{10^4}{\delta\log(1/\varrho_\frac12)}$. Then by Lemma~\ref{L:losssmalldbigR1} with $\varrho=\varrho_\frac12$ there exist constants $\hat\alpha_\frac12>0$, $C^0_\mathrm{round}$, $\hat C_\mathrm{p}$ (chosen so that for any $|R|\ge\varrho_\frac12 n$ we have $\Lambda(\hat C_\mathrm{p}|R|/n)=1/\rel$) such that if $C_\mathrm{round}>C^0_\mathrm{round}$ and $C_\mathrm{profile}>\hat C_\mathrm{p}$ then with probability at least $1-n^{-10}$, $M_{t}$ is $(\hat\alpha_\frac12,T)$-good, for $T=C_\mathrm{round}/\Lambda(C_\mathrm{profile}2^i/n)$, i.e.\ $\beta_i(\hat\alpha_\frac12)\le n^{-10}$.
\end{proof}

\section{Lower bounds: proof of Theorem \ref{thm:lower} and Proposition \ref{p:mixexkatleastmixrw1}}
\label{s:lower}

Recall that $\Pr_F^{\mathrm{EX}(k)}$ is the distribution of the exclusion process with initial set $F$. We denote by $\Pr_{\mu}^t$ (resp.~$\Pr_{\mu}$) the distribution of $X_t$ (resp.~$(X_t)_{t \ge 0 }$), given that the initial distribution is $\mu $.

\noindent \emph{Proof of Theorem \ref{thm:lower}.}
By the spectral decomposition $-\cL $ has eigenvalues  $0=\la_1 < \la_2 \le \cdots \le \la_n $. Denote  the corresponding orthonormal basis (w.r.t.\ $\langle \bullet, \bullet \rangle_{\pi} $) of eigenvectors by $f_1=\mathbf{1} ,f_2,\ldots,f_n $. W.l.o.g.\ we may assume that $\la=\la_i$, $f=f_i$ and that   $t=t(k,\delta,\eps,\la):=\frac{1}{2\la}(4\delta\log k - \log (16/\eps ) ) \ge 0$. Consider $B:=\{f \ge 0 \}$. W.l.o.g.\ $|B| \ge n/2 $ (otherwise consider $-f$).  Let $F \in \binom{V}{k} $ be such that $\mathbb{E}_{F}[|A_t \cap B| ]=\max_{J \in \binom{V}{k}}\mathbb{E}_{J}[|A_t \cap B| ] $.   Then by negative correlation \begin{equation}
\label{e:VarF}
\begin{split}
 \Var_{\pi_{\mathrm{EX}(k)}}|A_0 \cap B|& \le \mathbb{E}_{\pi_{\mathrm{EX}(k)}}[|A_0 \cap B| ]=k \pi(B). \\
 \Var_{F}|A_t \cap B| &\le \mathbb{E}_{F}[|A_t \cap B| ].
\end{split}
\end{equation}
Denote $\sigma^2:=\half ( \Var_{\pi_{\mathrm{EX}(k)}}|A_0 \cap B|+ \Var_{F}|A_t \cap B|)    $. By the standard method of distinguishing statistics \cite[Proposition 7.12]{levin} if $a:=|\mathbb{E}_{F}[|A_t \cap B| ]-\mathbb{E}_{\pi_{\mathrm{EX}(k)}}[|A_0 \cap B| ]|^{2} \ge 4r\sigma^2 $, then \[\|\Pr_F^{\mathrm{EX}(k)}(A_t \in \bullet) -\pi_{\mathrm{EX}(k)} \|_{\TV} \ge 1- \sfrac{1}{1+r}.  \]
We will show that $a \ge 4k/\eps  $, which means that we can take above $r=1/\eps$, as  
\[k \ge  \half ( \mathbb{E}_{\pi_{\mathrm{EX}(k)}}[|A_0 \cap B| ] + \mathbb{E}_{F}[|A_t \cap B| ] )\ge \sigma^2 ,  \]
where the first inequality is trivial and  the second inequality follows from \eqref{e:VarF}.

 If  $D \sim \mathrm{Unif}(\{U \subseteq B:U \in \binom{V}{k}  \}) $, $\pi_B $ is the uniform distribution on $B$ and $(X_s)_{s \in \R_+}$ is a random walk on the network $(G,(r_e:e \in E))$ then using the maximality of $F$ (first inequality) and the spectral decomposition in the third equality (namely, $1_B=\pi(B)+ \sum_{j=2}^n \sum_{b \in B}\pi(b)f_j(b) f_j $) 
\begin{equation*}
\begin{split}
& \mathbb{E}_{F}[|A_t \cap B| ] \ge \mathbb{E}_{D}[|A_t \cap B| ]= k \Pr_{\pi_B}[X_t \in B ]=\sfrac{k}{\pi(B)}\langle P_t 1_B,1_{B}\rangle_{\pi} 
\\ &=k \pi(B)+\sfrac{k}{\pi(B)}  \sum_{b' \in B}\pi(b')\sum_{j>1} \sum_{b\in B }\pi(b) f_j(b)f_j(b')  e^{- \la_j t}  
\\ & \text{(write $b_j:=\sum_{b \in B }\pi(b) f_j(b)$)} \quad =k \pi(B)+\sfrac{k}{\pi(B)}\sum_{j>1} b_j^2 e^{- \la_j t}   \\ & \ge k \pi(B)+\sfrac{k}{\pi(B)}b_{i}^2 e^{- \la t} = k \pi(B)+\sfrac{k}{2\pi(B)}\|f\|_1^{2}e^{- \la t},  
\end{split}
\end{equation*}
where we used the fact that $f=f_i $ is orthogonal to $f_1=\mathbf{1}$ and thus $\E_{\pi}f=0 $ and   $\sum_{b \in B }\pi(b) f_i(b)=\E_{\pi}[f \vee 0] =\|f\|_1/2 $. We get that $a\ge k^{2} \|f\|_1^{4}e^{-2 \la t}/4 $. By the choice $t= \frac{1}{2\la}(4\delta\log k - \log (16/\eps ) )   $ and the assumption $\|f\|_1^{4} \ge k^{-1+4 \delta } $ we get that $a \ge k^{2} \|f\|_1^{4}e^{-2 \la t}/4 \ge k(4/ \eps) \ge 4 \sigma^2 / \eps $,  \textcolor{black}{as desired.}  \qed 

\begin{remark}
It is interesting to note that when $\|f\|_1 \le k^{-1/8}$ for some unit eigenfunction $f$ as above, it follows from H\"older's inequality that $\|f \|_{\infty} \ge \|f\|_2^2/\|f\|_1 \ge k^{1/8}$ (the exponent $1/8$ in  $\|f\|_1 \le k^{-1/8}$  is taken as some arbitrary constant smaller than $1/4$, the exponent appearing in Theorem \ref{thm:lower}). In this case, Wilson's method (\cite{Wilson}, see \cite[\S\ 13.5]{levin} for a systematic presentation of the method) can sometimes yield that $\mix^{\RW(1)} \ge c \la^{-1}\log k$. We note that in \cite{Wilson} Wilson applied his method to prove a lower bound on the mixing time of $\mathrm{EX}(2^{d-1}) $ and $\IP(2^d)$ for the hypercube $\{ \pm 1\}^d $. Our argument is different, in that we obtain control on the variances ``for free" as a consequence of negative correlation.
\end{remark}

\noindent \emph{Proof of Proposition \ref{p:mixexkatleastmixrw1}} As processes EX$(k)$ and EX$(n-k)$ are identical it suffices to consider $k\le n/2$. For fixed such $k$ and $x\in V$, let $B$ be the set of vertices $y$ having the $k$ smallest $P_t(x,y)$ values, where $t=22\mixex$. By submultiplicativity $t\ge 2\mixex(2^{-11})$. 
Set $\delta := \max_{y \in B}  P_t(x,y)$ so that
$P_t(x,B) \leq \delta k$. \textcolor{black}{Using the general fact that 
\[   \min_{C \in \binom{V}{k}}P_t^{\mathrm{EX}(k)}(B,C)/\pi_{\mathrm{EX}(k)}(C) \geq
(1-2 \|P_{t/2}^{\mathrm{EX}(k)}(B,\cdot) - \pi_{\mathrm{EX}(k)} \|_{\mathrm{TV}})^2 \] (e.g. Lemma 7 of Chapter 4 of~\cite{aldous})
 we also have}
\[\textcolor{black}{P_t(x,B)=\sum_{C:\,x\in C}P_t^{\mathrm{EX}(k)}(B,C)\ge\sum_{C:\,x\in C}\frac{(1-2^{-10})^2}{\binom{n}{k}} = \frac{k}{n}(1-2^{-10})^2.}\] Hence we obtain the bound $\delta\ge \sfrac1{n}(1-2^{-10})^2$. We distinguish between two cases depending on the value of $k$.

Consider first the case $k\le n/8$. We have the bound
\begin{align*}
\sum_y\left(\sfrac1{n}-P_t(x,y)\right)_+&=\sum_{y\in B}\left(\sfrac1{n}-P_t(x,y)\right)_++\sum_{y\notin B}\left(\sfrac1{n}-P_t(x,y)\right)_+\\
&\le \frac{k}{n}+(n-k)\left(\sfrac1{n}-\delta\right)\\&
\le1-\left(1-2^{-10}\right)^2(1-k/n)<1/4.
\end{align*}
Now consider the case $k>n/8.$ We note that $\delta(n-k)\le\sum_{y\notin B}P_t(x,y)\le 1$ and so as $k\le n/2$ we obtain $\delta\le 2/n$. Thus by a simple counting argument (using $P_t(x,B)\ge \sfrac1{8}\left(1-2^{-10}\right)^2$ as $k>n/8$) there must be at least $\sfrac{n}{32}$ vertices $b\in B$ satisfying $P_t(x,b)\ge \sfrac1{8n}\left(1-2^{-10}\right)^2$. Thus we have the bound
\begin{align*}
\sum_y\left(\sfrac1{n}-P_t(x,y)\right)_+&=\sum_{y\in B}\left(\sfrac1{n}-P_t(x,y)\right)_++\sum_{y\notin B}\left(\sfrac1{n}-P_t(x,y)\right)_+\\
&\le \frac{n}{32}\left(\sfrac{1}{n}-\sfrac1{8n}\left(1-2^{-10}\right)^2\right)+\frac{1}{n}\left(k-\sfrac{n}{32}\right)+(n-k)\left(\sfrac1{n}-\delta\right)\\
&\le 1-\left(1-2^{-10}\right)^2\left(\sfrac12+\sfrac1{256}\right)<\frac12-2^{-9}.
\end{align*}
Thus, combining the two cases, we obtain that for all $k\le n/2$, \[
\mixrw\left(\frac12-2^{-9}\right)\le 22\mixex,
\]
and hence there exists a constant $c$ ($2^{-13}$ suffices) so that $c\mixrw\le \mixex$ for all $k\in[n-1]$.\qed

\section{Examples}
\label{s:examples}
 We present three additional applications of our results. In the following we denote by  $B_r  $  a ball of radius $r$. We prove the claimed bounds on $r_\ast(c_{1.1})$ in \S\ref{s:mod}. In \S\ref{s:hypercube} we show how we can apply our results to product graphs and in particular obtain the bounds claimed in \S\ref{ss:main} for the hypercube.
\begin{itemize}
\item[(i)] For an $n$-vertex $d$-regular vertex-transitive graph satisfying  $|B_r| \ge ce^{cr} $ for all $r $ such that  $|B_r| \le \sfrac{4}{c_{1.1}} \log n $, for some $c>0$, we have that $r_*(c_{1.1}) \lesssim d^{2} (\log \log n)^3  $ (see Proposition \ref{p:expamples}). Hence (by Theorem \ref{thm:main1})   $\max_k \mixex \lesssim \rel \log n  $, provided that $\rel \gtrsim  d^{2}(\log \log n)^3  $.
\item[(ii)] For an $n$-vertex $d$-regular vertex-transitive graph satisfying  $|B_r| \ge ce^{cr^{\alpha}} $ for all $r $ such that  $|B_r| \le \sfrac{4}{c_{1.1}}\log n $, for some $\alpha \in (0,1)$ and $c>0$, we have that $r_*(c_{1.1}) \lesssim d^{2} (\log \log n)^{1+\sfrac{2}{\alpha}}  $ (see Proposition \ref{p:expamples}). Hence   $\max_k \mixex \lesssim \rel \log n  $, provided that $\rel \gtrsim  d^{2} (\log \log n)^{1+\sfrac{2}{\alpha}}$. In particular, this holds if  $|B_r| \le Ce^{Cr^{\beta}} $ for all $r $,  for some $\beta \in (0,1)$ and $C>0$ (as this implies that $\rel \gtrsim  \sfrac{\mathrm{Diameter}}{\log n} \gtrsim  (\log n)^{(1-\beta)/\beta}  $). 
\item[(iii)] The following example is taken from \cite[\S4.2.1]{spectral} (we refer the reader there for the relevant definitions; See also \cite{Nash}, where it is shown that Cayley graphs of moderate growth satisfy a local-Poincar\'e inequality, and many other examples are given). 

If $G$ is a $d$-regular graph of diameter $\gamma$ and $(A,c)$-moderate growth, satisfying a local-Poincar\'e inequality with a constant $a$, then  $\rel \asymp \gamma^2   \asymp  t_{\mathrm{sp}}(\sfrac{1}{4}) $ (with the implicit constants depending on $d,A,c $ and $a$);  $\rel \asymp \gamma^2 $ is due to Diaconis and Saloff-Coste \cite[Theorem 3.1]{moderate} (cf.\ our \S\ref{s:mod}). By Corollary \ref{cor:example},  $\mixex \asymp_{a,d,A,c} \gamma^2 \log (k+1)$ uniformly in $k$.        
\end{itemize}
\subsection{Vertex-transitive graphs and the giant component of super-critical percolation}
\label{s:mod}
Let $G=(V,E)$ be an $n$-vertex connected graph. We say that $G$ is \emph{vertex-transitive} if the action of its automorphism group on its vertices is transitive. Denote the volume of a ball of radius $r$ in $G$ by $V(r)$. Denote the diameter of $G$ by $\gamma:=\inf\{r:V(r) \ge n \} $. Following Diaconis and Saloff-Coste we say that  $G$ has $(c,a)$-\emph{moderate growth} if $V(r) \ge cn(r/\gamma)^a $. \textcolor{black}{Breuillard and Tointon \cite{Matt} proved
that for Cayley graphs of constant degree, this condition is equivalent in some quantitative sense to the
condition that $n \le \beta \gamma^{\alpha} $ for some $\alpha,\beta > 0$, which is of course a much simpler condition.} Let $P$ be the transition matrix of simple random walk (SRW) on $G$. We consider the case of continuous-time SRW with $\cL=P-I$.    Diaconis and Saloff-Coste \cite{moderate} showed that for a  Cayley graph $G$ of $(c,a)$-moderate growth we have  \[c^2\gamma^2 4^{-2a-1} \lesssim \rel \lesssim t_{\mathrm{mix}}^{(\infty)} \lesssim_{c,a}\gamma^2.   \]  
We note that the proof of $c^2\gamma^2 4^{-2a-1} \lesssim \rel$ works even if $G$ is merely vertex-transitive of $(c,a)$-moderate growth. Namely, they argue that the function  $h(x):=\mathrm{distance}(x,\oo) $ (for some arbitrary $\oo \in V$) satisfies that $\Var_{\pi}h/\EE(h,h) \ge \Var_{\pi}h \ge \gamma^{2} (V( \lfloor \gamma/4 \rfloor )/2n)^{2} \ge  c^2\gamma^2 4^{-2a-1}  $. Indeed, if $h(x)= \gamma $ then for the vertices $y$ in the ball of radius $ r:=\lfloor \gamma/4 \rfloor$ centered at $x$ (resp.\ $\oo$) we have $h(y) \ge \frac{3}{4}\gamma $ (resp.\ $\le \sfrac{\gamma}{4} $). Denote these two balls by $B_{x}(r)$ and $B_{\oo}(r)$. If $X,Y$ are i.i.d.\ $\pi=\mathrm{Unif}(V)$ then
\begin{equation}
\label{e:h}
\Var_{\pi}h =\sfrac{1}{2}\E [(h(X)-h(Y))^2] \ge \sfrac{\gamma^{2}}{8}\pi(B_{x}(r))\pi(B_{\oo}(r)). \end{equation}
Lyons et al.\ \cite[Lemma 7.2]{LMS} showed that for an $n$-vertex vertex-transitive graph,  for all $A \subset V $ such that $|A| \le n/2$ we have 
\begin{equation}
\label{e:isoVT}
\frac{|\pd_{\mathrm{V}}^{\mathrm{in}} A |}{|A|} \ge \frac{1}{2 R(2|A|) },
\end{equation}
where $\pd_{\mathrm{V}}^{\mathrm{in}} A :=\{a \in A:P(a,A^c)>0 \} $ is the internal vertex boundary of $A$ and $ R(m):=\inf\{r:V(r) \ge m \} $ is the inverse growth function \textcolor{black}{(note that $R(m)=\infty$ for $m>n$, and so \eqref{e:isoVT} holds trivially when $|A|>n/2$). (Lemma 7.2 in \cite{LMS} is stated for infinite unimodular graphs, but the proof works verbatim for finite transitive graphs, which are always unimodular. See also \cite[Lemma 10.46]{LyonsPeres}, where $G$ is not assumed to be infinite.)}
 
\begin{proposition}
\label{p:moderate}
If $G$ is a $d$-regular vertex-transitive of $(c,a)$-moderate growth then
\begin{equation}
\label{e:VT}
c^2\gamma^2 4^{-2a-1} \le  \rel \lesssim t_{\mathrm{evolving-sets}}(1/4) \lesssim a (2/c)^{2/a} d^2\gamma^2. \end{equation}
Consequentially, (uniformly in $k$)
\begin{equation}
\label{e:VTEXmix}
\mixex \asymp_{c,a,d} \gamma^2  \log (k+1).
\end{equation}  

Similarly, if $G$ is the largest connected component of super-critical percolation on $(\Z/L \Z)^d$ with parameter $p$ then $\whp$ (as $L \to \infty $)
\begin{equation}
\label{e:perc}
\gamma^2 \lesssim_{d,p}   \rel \lesssim t_{\mathrm{evolving-sets}}(1/4) \lesssim_{d,p} \gamma^2. \end{equation}
Consequentially, $\whp$ (uniformly in $k$)
\begin{equation}
\label{e:GCEXmix}
\mixex \asymp_{d,p} \gamma^2  \log (k+1).
\end{equation}

\end{proposition}
\begin{remark}
In the setup of  \eqref{e:VT} the bound obtained on $ t_{\mathrm{mix}}^{(\infty)}$   in \cite{Lyonsev} via the spectral measure is often better than the one obtained via $t_{\mathrm{evolving-sets}}$. \end{remark}
\begin{proof}
We first note that \eqref{e:VTEXmix}  and \eqref{e:GCEXmix} follow by combining \eqref{e:VT} and  \eqref{e:perc} with \eqref{e:mixproptotsp}.  

The first inequality in \eqref{e:VT} was discussed above. The corresponding bound in \eqref{e:perc} is obtained by considering the same $h$ as in \eqref{e:h} (noting that the size of any ball of radius $\lfloor \gamma/4 \rfloor$ in the giant component $\whp$ has volume  comparable to the total number of vertices). The middle inequality in \eqref{e:VT} and \eqref{e:perc}  follows from  \eqref{e:spb1} and \eqref{e:mixrel}. The last inequality in \eqref{e:perc} is taken from \cite{GC}.  

   The proof of the last inequality in \eqref{e:VT} follows by plugging in \eqref{e:profiles} the estimate $\Phi^{-2}(\delta) \lesssim d^2 \gamma^2 (\sfrac{2 \delta}{c} )^{2/a}$, which can be derived via \eqref{e:isoVT}.
\end{proof}
\begin{proposition}
\label{p:expamples}
If $G$ is a $d$-regular vertex-transitive graph of size $n$ as in Example (i) (resp.\ (ii)) then \[r_*(c_{1.1}) \le C  d^2(\log \log n)^3  \quad (\text{resp.}\ r_*(c_{1.1}) \le C  d^2(\log \log n)^{1+\sfrac{2}{\alpha}} ).\]
\end{proposition}
\begin{proof}
As above, use \eqref{e:isoVT} to bound $\Phi^{-2}(\delta) $ for all $\delta \le 4 (\log n) / (c_{1.1}n) $. In the setup of Example (i) \eqref{e:isoVT} yields that $\Phi^{-2}(\delta) \lesssim [d \log (\delta n)  ]^2 $ and in that of Example (ii) that $\Phi^{-2}(\delta) \lesssim d^{2}[ \log (n \delta )]^{2/\alpha} $. The assertion of the proposition now  follows from \eqref{e:spb1} with $\eps= c_{1.1}n/\log n   $.  
\end{proof}
\subsection{The hypercube and product graphs}
\label{s:hypercube}
We now consider the hypercube $\{\pm 1\}^d$. We consider the case that each edge has rate $1/d$. Then $\rel =\sfrac{d}{2}=2/c_{\mathrm{LS}}$ (see \cite{diaconis}). By Proposition \ref{p:specLS} it is easy to verify that $t_{\mathrm{sp}}(\sfrac{1}{2}) \lesssim d \log d \asymp \mix $. \textcolor{black}{By \eqref{e:rLS}} $r_*(c_{1.1}) \lesssim \log d \ll \rel $. By Theorem \ref{thm:main1} in conjunction with Corollaries \ref{cor:1.9} and  \ref{cor:1.11}  we get that   $\mixex \asymp d \log (dk)$ \textcolor{black}{(for  $k \le d $ we use Theorem \ref{thm:lower}   to argue that $\mixex \gtrsim d \log d$).} 
\begin{definition}
\label{def:productchains}
The Cartesian product $G_1 \times G_2=(V',E')$ of two graphs $G_i=(V_i,E_i) $ is defined via $V':=V_{1}\times V_2 $ and $E':=\{\{(v_{1},v_{2}), (u_{1},u_{2})\}:v_1=u_1 \in V_1 \text{ and }u_1u_2 \in E_2, \text{ or vice-versa} \} $.
For a graph $G=(V,E)$ we denote the $d$-fold self (Cartesian) product of $G$ with itself by $G_{\otimes d}=(V^n,E(G_{\otimes n}))$. That is $G_{\otimes d}=G_{\otimes (d-1)} \times G=G \times \cdots \times G $.   
\end{definition}
\textcolor{black}{Note that the $d$-dim hypercube is the $d$-fold self-product of the complete graph on two vertices with itself. Consider the case that $G$ is $d_G$-regular. Then $G_{\otimes d}$ is $d \times d_G$ regular. Consider $\mathrm{EX}(k)$ on  $G_{\otimes d}$ in which each edge rings at rate $\frac{1}{d  d_G} $. We extend the above argument and show that    $\mixex \asymp d \log (dk)$ (all asymptotic notation here is as $d \to \infty$ for a fixed $G$ and some implicit constants may  depend on $G$). We note that the analysis below can easily be extended to the case of  $\mathrm{EX}(k)$ on  $G_1 \times \cdots \times G_{d}$ with the $G_i$'s being of uniformly bounded size. The regularity condition can be lifted as well.  Indeed if the $G_i$'s are of uniformly bounded size, then all vertices of   $G_1 \times \cdots \times G_{d}$ are of degree proportional to $d$. As explained in Appendix \ref{s:relax} our analysis can be extended to cover this case. 
}

\textcolor{black}{ A general result about the log-Sobolev constant of a product chain asserts that $c_{\mathrm{LS}}(G_{\otimes d})$ the log-Sobolev constant of the random walk on $G_{\otimes d}$ (with the above rates) is $c_{\mathrm{LS}}(G)/d $, where $c_{\mathrm{LS}}(G)$ is the log-Sobolev constant of a random walk on $G$, and likewise $\rel(G_{\otimes d})=d\rel(G) $ (see \cite{diaconis}). By Proposition \ref{p:specLS} we get that   $t_{\mathrm{sp}}^{G_{\otimes d}}(\sfrac{1}{2}) \lesssim \frac{ d \log (d \log |G|)}{c_{\mathrm{LS}}(G)}$. \textcolor{black}{By \eqref{e:rLS}}  $r_*(c_{1.1}) \lesssim \frac{ \log (d \log |G| )  }{c_{\mathrm{LS}}(G) \log |G| } \ll \rel(G_{\otimes d}) $. Finally, we have that  $  \mix(G_{\otimes d})=\frac{\rel(G)}{2}d \log d (1 \pm o(1)) $ \cite[Theorem 20.7]{levin}. In particular, we have that $t_{\mathrm{sp}}^{G_{\otimes d}}(\sfrac{1}{2}) \lesssim \mix(G_{\otimes d}) $. As before, by Theorem \ref{thm:main1} in conjunction with Corollaries \ref{cor:1.9} and  \ref{cor:1.11} we get that   $\mixex \asymp d \log (dk)$, as claimed.}

\appendix      \section{Proof of Proposition~7.17}
\label{s:upperhard}
We  consider the case  $k \le \sqrt{n} $, where the duration of a round of the chameleon process, starting with $r$ red particles such that $r \wedge (n-k+1 -r) \in (2^{i-1},2^{i}] $ is $L(r)=L_i$ as defined in \eqref{e:L(r)def}.
By \eqref{e:spb0} and the fact that $\Lambda(\eps) $ is non-decreasing in $\eps$ we obtain:

\begin{lemma}
\label{l:L1atleast1}
For all $\eps \in (0,1/2)$ we have that $\Lambda(\eps) \le -2\min_{x}\cL(x,x)$. In particular, in our setup $\Lambda(\eps) \le 2$ for all $\eps$ and so  $L_i \le( C_{\mathrm{round}} +2)/\Lambda(C_{\mathrm{profile}}2^{i}/n) $.
\end{lemma} 

While we are really interested in studying the process $(\ink_t)_{t \ge 0} $ (conditioned on $\Fill$), it is more convenient to study the related process
 $(\widehat Y_t)_{t \in \R_+} $ on $[n-k+1]$ which is defined by the following rule. Whenever it reaches state $r$ it stays put for $L(r) $ time units before making a step according to $\hat P $, the transition matrix of $\mathbf{Y}:= (Y_i)_{i \in \Z_+}$ (defined in Appendix~\ref{A:2}).

Recall that in this setup, each burn-in period has duration $t_{\mathrm{mix}}^{(\infty)}(\hat c/k)$, where $\hat c $ is some absolute constant (and again, the process starts with an initial burn-in period). Let $\mathrm{BIP} $ be the set of all times which are part of a burn-in period of the chameleon process. For all $s \ge 0$ let $t(s):=\inf\{t \notin \mathrm{BIP}: t-j(t)t_{\mathrm{mix}}^{(\infty)}(\hat c/k)=s  \} $, where $j(t)$ is the number of burn-in periods by time $t$. Then  $(\widehat Y_s)_{s \in \R_+} $  has the same distribution as that of    $(\ink_{t(s)})_{s \in \R_+} $ conditioned on $\Fill$. Since typically $s-t(s) \ll s$, we may translate estimates concerning   $(\widehat Y_s)_{s \in \R_+} $    to ones concerning $(\ink_t)_{t \ge 0}$.
Before diving into the analysis of $\mathbf{\widehat{Y}}:=(\widehat Y_t)_{t \ge 0} $ we need the following simple proposition concerning $\mathbf{Y}$.

 Let $\hat \ell := \lceil \log_{2} (n-k+1)  \rceil-1 $  and $m:=\lceil (n-k+1) /2 \rceil$. Our strategy is to decompose the process $\ink_t $ given $\Fill $ into three stages: (1) The time until it hits $[m-1]^{\complement}$, (2) the additional time from that moment until it never goes below $m$, and (3) the remaining time. The idea is that the process viewed at stage (3) is like $(\ink_t:t \ge 0) $ started above $m-1$, conditioned on hitting $n-k+1$ before $[m-1]$. A similar super-martingale as in Lemma \ref{l:expsubmart} can be used to study this process, with the crucial key difference that now we do not pick up a factor of $\sqrt{n-k+1}$ (as now $I_i \ge \half $). It remains to find bounds $t_i $ such that the probability that the duration of stage $i \in \{1,2\} $ is more that $t_i$ is $o(\eps/k)$. This is done by first showing that for the chain $\mathbf{Y} $ various relevant quantities have uniform exponential tails, and then translating this into corresponding statements about $\mathbf{\widehat{Y}} $.

For $i \le \hat \ell  $ let
\begin{equation}
\label{e:Tnotation}
\begin{split}
&T_{i}^{\uparrow}:=\inf \{j:Y_{j}\ge 2^i \wedge m \} = \text{The hitting time of } [(  2^i \wedge m )- 1]^{\complement}
\\ & T_{[m-1]}:=\inf\{j:Y_j < m \}= \text{The hitting time of } [m- 1], 
\\ & S:=\inf\{j: \min_{s:s \ge j} Y_s \ge m\}-T_{\hat \ell}^{\uparrow}=\text{Time between the first visit to } \\ & [m - 1]^{\complement} \text{ and the time following the last visit to }[m-1],
\\ & \mathrm{Cross}:=|\{i:Y_{i+1}<m\le Y_i  \}| = \text{number of down-crossings below $m$}.
\end{split}
\end{equation}

\begin{proposition}
\label{p:refined}There exist absolute constants $0<c_i<1<C_{i} $ (for $i \in [6]$) such that
\begin{itemize} 
\Item[(i)] 
\begin{equation}
 \label{e:Tiexptail}
\forall s , \quad \max_{i \le \hat \ell } \max_{r \in [ 2^{i-1},2^i ) }\Pr_r[T_i^{\uparrow}>s] \le C_1 \exp(-c_{1}s). \end{equation}
Hence for some $c_{6} \in (0,c_{1}/2) $, for all $\gamma \in (0,c_{6}) $ we have that
   \begin{equation}
 \label{e:Tiexpmom}
  \max_{i \le \hat \ell } \max_{r \in [ 2^{i-1},2^i ) } \E_r [ \exp(\gamma T_{i}^{\uparrow} ) ] \le \exp(C_6 \gamma ). \end{equation}
\item[(ii)]  Let $I:=[m,\sfrac{3}{2}m ]$. Then 
 \begin{equation}
 \label{e:cross}
 \forall s , \quad  \max_{r \in I}\Pr_r[\mathrm{Cross}>s] \le C_2 \exp(-c_{2}s). \end{equation}
 \begin{equation}
 \label{e:Tdown}
 \forall s , \quad  \max_{r \in I}\Pr_r[T_{[m-1]}  \mid \mathrm{Cross}\ge 1] \le C_3 \exp(-c_{3}s). \end{equation}
 \Item[(iii)]
 \begin{equation}
 \label{e:Sbound}
 \forall s , \quad  \max_{r \in I}\Pr_r[S   \ge s] \le C_4 \exp(-c_{4}s). \end{equation}
 \item[(iv)] For all $r \in I$, conditioned on  $Y_0=r $ and $S=0$ we have that $c_5^{-i}(1-\frac{Y_i}{n-k+1})$ is a super-martingale ($c_5=c_5(\alpha,p)$, where  $\alpha$ is as in the definition of $\Delta(r)$). 
 \end{itemize}  
\end{proposition}
\begin{proof}
We first prove \eqref{e:Tiexptail}. Let $U_t:=|\{j\le t: Y_j>Y_{j-1} \}| $ and $D_t=|\{j \le t: Y_j<Y_{j-1} \}|$. 
 Up to a rounding error (resulting from the ceiling in the definition of $\Delta(r)$), whenever the size of $Y_i$ changes, it is multiplied by a factor of either $1+\alpha$ or $1-\alpha$. Using the fact that $(1+ \alpha )^{ 1 + \alpha }(1-\alpha)^{1 - \alpha}>1 $ for all $\alpha \in (0,1)$ (and so also $(1+ \alpha )^{ p\frac{ 1 + \alpha}{2} }(1-\alpha)^{p\frac{ 1 - \alpha}{2}} >1$), ignoring the rounding error we get that there exists some $\eps>0$ and $C_{\eps}$ such that for all $i \le \hat \ell$ and all $r \in [ 2^{i-1},2^i )$,  if $s  \ge  C_{\epsilon}$, $U_s \ge ps(\frac{ 1 + \alpha}{2} - \eps ) $ and $D_s> ps(\frac{ 1 - \alpha}{2} + \eps )$ then  $T_{i}^{\uparrow} \le s $. It is easy to verify that this implies \eqref{e:Tiexptail}, as the probability that this fails for some fixed $s$ decays exponentially in $s$ (uniformly). To deal with the rounding error,  one can control its possible effects  whenever $Y_i$ is at least some constant $C \in \N $. Thus by the above reasoning $\max_{i \le \hat \ell } \max_{r \in [ 2^{i-1},2^i ) }\Pr_r[|\{t \le T_i^{\uparrow}:Y_t \ge C\}| >s] \le C' e^{-ct}  $ for all $s$. Hence,  it suffices to argue that $\max_{i \le \hat \ell } \max_{r \in [ 2^{i-1},2^i ) }\Pr_r[|\{t \le T_i^{\uparrow}:Y_t < C\}| >s] \le C' e^{-ct}  $ for all $s$. This follows from the fact that $$\max_{i \le \hat \ell } \max_{r \in [ 2^{i-1},2^i ) }\Pr_r[|\{t \le T_i^{\uparrow}:Y_{t+1} < C \le Y_t \}| >s] \le C' e^{-ct}  $$ for all $s$. We leave the details as an exercise. 

Observe that \eqref{e:Tiexpmom} follows easily from \eqref{e:Tiexptail}. We now prove \eqref{e:Sbound}. It suffices to show that $\max_{r \in I}\E[z^{S}]<\infty  $ for some $z>1$. We may write $S=\sum_{i=1}^{\mathrm{Cross}}K_i $, where $K_i$ is the time the chain spends above $m$ during its $i$th epoch above $m$.  Noting that by part (ii)  $M(z):=\max_{r' \in I} \E_{r'}[z^{ K_1}  ] $ satisfies $\lim_{z \to 1^+}M(z)=1 $, and $\E [z^{\mathrm{Cross}}]< \infty $ for all $0<z \le z_0>1$. As $\alpha \in (0,1/2)$ it follows that if $Y_i<m<Y_{i+1} $ then $Y_{i+1} \in I$.  Hence by the strong Markov property, for some $z >1$  \[\max_{r \in I}\E[z^{S}] \le \max_{r \in I} \E_{r} [M(z)^{\mathrm{Cross} } ]  < \infty. \] 
The proof of part (iv) is analogous to that of Lemma \ref{l:expsubmart} and is thus omitted.

 Inequality \eqref{e:cross} follows from the fact that for every fixed $\eps>0$ with positive probability we have that  $U_s \ge \lceil ps(\frac{ 1 + \alpha}{2} - \eps ) \rceil $ and $D_s> \lfloor ps(\frac{ 1 - \alpha}{2} + \eps ) \rfloor $ for all $s>0$, and this probability is uniform in $r \in n-k+1$. Thus $ a_{*}:=\min_{r \ge m }\Pr_x[\mathrm{Cross}=0]$ is bounded from below (uniformly in $n-k+1$) and by the strong Markov property $\mathrm{Cross} $ is stochastically dominated by the (shifted) Geometric distribution of parameter $a_{*} $.

Finally, \eqref{e:Tdown} follows by considering the Doob's transform of $\mathbf{Y}$ obtained by conditioning on $T_{[m-1]}< \infty$. An elementary calculation shows that under this conditioning, up to time  $T_{[m-1]}< \infty$ the chain has transition probabilities $Q$ satisfying $Q(r,r-\Delta(r)) < Q(r,r+\Delta(r)) $ while for $r \in I':= \{\sfrac{3}{2}m,\ldots,n-k+1 -1 \}$ we have  $Q(r,r-\Delta(r)) < c_{\alpha,p}' Q(r,r+\Delta(r))$ for some $c_{\alpha,p}' \in (0,1)$ (independent of $n-k+1$). We may write $T_{[m-1]}:=\sum_{j=1}^{\mathrm{\widehat{Cross}}} F_j+F'_j $, where $\mathrm{\widehat{Cross}} $ is the number of times the chain enters the interval $I' $ and then leaves it, $F_i$ (resp.\ $F'_i$) is the time it spends in $I$ (resp.\ $I'$) during the $i$th epoch. As above, it is not hard to verify that $\mathrm{\widehat{Cross}}$, the $F_i$'s and the $F'_i$'s have uniformly exponentially decaying tails. This implies the assertion of part (iv) in a similar fashion to the derivation of part (iii) from part (ii). We leave the details as an exercise.   
\end{proof}

\begin{proposition}
\label{p:hatY}
Let $\tau:=\inf\{t: \min_{s:s \ge t} \widehat Y_s \ge m \}$. Then (starting from $\widehat Y_0=0$)
\begin{align}
\label{e:tau1}
\E[1-\sfrac{\widehat Y_{s+t}}{n-k+1} ] \le \Pr[\tau \ge t]+C \exp(- c s/\rel ),\\
\label{e:tau2}
\Pr[\tau \ge C t_{\mathrm{sp}}(\sfrac{\eps}{4k}) ] \le \sfrac{\eps}{16k^{2}}.
\end{align}
\end{proposition}
\begin{proof}
Observe that \eqref{e:tau1} is a direct consequence of part (iv) of Proposition \ref{p:refined}. We now prove \eqref{e:tau2}. We use the same notation as in \eqref{e:Tnotation}, but now for the chain $\mathbf{\widehat{Y}}$. In this notation  $\tau =S+ \sum_{i \in [\hat \ell]}T_i^{\uparrow}$. By \eqref{e:Sbound}, for all $s \ge 0$, $\Pr[S   \ge s \rel ] \le C_4 \exp(-c_{4}s)$. 
Hence
\begin{equation}
 \label{e:Sbound'}
 \quad \Pr[S   \ge C' \rel \log (k/\eps) ] \le \sfrac{\eps}{32k^{2}}. \end{equation}
By \eqref{e:Tiexpmom}, there exist  $c \in (0,1) $ and $C_6$ such that for all $\gamma \le c / \rel $ and all $i \le \hat \ell $ we have
   \begin{equation}
 \label{e:Tiexpmom'}
 \max_{r \in [ 2^{i-1},2^i ) } \E_r [ \exp(\gamma T_{i}^{\uparrow} ) ] \le \exp( C_6 \gamma L_i  ), \end{equation}
where $L_i \le ( C_{\mathrm{round}} +2)/\Lambda(C_{\mathrm{profile}}2^{i}/n)$ by Lemma~\ref{l:L1atleast1}.   Thus,
   \begin{align*}
   \E[e^{\gamma(\tau-S)}]&=\E [ \exp(\gamma \sum_{i \in [\hat \ell]}T_i^{\uparrow}) ] \le \exp( C_6 \gamma \sum_{i  \in [\hat \ell] } L_i  ) \\&\le \exp(C_7 \gamma t_{\mathrm{sp}}(\sfrac{1}{4k})). 
\end{align*} 
Picking $C_{8} = 6(C_7 \vee 1) /c$ and $\gamma = c / \rel$ we get that
\begin{equation*}
\begin{split}
\Pr[ \tau - S \ge C_8 t_{\mathrm{sp}}(\sfrac{\eps}{4k})  ] & \le \E[e^{\gamma(\tau-S)}]e^{-\gamma C_8 t_{\mathrm{sp}}(\sfrac{\eps}{4k})} \\ & \le e^{5(C_7 \vee 1)t_{\mathrm{sp}}(\sfrac{\eps}{4k})/ \rel } \le \sfrac{\eps}{32k^{2}}, 
\end{split}
\end{equation*} 
 where we have used the fact that $t_{\mathrm{sp}}(\sfrac{\eps}{4k}) \ge  \rel \log(\sfrac{2k}{\eps})$. This, in conjunction with \eqref{e:Sbound'}, concludes the proof. 
\end{proof}
\emph{Proof of Proposition \ref{p:chambound3}.} Let $s \in [k,n^3] $. Let $M \ge 1$ be some absolute constant to be determined shortly.  Recall that   $\hat t(s):=t_{\mathrm{mix}}^{(\infty)}(\hat c /k )+q $, where $q=q(s,M):=Mt_{\mathrm{sp}}(\sfrac{1}{4s} ) $. By Proposition \ref{p:hatY} we may pick $M$ such that $q \ge C [t_{\mathrm{sp}}(\sfrac{1}{4s})+ \rel \log (\sfrac{1}{ 4s})]$, and so  
\[\hE_{(\mathbf{w},y)}[ 1 - \ink_{\hat t(s)} / (n-k+1) ] \le \E[1-\sfrac{\widehat Y_{q}}{n-k+1} ]+ \hP_{(\mathbf{w},y)}[j(\hat t(s)) \ge 2 ] \le s^{-1} + \hP_{(\mathbf{w},y)}[j(\hat t(s)) \ge 2 ]  .  \]
Finally, $ \hP_{(\mathbf{w},y)}[j(\hat t(s)) \ge 2] \le (n-k+1) \Pr_{(\mathbf{w},y)}[j(\hat t(s)) \ge 2]   \le (n-k+1) q \max_i \beta_{i}(\alpha) $, by a simple union bound (over all rounds by time $\hat t(s)$), using the fact that the duration of each round is at least $1$ time unit. \qed 

 \section{Relaxing the assumptions}
\label{s:relax} 
\subsection{Relaxing the regularity assumption}
The assumption that $G$ is a regular graph can be relaxed. Instead we may assume that if $\{x,y\} \in E$ then $\sfrac{\deg(x)}{\deg (y)} \le C_{\mathrm{deg-ratio}} $ (\textit{i.e.}, adjacent vertices have comparable degrees). When $G$ is not regular we take $r_e \equiv 1$ (note that if all degrees are within factor, say $2$, from $d$, this chain evolves roughly $d$ times faster than in the case when $r_e \equiv 1/d $). In this case, Theorems \ref{thm:main1}-\ref{thm:main3} hold with the following modifications: 
\begin{itemize}
\item[(1)]
The bounds should include an additional $C_{\mathrm{deg-ratio}}$ multiplicative term and an additional additive term  of order $\sfrac{1}{d_\mathrm{min}} \log (n/\eps)$, where $d_\mathrm{min}:=\min_{v \in V} \deg(v)$ is the minimal degree. This extra additive term comes from taking the round length to be $T+1/d_\mathrm{min}$ instead of $T+1$ since the expected time it takes for some edge connected to a vertex to ring is at most $1/d_\mathrm{min}$. Note that $T$ is defined as before (for each of the cases) except the definition of $s_\ast(\epsilon)$ is modified: see point (3).
\item[(2)]  The assumption  in \eqref{e:main3} should be changed to $d_\mathrm{min} \ge C_{\mathrm{deg}}  \log_{n/k} n$.
\item[(3)] The definition of $s_*(\epsilon) $ should be changed to  $$s_{*}(\epsilon):=\inf \{t:\max_{v \in V}P_t(v,v)-1/n \le \sfrac{\epsilon}{d_{\mathrm{max}}t_\ast(\epsilon)} \}, $$ where $d_{\mathrm{max}}=\max_{x \in V }\deg (x) $ is the maximal degree. We note that in this case \eqref{e:ts_*2} may fail. In this case, it seems that  the method of \cite{sensitivity} can be adapted to show that        $P_t(v,v)-\sfrac{1}{n} \lesssim  (d_\mathrm{min}t+1)^{-1/2} \wedge \exp(-t/\rel)   $ for all $t$ (unfortunately, we could not find a reference that treats the case $r_e \equiv 1 $, as opposed to $r(x,y)=\sfrac{1}{\deg (x)} $). If this is correct, then one gets that in this setup the following analog of \eqref{e:ts_*2}  \[r_*(\epsilon) \lesssim \mix^{(\infty)}(\sfrac{d_\mathrm{min}\epsilon n}{d_{\mathrm{max}}(\log n)^2}) \lesssim_\epsilon \rel \log ( \sfrac{d_{\mathrm{max}}}{d_\mathrm{min}}\log n ) \wedge d_{\mathrm{max}}^2 (\log n)^4/d_\mathrm{min}^3.\] 
\end{itemize}
       If the stronger assumption $ \sfrac{d_{\mathrm{max}}}{d_\mathrm{min}} \le C $ holds, then  the method of \cite{sensitivity} can indeed be adapted to show that $r_*(\epsilon) \lesssim_{C,\epsilon} \sfrac{1}{d_\mathrm{min}}(\log n)^4 \wedge \rel \log \log n   $. Moreover, in this case $\sfrac{1}{d_\mathrm{min}} \lesssim \rel $ and $\sfrac{1}{d_\mathrm{min}} \log n \lesssim t_{\mathrm{sp}}(\sfrac{1}{2})$, which means that the aforementioned additional additive term from (1) does not increase the order of our bounds.

We  strongly believe that the regularity and equal rates assumptions may be replaced by the condition that $r(x):= \sum_{e:e \ni x }r_e $ may vary only by a constant factor as a function of $x$, in order to obtain the same bounds, with $\min_{x \in V}r(x)$ playing the role of $d_\mathrm{min}$ above (apart from in the condition $d_\mathrm{min} \ge C_{\mathrm{deg}} \log_{n/k} n$ for \eqref{e:main3}). 

\subsection{Relaxing the requirement $d \ge C_{\mathrm{deg}}  \log_{n/k} n $ in \eqref{e:main3}}

 The following is useful in extending \eqref{e:main3} to  regular graphs of degree $d \asymp (\log n)^{\Omega (1)}$ in which each vertex belongs to a bounded number of short cycles. 
Denote by $S_i(v)$ the collection of vertices of distance exactly $i$ from $v$.
Then there exists an absolute constant $C_{\mathrm{deg}}>0$ such that  \eqref{e:main3} holds (with $C_{1.2}$ in \eqref{e:main3} depending on  $C_{\mathrm{distance}},C_{\mathrm{tree-excess}},C_{\mathrm{\# parents}}   $), if for some constants $C_{\mathrm{distance}},C_{\mathrm{tree-excess}},C_{\mathrm{\# parents}} \in \N $ for all $v\in V$ (i)-(iii) below hold  for some $i=i(v) \le C_{\mathrm{distance}}$:
\begin{itemize}
\item[(i)] $|S_i(v)| > C_{\mathrm{deg}} \log_{n/k} n$.
\item[(ii)] The ball $B_{i-1}(v):=\cup_{j=0}^{i-1}S_j(v) $ of radius $i-1$ centered at $v$ satisfies that the graph obtained by deleting from the induced graph on $B_{i-1}(v)$ all edges connecting two vertices in $S_{i-1}(v)$, has \emph{tree excess} at most $C_{\mathrm{tree-excess}}$. (The tree excess of a graph is the minimal number of edges whose deletion turns the graph into a tree.) \item[(iii)] Each vertex in $S_i(v)$ has at most $C_{\mathrm{\# parents}}$ neighbours in $S_{i-1}(v)$. \end{itemize} 
For instance, for the hypercube we can take $C_{\mathrm{distance}}=2=C_{\mathrm{\# parents}}$ and $C_{\mathrm{tree-excess}}=0$. For a random $d$-regular graph with  $d \asymp (\log n)^{a}$ for some $a \in (0,\infty)$ the above holds $\whp$ with $C_{\mathrm{tree-excess}}=1$, $C_{\mathrm{\# parents}}=2$ for some $C_{\mathrm{distance}} $ depending on  $a $.

We sketch the adaptations needed to verify this assertion. Assume that (i)-(iii) hold for vertex $v$ with constants $C_{\mathrm{distance}},C_{\mathrm{tree-excess}},C_{\mathrm{\# parents}}    \in \N $. Assume that at the current time, which we think of as time $0$, we have a red particle at vertex $v$ and that $i=i(v) \le C_{\mathrm{distance}} $ satisfies that $|S_i(v)| \ge C_{\mathrm{deg}}  \log_{n/k} n $ and that at least an $\eps $-fraction of the vertices in  $S_i(v) $ are occupied by white particles.
It is not hard to see that it is possible to modify the proof of \eqref{e:main3} and extend it to the above setup as long as in the above scenario there exists some constant $p=p(\eps,C_{\mathrm{distance}},C_{\mathrm{tree-excess}},C_{\mathrm{\# parents}}) $ such that with probability at least $p$ within one time unit the red particle reaches $S_{i-1}(v)$ and then an edge connecting it to a white particle rings, while the white particle had not moved prior to that. 

As the probability of a white particle not moving in one time unit is bounded from below it suffices to show that there exist some constants $\hat p= \hat p(\eps,C_{\mathrm{distance}},C_{\mathrm{tree-excess}},C_{\mathrm{\# parents}}) $ and $\delta=\delta(C_{\mathrm{tree-excess}},C_{\mathrm{\# parents}}) $ such that with probability at least $\hat p$, within one time unit the red particle reaches $S_{i-1}(v)$ and hits it at some vertex which had at least a $\delta \eps $-fraction of its neighbours white at time 0. To see that this is indeed the case, observe that by requirement (iii) a point on $S_{i-1}(v) $ picked uniformly at random has probability bounded from below of having at least some  $\delta' \eps $-fraction of its neighbours white at time 0 (for some $\delta'=\delta'(C_{\mathrm{\# parents}})>0$). Since the red particle hits $S_{i-1}(v)$ within one time unit with probability bounded from below, the claim follows once we show that its hitting distribution, conditioned on hitting  $S_{i-1}(v)$  before time 1, $\mu$ satisfies that $\max_{x,y \in S_{i-1}(v) } \sfrac{\mu(x)}{\mu(y)} \le C_1=C_1(C_{\mathrm{distance}},C_{\mathrm{tree-excess}})$. This indeed follows from requirement (ii). While this claim is intuitively obvious (e.g.\ if $C_{\mathrm{tree-excess}}=0$ then $\mu$ is the uniform distribution), we  sketch the details for the sake of completeness. The red particle has probability bounded from below (by some $q=q(C_{\mathrm{distance}})>0$) of making in one time unit $i-1 $ consecutive steps away from $v$ until reaching $S_{i-1}(v)$. The probability it hits a certain vertex  $u \in S_{i-1}(v)$ upon completion of its $i-1$ jump is proportional to the number of paths of length $i-1$ connecting $u$ to $v$. This number is at least 1 and is clearly bounded by some $C_2=C_2(C_{\mathrm{tree-excess}}) $ by condition (ii). 

\subsection{Proof of Corollaries \ref{cor:example}--\ref{cor:1.11}}\label{appb1}
 We start with proving \eqref{e:mixproptotsp}. Here we assume $\rel(G_m)\asymp t_\mathrm{sp}^{G_m}(\half)$ \textcolor{black}{and so recalling that under reversibility  $\sfrac{1}{c_{\mathrm{LS}}(G_m)} \lesssim \mix^{(\infty),G_m} $ \cite{diaconis} we get that}  $$\rel(G_m)\asymp t_\mathrm{sp}^{G_m}(\half)\ge t^{(\infty),G_m}_\mathrm{mix}\gtrsim \sfrac{1}{c_\mathrm{LS}(G_m)}.$$ \textcolor{black}{Thus if $\delta_m \in [  \frac{1}{k_m^{1/3}},1) $, the lower bound} \[\mix^{\mathrm{EX}(k_m),G_m}(1-\delta_m) \gtrsim  \rel(G_m)( \log (k_m +1)+\log \delta_m) \gtrsim \rel(G_m)\log (k_m +1)  \] by Theorem~\ref{thm:lower} and Proposition~\ref{p:LS} \textcolor{black}{(note that by Proposition~\ref{p:LS}  we can apply Theorem~\ref{thm:lower} with $\delta=1/8$ provided that $k_m \ge 2^8 \exp(8   \rel(G_m)c_\mathrm{LS}(G_m)) $)}. \textcolor{black}{When  $k_m \gg 1$,  we can take $\delta_m\to 0$.
} For the upper bound\textcolor{black}{,} using
 \eqref{e:ts_*2} in the first inequality we have 
 \[\frac{r_*^{G_m} (\epsilon)\log n_m}{\log \log n_m}\lesssim_\epsilon \mix^{(\infty),G_m} \le  t_{\mathrm{sp}}^{G_m}(\half) \asymp \rel(G_{m}) ,\]
 which gives $ r_*^{G_m} (\epsilon)\ll\rel(G_m)$.
 Hence by Theorems \ref{thm:main1} and \ref{thm:main3}  we obtain \[\mix^{\mathrm{EX}(k_m),G_m} (\delta_m)\lesssim \rel(G_m)\left( \log (k_m +1)+\log (1/\delta_m)\right).\]   Combining the upper and lower bounds shows the claimed precutoff. The equality $\mix^{\mathrm{RW}(k_m),G_m} \asymp \rel(G_m) \log (k_m +1)=:b_m  $ in \eqref{e:mixproptotsp} is obtained by recalling that by \eqref{e:kvs1} $\mix^{\mathrm{RW}(k_m),G_m} \asymp \mix^{\RW(1),G_m}(\sfrac{1}{4k_{m}})$,  and further noting that $ \mix^{\RW(1),G_m}(\sfrac{1}{4k_{m}}) \ge   b_m   $ by \eqref{e:mixrel}, while as $\rel (G_{m})\asymp t_{\mathrm{sp}}^{G_m}(\sfrac{1}{2})$ (used in the last ineq.)  \begin{equation}
 \label{e:1/4km}
 \mix^{\RW(1),G_m}(\sfrac{1}{4k_{m}}) \le t_{\mathrm{sp}}^{G_{m}}(\sfrac{1}{4k_{m}}) \lesssim t_{\mathrm{sp}}^{G_{m}}(\sfrac{1}{2})+b_m \lesssim b_m. \end{equation}
 
We now prove \eqref{e:mixproptotsp2}. 
 Here we assume that  $\mix^{\RW(1),G_m} \asymp t^{G_{m}}_{\mathrm{sp}}(\sfrac{1}{2})  $.  The claim $\mix^{\mathrm{RW}(k_m),G_m} \asymp \mix^{\RW(1),G_m}+b_{m}  $ follows from  $ \mix^{\RW(1),G_m}(\sfrac{1}{4k_{m}}) \ge  \mix^{\RW(1),G_m}   \vee b_m   $ in conjunction with \textcolor{black}{the first two inequalities in} \eqref{e:1/4km}. The upper bound  \begin{align}\label{e:75}\mix^{\mathrm{EX}(k_m),G_m}(\delta_m)\lesssim \mix^{\RW(1),G_m} +\rel(G_m)(\log(k_m+1)-\log\delta_m)\end{align} follows by Theorem \ref{thm:main3}. For the lower bound we have $\mix^{\mathrm{EX}(k_m),G_m} \gtrsim \mix^{\RW(1),G_m}$ by Proposition \ref{p:mixexkatleastmixrw1}. Thus $\mix^{\mathrm{EX}(k_m),G_m} \gtrsim  \mix^{\RW(1),G_m}+b_m$, if $b_m \le C\mix^{\RW(1),G_m} $ for some absolute constant $C>0$ to be determined soon.  If $b_m \ge C \mix^{\RW(1),G_m}$, then as   $\mix^{\RW(1),G_m} \asymp t^{G_{m}}_{\mathrm{sp}}(\sfrac{1}{2})  $, we have that $b_m \ge C  c t^{G_{m}}_{\mathrm{sp}}(\sfrac{1}{2}) \ge \sfrac{ C  c'}{c_{\mathrm{LS}}}$, (\textit{i.e.}, $\sfrac{\mathrm{gap}}{c_{\mathrm{LS}}} \le \sfrac{1}{Cc'} \log[(1+k_m)^{}] $) and so by Theorem \ref{thm:lower} in conjunction with Proposition \ref{p:LS}    $ \mix^{\mathrm{EX}(k_m),G_m} \gtrsim  b_m  $, provided that $C \ge  16/c'$.  In particular, under \textcolor{black}{the} assumption $b_m\gg\mix^{\RW(1),G_m}  $ we have that \[ \mix^{\mathrm{EX}(k_m),G_m}(1-\delta_m)\gtrsim  \rel(G_m)( \log (k_m +1)+\log \delta_m).\]Combining with the upper bound \eqref{e:75} gives the claimed precutoff.
 
 Statement \eqref{e:cor1.10} follows directly from Proposition~\ref{p:mixexkatleastmixrw1}, Theorem~\ref{thm:lower} and Proposition~\ref{p:LS}.

For \eqref{e:13}, we obtain the lower bound in the first relation directly from Theorem~\ref{thm:lower} and Proposition~\ref{p:LS} (as above, we can apply  Theorem~\ref{thm:lower} with $\delta=1/8$ provided that $k_m \ge 2^8 \exp(8   \rel(G_m)c_\mathrm{LS}(G_m)) $). This gives that for any $\delta_m\to0$, \[\mix^{\mathrm{EX}(k_m),G_m}(1-\delta_m) \gtrsim  \rel(G_m)( \log (k_m +1)+\log \delta_m).\]On the other hand by Theorem\textcolor{black}{s}~\ref{thm:main1} and \ref{thm:main3}, \textcolor{black}{Lemma} \ref{lem:r}  Proposition~\ref{p:specLS} we have for any $\delta_m\to0$, the upper bound
\[
 \mix^{\mathrm{EX}(k_m),G_m}(\delta_m)\lesssim \rel(G_m)(\log (k_m+1)-\log\delta_m).
 \]
Combining these two bounds gives the claimed precutoff. Next, the equality $\mix^{\mathrm{RW}(k_m),G_m} \asymp b_m$ follows by combining Proposition~\ref{p:specLS} with the arguments used in \eqref{e:1/4km} and the preceding paragraph.
 
 For \eqref{e:14}, by Theorem~\ref{thm:lower} and Proposition~\ref{p:LS} we have for any $\delta_m\to0$, the lower bound \[\mix^{\mathrm{EX}(k_m),G_m}(1-\delta_m) \gtrsim  \rel(G_m)( \log (k_m +1)+\log \delta_m).\] On the other hand, by Theorems~\ref{thm:main1} and~\ref{thm:main3}, Proposition~\ref{p:specLS} and Lemma~\ref{lem:r} we have for any $\delta_m\to0$, the upper bound
 \[
 \mix^{\mathrm{EX}(k_m),G_m}(\delta_m)\lesssim \rel(G_m)(\log\log n_m+\log(k_m+1)-\log\delta_m).
 \]
 Combining the upper and lower bounds gives the claimed precutoff.

 \section{Technical Proofs}

 \subsection{Proof of Proposition~\ref{p:inkatb}}\label{A:expink}
Let $\{\bar\tau_n\}_{n\in\mathbb N}$ denote the update times of the chameleon process $\{M_t\}_{t\ge0}$; thus each $\bar\tau_n$ is either an incident time of the Poisson process $\Lambda$, or a depinking time (of the form $it_\mathrm{round}$ with $i\in\mathbb N$).  For each $j\in\mathbb{N}$, consider a process $\{M_t^j\}_{t\ge0}$
  which is identical to $\{M_t\}_{t\ge0}$ for all $t<\bar\tau_j$ but
  evolves as the interchange process (\textit{i.e.}, with no further
  recolourings) for all $t\ge\bar\tau_j$. More formally, for all $t\ge
  \bar\tau_j$,
  \[
  M_t^j=(I_{(\bar\tau_j,t]}({\bf
    z}_{\bar\tau_j}),I_{(\bar\tau_j,t]}(R_{\bar\tau_j}),I_{(\bar\tau_j,t]}(P_{\bar\tau_j}),I_{(\bar\tau_j,t]}(W_{\bar\tau_j})).
  \]
  Notice that the almost-sure limit of $\{M_t^j\}_{t\ge0}$ as
  $j\to\infty$ is the chameleon process $\{M_t\}_{t\ge0}$. As a
  result, by the dominated convergence theorem, it suffices to prove
  that for each $j\in\mathbb{N}$, $b\in V$, and $\mathbf{c}\in(V)_{k-1}$,\[
  \P\bra{x_t^\mathrm{IP}=b\,|\,{\bf
      z}_t^\mathrm{IP}=\mathbf{c}}=\E[\ink_t^j(b)\,|\,{\bf z}_t^\mathrm{IP}=\mathbf{c}],
  \]
  where $\ink_t^j(b)$ is the amount of ink at vertex
  $b$ in the process $M^j_t$.  We prove this by induction on $j$. The case $j=1$ is trivial
  since the particle initially at $x$ is the only red particle (and
  there are no pink particles).  For the inductive step we wish to
  show that 
  \begin{align}\label{eq:expink}
  \E[\ink_t^j(b)\,|\,{\bf z}_t^\mathrm{IP}=\mathbf{c}]=\E[\ink_t^{j+1}(b)\,|\,{\bf
    z}_t^\mathrm{IP}=\mathbf{c}].
  \end{align}
  For $t<\bar\tau_j$, these are equal since the two processes evolve
  identically for such times. The update at time $\bar\tau_j$ of
  process $\{M_t^{j+1}\}$ is a chameleon step and could be of two
  types: also an update of the interchange process (\textit{i.e.}, $\bar\tau_j$
  is an incident time of the Poisson process $\Lambda$), or not (\textit{i.e.}, it is a depinking time). Suppose we are in the first case and that edge $e$ rings at time $\bar\tau_j$. By the strong Markov property at time $\bar\tau_{j-1}$ we
  can construct a process $\{\tilde M_t^j\}$ which behaves exactly like $\{M_t^j\}$ except that if the particles on edge $e$ are red and white in which case it switches them if and only if $\{M_t^j\}$ does not switch them (which is decided by the coin flip) at time $\bar\tau_j$. Clearly $\tilde M_t^j$ has the same distribution as $M_t^j$ and so
   \[
 \E[\ink_{\bar\tau_j}^j(b)\,|\,{\bf
        z}_{\bar\tau_j}^\mathrm{IP}=\mathbf{c}] =\E[\widetilde{\ink}_{\bar\tau_j}^j(b)\,|\,\tilde{\bf
        z}_{\bar\tau_j}^\mathrm{IP}=\mathbf{c}]
  \] for all $b\in V$ and $\mathbf{c}\in(V)_{k-1}$, (where $\widetilde{\ink}$ is the
  ink process under $\tilde M^j$). But also we have
  \[
  \frac12\ink_{\bar\tau_j}^j(b)+\frac12\widetilde{\ink}_{\bar\tau_j}^j(b)=\ink_{\bar\tau_j}^{j+1}(b),
  \]
  for all $b\in V$ and $\mathbf{c}\in(V)_{k-1}$, and so taking a conditional expectation gives \eqref{eq:expink} in this case.

  We are left to deal with the second case, when $\bar\tau_j$ is not
  an update of the interchange process, \textit{i.e.}, $\bar\tau_j$ is a depinking time. By the strong Markov property at time $\bar\tau_{j-1}$ we
  can construct a process $\{ \overleftrightarrow{M}_t^j\}$ which behaves exactly like $\{M_t^j\}$ except that if the depinking is of type 1, then it makes the opposite colouring choice (\textit{i.e.}, if $M_t^j$ colours all pink red at time $\bar\tau_j$, then $ \overleftrightarrow{M}_t^j$ colours all pink white, and vice-versa). If the depinking is of type 2, then $\hat M_t^j$ makes the same choice of half the pink particles but switches which half is coloured red and which half white.
  
   Clearly $\overleftrightarrow{M}_t^j$ has the same distribution as $M_t^j$ and so 
   \[
 \E[\ink_{\bar\tau_j}^j(b)\,|\,{\bf
        z}_{\bar\tau_j}^\mathrm{IP}=\mathbf{c}] =\E[\overleftrightarrow{\ink}_{\bar\tau_j}^j(b)\,|\,\tilde{\bf
        z}_{\bar\tau_j}^\mathrm{IP}=\mathbf{c}]
  \] for all $b\in V$ and $\mathbf{c}\in(V)_{k-1}$ (where $\overleftrightarrow{\ink}$ is the
  ink process under $\overleftrightarrow{M}^j$). But also we have
  \[
  \half \ink_{\bar\tau_j}^j(b)+\half \overleftrightarrow{\ink}_{\bar\tau_j}^j(b)=\ink_{\bar\tau_j}^{j+1}(b),
  \]
  for each $b\in V$, and so taking a conditional expectation gives \eqref{eq:expink} in this case.
\subsection{Proof of Proposition~\ref{p:chambound}}
\label{s:p4.6}
Recall that that $\Fill:=\{\lim_{t \to \infty} \ink_{t} = n-k+1  \}$. Recall from \S\ref{SS:technical} that $\Pr[ \Fill ]=(n-k+1)^{-1} $.   

\begin{proof}[Proof of Proposition~\ref{p:chambound}]
It follows from $\BB_t=\mathbf{w}(t) $, $\Pr[ \Fill ]=(n-k+1)^{-1} $    and Lemma~\ref{L:Fill} that
\begin{equation*}
\begin{split}
& \Pr_{(\mathbf{w},y)}[(\mathbf{w}(t),U)=(\mathbf{x},z) ]=\Pr_{(\mathbf{w},y)}[(\mathrm{B}_{t},U)=(\mathbf{x},z)  ] 
 \\ &=\Pr_{(\mathbf{w},y)}[\mathrm{B}_{t} 
 =\mathbf{x},\Fill ]=\E_{(\mathbf{w},y)}[\Ind{\mathrm{B}_t=\mathbf{x}, \Fill}  ],
\end{split}
 \end{equation*}
where we have used the convention described before Proposition \ref{p:chambound} regarding $(\mathbf{w},y)$, although the $k$th co-ordinate $y$ plays no role above.  By Proposition \ref{p:inkatb}
\begin{align*}\Pr_{(\mathbf{w},y)}[(\mathbf{w}(t),y(t))=(\mathbf{x},z) ]&=\E_{(\mathbf{w},y)}[\Ind{\mathrm{B}_t=\mathbf{x}}\ink_t(z)  ] \\&\ge \E_{(\mathbf{w},y)}[\Ind{ \mathrm{B}_t=\mathbf{x}, \Fill}\ink_t(z)  ].  \end{align*}This estimate may seem wasteful. However, when averaging over $z$, it is not wasteful if we consider $t$ such that $\E[\ink_t \Ind{  \Fill^{\complement}}] \ll \E[\ink_t \Ind{  \Fill}]$, which holds \textit{e.g.} if either $\Pr[\ink_t \notin \{0,n-k+1 \}] \ll (n-k+1)^{-2}  $ or   $\hP[\max_{s: s \le t } \ink_s \le (n-k+1) /2]  \ll 1  $.

For $c \in \R$ let $c_+:=c \vee 0$. Finally, for all $(\mathbf{w},y)\in (V)_k$, $$A:=\| \cL_{(\mathbf{w}(t),y(t))}-\cL_{(\mathbf{w}(t),U)} \|_{\TV} $$ satisfies
\begin{equation}
\label{e:TVcal}
\begin{split}
& A=\sum_{(\mathbf{x},z) \in (V)_k }(\Pr_{(\mathbf{w},y)}[(\mathbf{w}(t),U)=(\mathbf{x},z) ]-\Pr_{(\mathbf{w},y)}[(\mathbf{w}(t),y(t))=(\mathbf{x},z) ])_{+} 
\\ & \le \sum_{(\mathbf{x},z) \in (V)_k }\E_{(\mathbf{w},y)}[\Ind{\mathrm{B}_t=\mathbf{x}, \Fill}  ]-\E_{(\mathbf{w},y)}[\Ind{ \mathrm{B}_t=\mathbf{x}, \Fill}\ink_t(z)  ]
 \\ & \text{(summing over all $z \in \mathbf{x}^{\complement} $ and then over all $\mathbf{x} \in (V)_{k-1} $)} 
\\& = \sum_{(\mathbf{x},z) \in (V)_k } \E_{(\mathbf{w},y)}[\Ind{ \mathrm{B}_t=\mathbf{x}, \Fill}(1-\ink_t(z)  )]= \E_{(\mathbf{w},y)}[\Ind{ \Fill} (n-k+1 - \ink_t) ]  \\ & \text{(using $\Pr[ \Fill ]=(n-k+1)^{-1} $)} =  \hE_{(\mathbf{w},y)}[ 1 - \ink_t / (n-k+1) ].  
\end{split}
\end{equation}
The proof of \eqref{e:IPtoCP} is concluded by combining \eqref{e:trianglein2},\eqref{e:trianglein3} and \eqref{e:TVcal}. 
\end{proof} 
\subsection{Proof of Proposition \ref{p:chambound2}}\label{A:2}
Let $\alpha \in (0,1/4) $ be as in our version of the chameleon process. Let $p:=\alpha/2$ and \[\Delta(r):=\lceil \alpha [r \wedge (n-k+1 - r) ] \rceil.\] Recall that  $\hat \tau_i$  is the time at which the $i$th round ended. Let  \[\hink_i=\ink_{\hat \tau_i}=|\mathrm{R}_{\hat \tau_i}|\]  be the number of red particles at the end of the $i$th round (there are no pink particles at such times) and $\hink_i(y)=\ink_{\hat \tau_i}(y)=\Ind{y \in \mathrm{R}_{\hat \tau_i}} $.     Let  $T_{0}:=\inf\{j: \hink_j = 0 \}$ and $T_{\Fill}:=\inf\{j: \hink_j = n-k+1 \} $.  
 For $i> T_{\Fill} \wedge T_0$ we set $\hink_i:=\hink_{T_{\Fill} \wedge T_0}$. Since each round has success probability exactly $p$, we get that  $\hink_i $ is a Markov chain martingale on $\{0,1,\ldots,n-k+1\} $ with transitions 
$ P(r,r \pm \Delta(r) )=\frac{p }{2}$ and $ P(r,r  )=1-p $, which has $0 $ and $n-k+1$ as absorbing states.

Consider the Doob's transform of $\hink $ conditioned on $\Fill $. This is a  Markov chain on $[n-k+1]$ that has $n-k+1$ as an absorbing state and for $r \in [n-k+1 -1]$ has transitions
$\hat P(r,r \pm \Delta(r) )=\frac{r \pm \Delta(r)}{2r}p$ and $\hat P(r,r  )=1-p $ (cf.\ \cite[p.\ 910]{olive}). Denote this Markov chain by $(Y_i)_{i \in \Z_+}$.
\begin{lemma}
\label{l:expsubmart}[\cite{olive} Proof of Proposition B.1]
Let $I_i:=Y_i/ (n-k+1) $ and $Z_i:=\frac{\sqrt{ I_i \wedge (1-I_i) } }{I_i} $. Then there exists some $c=c_{\alpha}<1$ such that $c^{-i}Z_i$ is a super-martingale. In particular, \[\mathbb{E}[ 1-I_i]  =\mathbb{E}[\sfrac{I_i (1-I_i)}{I_i}] \le \sfrac{1}{2} \mathbb{E}[\sfrac{\sqrt{I_i (1-I_i)}}{I_i}] \le   \mathbb{E}[Z_i] \le c^iZ_0  = c^i \sqrt{ n-k+1}.\] 
\end{lemma} 
\emph{Proof of Proposition \ref{p:chambound2}.} Let $j(t) $ be the number of burn-in periods performed by the chameleon process by time $t$ (recall that the chameleon process always starts with a burn-in period, which in the current setup is of duration $\mix^{(\infty)}(n^{-10})$). Let   $t(j):=t_{\mathrm{mix}}^{(\infty)}(n^{-10})+j t_\mathrm{round} $. Then by Lemma \ref{l:expsubmart}
\begin{align*}
\label{e:ji}
\hE_{(\mathbf{w},y)}[ 1 - \ink_{t(i)} / (n-k+1)     ] &\le  \hP_{(\mathbf{w},y)}[j(t(i)) \ge 2]    +\hE_{{(\mathbf{w},y)}}[ 1 - \hink_i / (n-k+1)     ] \\&\le   \hP_{(\mathbf{w},y)}[j(t(i)) \ge 2] +c^i \sqrt{ n-k+1}. \end{align*}
 Finally, $$  \hP_{(\mathbf{w},y)}[j(t(i)) \ge 2]   =  \widehat{\Pr}_{(\mathbf{w},y)}[\cup_{j=0}^{i-1}  A(j) ] \le (\Pr[\Fill])^{-1}  \Pr_{(\mathbf{w},y)}[\cup_{j=0}^{i-1}  A(j)]  \le (n-k+1) i \beta .$$ \qed
 
\subsection{Proof of Lemmas~\ref{L:piNice} and~\ref{L:piNice2}}\label{proofpiNice}
We apply Proposition~\ref{prop: Lagrange} with $A=\mathrm{Nice}(S)^\complement$ and deduce that if $\Pr_{\pi_S}[X_T\in\mathrm{Nice}(S)^\complement]\ge\pi(\mathrm{Nice}(S)^\complement)+\zeta\pi(\mathrm{Nice}(S)),$ for some $\zeta>0$ then 
\[
\|\Pr_{\pi_S}[X_T\in\bullet]-\pi\|_{2,\pi}^2\ge \sfrac{\zeta^2 \pi(\mathrm{Nice}(S))}{\pi(\mathrm{Nice}(S)^\complement)}.
\]
On the other hand, for proving Lemma~\ref{L:piNice} in which $T\ge C_\mathrm{round}\rel$, we use the Poincar\'e inequality \eqref{e:Poincare} to obtain
\begin{align}\label{eq:gapS}
\|\Pr_{\pi_S}[X_T\in\bullet]-\pi\|_{2,\pi}^2\le e^{-2T\mathrm{gap}}\|\pi_S-\pi \|_{2,\pi}^2 =e^{-2T\mathrm{gap}}\sfrac{\pi(S^\complement)}{\pi(S)}.
\end{align}

Hence combining these two inequalities gives
\[
\zeta^2\le \sfrac{\pi(\mathrm{Nice}(S)^\complement)}{\pi(\mathrm{Nice}(S))}\sfrac{\pi(S^\complement)}{\pi(S)}e^{-2T\mathrm{gap}}.
\]
We bound $\frac{\pi(\mathrm{Nice}(S)^\complement)}{\pi(S)}$ using Lemma~\ref{L:niceC} to obtain
\[
\zeta^2\le \left(\sfrac1{32}+\sfrac{|S|}{n}\right)^{-1} e^{-2T\mathrm{gap}}\sfrac{\pi(S^\complement)}{\pi(\mathrm{Nice}(S))}.
\]  Hence for each $\varepsilon>0$, there exists a $C_0$  
such that for all $C_\mathrm{round}>C_0$, uniformly over the choice of $S$, we have $\zeta\pi(\mathrm{Nice}(S))\le \varepsilon$, and hence
\[
\Pr_{\pi_S}[X_T\in\mathrm{Nice}(S)^\complement]\le\pi(\mathrm{Nice}(S)^\complement)+\varepsilon,
\]
which completes the proof for these cases. 

For proving Lemma~\ref{L:piNice2} under assumption $T\ge C_\mathrm{round}/\Lambda(C_\mathrm{profile}|S|/n)$ we instead 
use \eqref{e:spb3} which gives for any $\eps\in(0,1)$, we have 
\[
\|\Pr_{\pi_S}[X_T\in\bullet]-\pi\|_{2,\pi}^2\le \eps\sfrac{\pi(S^\complement)}{\pi(S)},
\]
provided $C_\mathrm{round}\ge\log(1/\eps)$ and $C_\mathrm{profile}\ge8/\eps$ (and we have used that $|S^\complement|\ge n/2$). This bound replaces \eqref{eq:gapS} in the above argument to complete the proof for this case.
\subsection{Proof of Lemma~\ref{L:largedev}}\label{proofoflargedev}
\begin{lemma}\label{L:maxPa}For each $\epsilon\in(0,1)$, we have
\[\max_{a,u,x,v} Q(a,u,x,v,\epsilon)\le \max_{z,z'}P_{t_\ast(\epsilon)}(z,z') \le \sfrac{\epsilon}{\log n}.\]
\end{lemma}
\emph{Proof.}
The second inequality is immediate by the definition of $t_\ast(\epsilon)$.  By  averaging over $(I_{[0,s]}(x):s \in [0,T] )$ the trajectory performed by the particle from $x$,  it is easy to see that for all $b,c,u,v \in V $ we have that 
\begin{align*}&\Pr\left[I_{[0,t_\ast(\epsilon)]}(a)=c,\,N_{t_\ast(\epsilon)}(a,x)=0,I_{[0,t_\ast(\epsilon)]}(x)=b\bigm| I_{[0,T]}(x)=v\right] \\& \le \Pr\left[I_{[0,t_\ast(\epsilon)]}(a)=c \right] q(b), \end{align*} where  we define $q(b):=\Pr\left[I_{[0,t_\ast(\epsilon)]}(x)=b\bigm| I_{[0,T]}(x)=v \right] $ and  $p(c,u \, | \, b,v ):=\Pr[I_{[t_\ast(\epsilon),T]}(c)=u \bigm| I_{[0,t_\ast(\epsilon)]}(x)=b, I_{[t_\ast(\epsilon),T]}(b)=v ]$. Then 
\begin{align*}
Q(a) &\le \sum_{b,c}\Pr\left[I_{[0,t_\ast(\epsilon)]}(a)=c \right]q(b)p(c,u \, | \, b,v ) \\&\le \max_{z,z'}P_{t_\ast(\epsilon)}(z,z')  \sum_{b}q(b) \sum_c p(c,u \, | \, b,v )=\max_{z,z'}P_{t_\ast(\epsilon)}(z,z') . \quad \text{\qed}  \end{align*}
\begin{proof}[Proof of Lemma~\ref{L:largedev}]
Let $2\le k\le n/2$, $B\in(V)_{k-1}$,  $s\ge t_{\mathrm{mix}}^{(\infty)}(n^{-10})$ and  $\lambda \in (0,\sfrac{1}{\max_b Q(b)} ] $.
For each $a\in V$, we have 
\begin{align*}
\mathbb{E}_B\left[e^{\lambda\indic{a\in \BB_s}Q(a)}\right]&=1+\Pr_B[a\in \BB_s]\big(e^{\lambda Q(a)}-1\big)\\&\le 1+\Pr_B[a\in \BB_s]\lambda Q(a)[1+\lambda Q(a)],
\end{align*}
using $e^x\le 1+x+x^2$, for $x \in [0,1] $. Hence by Lemma~\ref{L:maxPa} and the choice of $\epsilon$ we obtain
\begin{align*}
\mathbb{E}_B\left[e^{\lambda\indic{a\in \BB_s}Q(a)}\right]&\le 1+\Pr_B[a\in \BB_s]\lambda Q(a)(1+\sfrac{\lambda}{ 10^{4}\log n})\\&\le \exp\left\{\Pr_B[a\in \BB_s]\lambda Q(a)(1+\sfrac{\lambda}{ 10^{4}\log n})\right\},
\end{align*}
using $1+x\le e^x$.
Now as $s\ge t_{\mathrm{mix}}^{(\infty)}(n^{-10})$ and using the NA property,
\begin{align*}
&\mathbb{E}_B\left[\exp\left\{\lambda\sum_a \indic{a\in \BB_s}Q(a)\right\}\right]\\&\le \prod_a\mathbb{E}_B\left[e^{\lambda\indic{a\in \BB_s}Q(a)}\right]\\&\le \prod_a\exp\left\{\Pr_B[a\in \BB_s]\lambda Q(a)(1+\sfrac{\lambda}{ 10^{4}\log n})\right\}\\
&=\exp\left\{\sum_a\Pr_B[a\in \BB_s]\lambda Q(a)(1+\sfrac{\lambda}{ 10^{4}\log n})\right\}\\
&\le \exp\left\{\lambda\sfrac{k}{n}(1+n^{-10})\sum_a Q(a)(1+\sfrac{\lambda}{ 10^{4}\log n})\right\}\\
&\le \exp\left\{\lambda\sfrac{k}{n}(1+n^{-10})(1+\sfrac{\lambda}{ 10^{4}\log n})\right\}.
\end{align*}
Hence using a Chernoff bound we obtain that for any $\lambda \in (0,10^{4}\log n] $ and $c>0$,
\[
\Pr_B\left[\sum_a\indic{a\in \BB_s}Q(a)>c \right]\le e^{-c\lambda }\exp\left\{\lambda\sfrac{k}{n}(1+n^{-10})(1+\lambda 10^{-4}/\log n)\right\}.
\]
Thus if we take $\lambda=300 \log n$ then, for any $c\in[\frac{k}{n}+\frac1{16},1)$, we obtain the desired result provided $n$ is sufficiently large.\end{proof}
\subsection{Proof of Lemma~\ref{L:black}}\label{proofblack}

 We bound the probability of interest using a Chernoff bound and negative association (NA). For any $v\in V, B\in (V)_{k-1}$, $s\ge t_{\mathrm{mix}}^{(\infty)}(n^{-10})$, and $\theta,\lambda>0$,
\begin{align*}
\Pr_B\Big[\sum_{u:\,v  \stackrel{\rightarrow}{\sim} u}\indic{u\in \BB_s}>\theta \hat d\Big]&\le\exp(-\lambda\theta \hat d)\mathbb{E}_B\Big(\exp\big(\lambda \sum_{u:\,v  \stackrel{\rightarrow}{\sim} u}\indic{u\in \BB_s}\big)\Big)\\
&\le \exp(-\lambda\theta \hat d)\prod_{u:\,v  \stackrel{\rightarrow}{\sim} u}\mathbb{E}_B(\exp(\lambda\indic{u\in\BB_s})\mid \BB_0=B).
\end{align*}
Since $\indic{u\in \BB_s}$ are Bernoulli random variables, which take value 1 with probability $P_s(u,B)$, conditionally on $\BB_0=B$, the above bound becomes
\begin{align*}
\Pr_B\Big[\sum_{u:\,v  \stackrel{\rightarrow}{\sim} u}\indic{u\in \BB_s}>\theta \hat d\Big]&\le\exp(-\lambda\theta \hat d)\prod_{u:\,v  \stackrel{\rightarrow}{\sim} u}\left(P_s(u,B)e^\lambda+1-P_s(u,B)\right)\\
&\le \exp(-\lambda\theta \hat d)\prod_{u:\,v  \stackrel{\rightarrow}{\sim} u}\left(1+\sfrac{k}{n}(e^\lambda-1)(1+n^{-10})\right)\\
&\le \exp\left\{- \hat d\left(\lambda\theta-\sfrac{k}{n}(e^\lambda-1)(1+n^{-10})\right)\right\}.
\end{align*}
where we have used $1+x\le e^x$ in the last line. With $\theta=\sfrac{k}{n}+\varepsilon$, the optimal choice of $\lambda$ is $\log\left(\frac{1+\varepsilon n/k}{1+n^{-10}}\right)$. With these values the bound becomes
\[
\Pr_B\Big[\sum_{u:\,v  \stackrel{\rightarrow}{\sim} u}\indic{u\in \BB_s}>(k/n+\eps) \hat d\Big]\le\exp\Big(- \hat d\eps\max\Big\{\log\sfrac{\eps n}{e^2k},\sfrac{\eps n}{2k}\big(\sfrac12-\sfrac{\eps n}{k}\big)\Big\}\Big),
\]
for any $\eps\in(0,1)$, and $n$ sufficiently large (depending on $\eps$) as required. 

\subsection{Proof of Lemma~\ref{L:boundA}}\label{A:prelim}
Suppose $\mathrm{Nice}(S)=\{u_1,\ldots,u_m\}$. For each $1\le i\le m$, let $p_i=\Pr_{\pi_S}[X_T=u_i]$. 
We order the values of $p_i$, defining 
\[
p_{(1)}\le p_{(2)}\le\cdots\le p_{(m)},
\]
and similarly define $u_{(i)}$ via
\[
p_{(i)}=\Pr_{\pi_S}[X_T=u_{(i)}].
\]
We can construct a set $A$ with the desired property in an iterative manner. Firstly set $A=\{u_{(m)}\}$ and define a set $B=\mathrm{Nice}(S)\setminus\{u_{(m)}\}$. Then remove from $B$ all $u_i$ within distance $2\times 10^4$ from $u_{(m)}$. This removes at most $d^{2\times 10^4}<10^{10^5}$ vertices. Then we set $A=A\cup \{u_\ast\}$ where \[
u_\ast=\argmax_{u\in B}\Pr_{\pi_S}(X_T=u).
\] From here we iterate the procedure until $B$ is the empty set. It is clear that with this algorithm in the worst case scenario the set $A$ is \[A=\{u_{(m)},u_{(m-r)},u_{(m-2r)},\ldots\}\] where $r=10^{10^5}$. The result is proved by noting that for $i\in\{0,1,\ldots\}$,
\[
p_{(m-ir)}+\cdots+p_{(m-(i+1)r+1)}\le rp_{(m-ir)},
\]
and so for this choice of $A$
\[
\sum_{u\in\mathrm{Nice}(S)}p_u\le r\sum_{u\in A}p_u.
\]
The proof is thus complete for any $c_\mathrm{frac}\le 10^{-10^5}$.

\section*{Acknowledgements}
The authors are grateful to Nathana\"el Berestycki, Gady Kozma, Ben Morris  and Roberto Oliveira for useful discussions. In particular, we wish to thank  Gady Kozma for pointing out that the mixing time in the case of the hypercube was previously considered in \cite{Wilson}.  

\nocite{}
\bibliographystyle{plain}
\bibliography{Exclusion}

\vspace{2mm}


\end{document}